\documentclass[english,reqno,10pt]{article}
\usepackage{geometry}
\usepackage[T1]{fontenc}
\usepackage[latin9]{inputenc}
\usepackage{babel}
\usepackage{float}
\usepackage{amsmath}
\usepackage{stmaryrd}
\usepackage{amsthm}
\usepackage{amssymb}
\usepackage{cancel}
\usepackage{graphicx,xcolor}
\usepackage{wrapfig}
\usepackage{esint}
\usepackage[toc,page]{appendix}
\usepackage{subcaption}
\usepackage{chngcntr}
\usepackage{apptools}
\usepackage{algorithm,algpseudocode}
\newcommand{\R}{{\mathbb R}}
\newcommand{\E}{{\mathbb E}}\makeatletter
\newcommand{\N}{{\mathbb N}}
\DeclareMathOperator{\sgn}{sgn}
\floatstyle{ruled}
\newfloat{algorithm}{tbp}{loa}
\providecommand{\algorithmname}{Algorithm}
\floatname{algorithm}{\protect\algorithmname}
\theoremstyle{definition}
\newtheorem*{condition*}{\protect\conditionname}
\theoremstyle{plain}
\newtheorem{thm}{\protect\theoremname}[section]
\theoremstyle{plain}
\newtheorem{prop}[thm]{\protect\propositionname}
\theoremstyle{definition}

\theoremstyle{plain}
\newtheorem{lem}[thm]{\protect\lemmaname}
\theoremstyle{plain}

\theoremstyle{plain}
\newtheorem{remark}[thm]{\protect\remarkname}
\theoremstyle{plain}

\theoremstyle{plain}
\newtheorem*{assumption*}{\protect\assumptionname}
\numberwithin{equation}{section}
\makeatother
\providecommand{\assumptionname}{Assumption}
\providecommand{\conditionname}{Condition}
\providecommand{\corollaryname}{Corollary}
\providecommand{\definitionname}{Definition}
\providecommand{\lemmaname}{Lemma}
\providecommand{\propositionname}{Proposition}
\providecommand{\theoremname}{Theorem}
\providecommand{\remarkname}{Remark}
\providecommand{\examplename}{Example}

\makeatletter
\newcommand{\opnorm}{\@ifstar\@opnorms\@opnorm}
\newcommand{\@opnorms}[1]{%
  \left|\mkern-1.5mu\left|\mkern-1.5mu\left|
   #1
  \right|\mkern-1.5mu\right|\mkern-1.5mu\right|
}
\newcommand{\@opnorm}[2][]{%
  \mathopen{#1|\mkern-1.5mu#1|\mkern-1.5mu#1|}
  #2
  \mathclose{#1|\mkern-1.5mu#1|\mkern-1.5mu#1|}
}
\makeatother

\begin{document}

\title{Convergence in total variation of the Euler-Maruyama scheme applied to diffusion processes with  measurable drift coefficient and additive noise}

\author{O. Bencheikh and B. Jourdain\thanks{Cermics, \'Ecole des Ponts, INRIA, Marne-la-Vall\'ee, France. E-mails : benjamin.jourdain@enpc.fr, oumaima.bencheikh@enpc.fr. The authors would like to acknowledge financial support from Universit\'e Mohammed VI Polytechnique.}}

\maketitle

\begin{abstract}
  We are interested in the Euler-Maruyama discretization of a stochastic differential equation in dimension $d$ with constant diffusion coefficient and bounded measurable drift coefficient. In the scheme, a randomization of the time variable is used to get rid of any regularity assumption of the drift in this variable. We prove weak convergence with order $1/2$ in total variation distance. When the drift has a spatial divergence in the sense of distributions with $\rho$-th power integrable with respect to the Lebesgue measure in space uniformly in time for some $\rho \ge d$, the order of convergence at the terminal time improves to $1$ up to some logarithmic factor. In dimension $d=1$, this result is preserved when the spatial derivative of the drift is a measure in space with total mass bounded uniformly in time. We confirm our theoretical analysis by numerical experiments.

\end{abstract} 



\section{Introduction}

In numerous areas such as mathematical finance when, for example, modelling a stock price process whose trend dramatically changes when a factor goes down a threshold value, or in stochastic control theory when choosing a control process that minimizes the expected discounted cost, we end-up dealing with diffusions that do not show a smooth behavior which results in Stochastic Differential Equations with discontinuous coefficients.
In the present paper, we are interested in the Euler-Maruyama discretization of the stochastic differential equation
\begin{equation}
  X_t = X_0 + W_t + \int_0^t b(s,X_s)\,ds,\quad t\in[0,T]\label{edsnl}
\end{equation}
where $(W_t)_{t\ge 0}$ is a $d$-dimensional Brownian motion independent from the initial $\R^d$-valued random vector $X_0$, $T\in(0,+\infty)$ is a finite time horizon and the drift coefficient $b:[0,T]\times\R^d\to\R^d$ is merely measurable and bounded.\\

While the convergence properties of the Euler-Maruyama scheme are well understood for SDEs with smooth coefficients, the case of irregular coefficients is still an active field of research.
Concerning the strong error, the additive noise case is investigated in \cite{HalKloe} where Halidias and Kloeden only prove convergence and in \cite{DarGer,NeuSzo} where rates are derived. Dareiotis and Gerencs\'er \cite{DarGer} obtain convergence with $L^2$-order $1/2-$ (meaning $1/2-\epsilon$ for arbitrarily small $\epsilon >0$) in the time-step for bounded and Dini-continuous time-homogeneous drift coefficients and check that this order is preserved in dimension $d=1$ when the Dini-continuity assumption is relaxed to mere measurability. In the scalar $d=1$ case, Neuenkirch and Sz\"olgyenyi \cite{NeuSzo} assume that the drift coefficient is the sum of a $ \mathcal{C}^{2}_b$ part and a bounded integrable irregular part with a finite Sobolev-Slobodeckij semi-norm of index $\kappa \in (0,1)$. They prove $L^2$-convergence with order $\frac{3}{4}\wedge\frac{1 + \kappa}{2} -$ for the equidistant Euler-Maruyama scheme, the cutoff of this order at $\frac{3}{4}$ disappearing for a suitable non-equidistant time-grid.
Note that an exact simulation algorithm has been proposed by \'Etor\'e and Martinez \cite{EtoMArt} for one-dimensional SDEs with additive noise and time-homogeneous and smooth except at one discontinuity point drift coefficient.\\More papers have been devoted to the strong error of the Euler scheme for SDEs with a non constant diffusion coefficient. The first result goes back to Gy\"ongy and Krylov \cite{GyoKryl} who established convergence in probability (without any rate) when the coefficients are continuous in space and pathwise uniqueness holds for the stochastic differential equation.
Gy\"ongy \cite{Gyo} proves almost sure convergence with order $1/4-$ when the diffusion coefficient is locally Lipschitz in space and the drift coefficient locally one-sided Lipschitz in space uniformly in time and some Lyapunov condition holds. Yan \cite{Yan} investigates conditions under which the Euler scheme converges to the unique weak solution of the SDE. In dimension $d=1$, this author also shows the $L^1$-order $\beta_1\wedge\frac{\alpha}{2}\wedge\frac{\alpha}{1+\alpha}\beta_2$ when the drift coefficient is Lipschitz in the spatial variables and $\beta_1$-H\"older in time while the diffusion coefficient is $(\frac{1}{2}+\alpha)$-H\"older in space and $\frac{\beta_2}{2}$-H\"older in time. Still in dimension one, Gyongy and Rasonyi \cite{GR} obtain the $L^1$-order $\alpha\wedge\frac{\gamma}{2}$ when the diffusion coefficient is $(\frac{1}{2}+\alpha)$-H\"older continuous in space and the drift coefficient is the sum of a function Lipschitz continuous in space and a function non-increasing and $\gamma$-H\"older continuous in space. Like in \cite{GyoKryl,Gyo}, the discretization only concerns the spatial variable of the coefficients while the time variable still moves continuously in the scheme analysed. In \cite{NgTag1}, Ngo and Taguchi prove $L^1$-order $1/2$ when (resp. $\alpha\in (0,1/2]$, when $d=1$,) the diffusion coefficient is uniformly elliptic, bounded, Lipschitz (resp. $(\frac{1}{2}+\alpha)$-H\"older) in space and the drift coefficient one-sided Lipschitz, bounded and with bounded variation in space with respect to some Gaussian measure. The coefficients are assumed to be $1/2$-H\"older with respect to the time variable. In a second paper \cite{NgTag2} specialized to dimension $d=1$ with time-homogeneous coefficients, they show $L^1$-order $\frac{\beta}{2}\wedge \alpha$ under the same assumption on the diffusion coefficient and when the drift is the sum of a bounded $\beta$-H\"older function and a bounded function with bounded variation with respect to some Gaussian measure. When the diffusion coefficient is uniformly elliptic, Lipschitz continuous in space, Dini-continuous in time and the drift coefficient is Dini-continuous in both variables, Bao, Huang and Yuan \cite{BHY} prove $L^2$-convergence with an order expressed in terms of the Dini modulus of continuity.
\\Recent attention has been paid to the Euler-Maruyama discretization of SDEs with a piecewise Lipschitz drift coefficient and a globally Lipschitz diffusion coefficient which satisfies some non-degeneracy condition on the discontinuity hypersurface of the drift coefficient. Leobacher and Sz\"olgyenyi \cite {LeoSzo3} prove convergence with $L^2$-order $1/4-$ of the Euler-Maruyama method. This result is proved by comparison with a scheme with $L^2$-order $1/2$ \cite {LeoSzo1,LeoSzo2} obtained by the Euler discretization of a transformation of the original SDE which permits to remove the discontinuity of the drift. In dimension $d=1$,  M\"uller-Gronbach and Yaroslavtseva \cite{MGYaro1} recover for each $p\in[1,\infty)$ the  $L^p$-order of convergence $1/2$ valid when the drift coefficient is globally Lipschitz. In higher dimension, this order $1/2$ (up to some logarithmic factor) is proved by Neuenkirch, Sz\"olgyenyi and Szpruch \cite{NeunSzoSpr} for the $L^2$-error of an Euler-Maruyama scheme with adaptive time-stepping.\\

Concerning the weak error, Mikulevicius and Platen \cite{MikP} prove that, under uniform ellipticity, when the coefficients are $\alpha$-H\"older with respect to the spatial variables and  $\frac{\alpha}{2}$-H\"older with respect to the time variable for $\alpha\in(0,1)$, then $\E[f(X_T)]$ is approximated with order $\frac{\alpha}{2}$ by replacing $X_T$ by the Euler-Maruyama scheme at time $T$ when the test function $f:\R^d\to \R^d$ is twice continuously differentiable with $\alpha$-H\"older second order derivatives. In the additive noise case, Kohatsu-Higa, Lejay and Yasuda \cite{KohY}, prove that for $f$ thrice continuously differentiable with polynomially growing derivatives, the convergence holds with order $1/2-$ when $d\ge 2$ (resp. $1/3-$ when $d=1$) and the drift coefficient is time homogeneous, bounded and Lipschitz except on a set $G$ such that $\varepsilon^{-d}$ times the Lebesgue measure of $\{x\in\R^d:\inf_{y\in G}|x-y|\le\varepsilon\}$ is bounded. Their approach, which consists in regularizing the drift coefficient and considering both the stochastic differential equation and the Euler scheme for the regularized coefficient, is also applied in \cite{Koh} to the case of time-dependent, bounded, uniformly elliptic and continuous diffusion coefficients. In \cite{MenoKona}, Konakov and Menozzi regularize both coefficients to obtain that the absolute difference between the densities with respect to the Lebesgue measure of the solution and its Euler-Maruyama discretization is bounded from above by a Gaussian density multiplied by a factor with order $\frac{\alpha}{2}-$, when these coefficients are uniformly elliptic, bounded and $\frac{\alpha}{2}$-H\"older continuous in time and $\alpha$-H\"older continuous in space. The order of the factor is $\frac{1}{2d}-$ when the coefficients are bounded, uniformly elliptic, continuously differentiable with bounded derivatives up to the order $2$ in time and the order $4$ in space at the possible exception, for the drift coefficient, of a finite union of time-independent smooth submanifolds where it can be discontinuous. In \cite{NFrik}, Frikha deals with time-homogeneous one-dimensional stochastic differential equations possibly involving a local time term in addition to a bounded measurable drift coefficient and a bounded uniformly elliptic and $\alpha$-H\"older diffusion coefficient. He proves that the absolute difference between the densities is smaller than a Gaussian density multiplied by a factor with order $\frac{\alpha}{2}$ in the time-step. The latest work in this field is by Suo, Yuan and Zhang \cite{SuYuZ} who study in a multidimensional setting stochastic differential equations with additive noise and time-homogeneous drift coefficients with at most linear growth and satisfying an integrated against some Gaussian measure $\alpha$-H\"older type regularity condition. When this coefficient has sublinear growth (and under some restriction on the time-horizon when it has linear growth), they prove convergence in total variation with order $\frac{\alpha}{2}$.\\

In the current paper, we consider the stochastic differential equation \eqref{edsnl} with additive noise and bounded and measurable drift function $b:[0,T]\times\R^d\to\R^d$. We are interested in estimating the spatial integral of the absolute difference between the densities with respect to the Lebesgue measure of the solution and its Euler-Maruyama discretization with time-step $h\in(0,T]$. This integral is equal to the total variation distance between the probability measures that admit these densities with respect to the Lebesgue measure. Note that the approximation of a Markovian semi-group in total variation distance has been investigated by Bally and Rey \cite{Bally} who apply their results to the Ninomiya discretization scheme.

To get rid with any assumption stronger than mere measurability concerning the regularity of the drift coefficient with respect to the time variable, we consider the Euler-Maruyama discretization with randomized time variable of \eqref{edsnl}. It evolves inductively on the regular time-grid $(kh)_{k \in \left\llbracket0,\left\lfloor \frac{T}{h} \right\rfloor \right\rrbracket}$ by:
  \begin{equation}
    X^{h}_{(k+1)h}=X^h_{kh}+\left(W_{(k+1)h}-W_{kh}\right)+b\left(\delta_{k}, X^h_{kh}\right)h, \label{euler}
  \end{equation}
where the random variables $(\delta_k)_{k \in \left\llbracket0,\left\lfloor \frac{T}{h} \right\rfloor \right\rrbracket}$ are independent, respectively distributed according to the uniform law on $[kh,(k+1)h]$ and are independent from $\left(X_0,(W_t)_{t\ge 0}\right)$. Notice that this sequence is of course not needed to randomize the time variable when $b$ is time-homogeneous. To our knowledge, such randomization techniques have been proposed so far to improve the strong convergence properties of discretizations of ordinary differential equations \cite{Stengle1,Stengle2,HeinMil,JenNeue,Daun} or stochastic differential equations \cite{CruW}. They also happen to be quite efficient in terms of weak error. Indeed, the above randomization turns out to enable convergence in total variation distance with order $1$ up to some logarithmic factor. For $s\in[0,T)$, we denote by $\ell_s = \lfloor s/h\rfloor$ the index of the corresponding time interval $s \in [kh,(k+1)h)$. s.t. $k \le \lfloor T/h\rfloor -1$. We consider then the following continuous time interpolation of the scheme:
  \begin{equation}
    X^h_t = X_0 + W_t + \int_0^t b\left(\delta_{\ell_s},X^h_{\tau^h_s}\right)\,ds,\quad t\in[0,T] \mbox{ where } \tau^h_s = \lfloor s/h\rfloor h \label{conteuler}.
  \end{equation}
 For $t\ge 0$, we denote by $\mu_t$ the law of $X_t$ and by $\mu^h_t$ the law of $X^h_t$. We have $\mu_0=\mu^h_0 = m$. For $t>0$, according to Proposition \ref{propspde} below, $\mu_t$ and $\mu^h_t$ admit densities $p(t,.)$ and $p^h(t,.)$ with respect to the Lebesgue measure. 
Therefore, our approach amounts to study the rate of convergence of the $L^1$-norm of the difference between $p(t,.)$ and $p^h(t,.)$. We assume, in what is next, that $X_0$ is distributed according to a probability measure $m$ on $\R^d$ and the drift $b=(b_{i})_{1 \le i \le d}:[0,T]\times\R^d \to \R^d$ is a measurable function bounded by $B<+\infty$ when $\R^d$ is endowed with the $L^\infty$-norm.\\

The paper is organized as follows. In Section \ref{MainRes}, we state our main results. We first obtain the convergence of the weak error in total variation in $\mathcal{O}(\sqrt h)$ when $b$ is measurable and bounded. When assuming more regularity on $b$ with respect to the space variables, namely that the divergence in the sense of distributions of $b$ with respect to these variables is in $L^{\rho}\left(\R^d\right)$ for some $\rho \ge d$ uniformly with respect to the time variable, the weak rate of convergence $\left\|\mu_{kh} - \mu^h_{kh}\right\|_{\rm TV}$ is $\frac{1}{\sqrt{kh}}\left(1 + \ln\left(k\right)\right)h$ and it improves, at the terminal time, to $1$ up to a logarithmic factor. Furthermore, when assuming more regularity on the probability measure $m$ in addition to the spatial regularity on $b$, we improve the previous weak rate of convergence by eliminating the prefactor $\frac{1}{\sqrt{kh}}$. We obtain these results by comparing the mild equation satisfied by $p(t,.)$ and the perturbed mild equation satisfied by $p^h(t,.)$. We investigate through the Lamperti transform the application of those theorems to one-dimensional SDEs with a non-constant diffusion coefficient. Sections \ref{preuve} and \ref{preuveB} are dedicated to the proofs of the main results in Section \ref{MainRes}. We finally provide numerical experiments in Section \ref{NumExp} to illustrate our results. \\

Beforehand, note that our results also apply to the more general case of SDEs with constant and non degenerate diffusion coefficient $\sigma \in \R^d \times \R^d$:
\begin{align*}
  \displaystyle Y_t = Y_0 + \sigma W_t + \int_0^t \tilde b \left(s,Y_s\right)\,ds, \quad t\in[0,T]
\end{align*}
where the $\R^d$-valued random variable $Y_0$ is independent from $\left(W_t\right)_{t \ge 0}$ and $\tilde b:[0,T]\times \R^d \to \R^d$ is measurable and bounded.  Indeed, our results remain true for this type of diffusions since the transformation $\left(X_t = \sigma^{-1}Y_t\right)_{t \in [0,T]}$ is solution to the dynamics \eqref{edsnl} initialized by $X_0 = \sigma^{-1}Y_0$ for the choice of $b(t,x) = \sigma^{-1} \tilde b(t,\sigma x)$. The associated Euler scheme evolving inductively on the time grid $\left(kh\right)_{k \in \left \llbracket 0, \left\lfloor \frac{T}{h} \right\rfloor \right \rrbracket}$ is defined by:
\begin{align*}
  \displaystyle X^h_{(k+1)h} = X^h_{kh} + \left(W_{(k+1)h}-W_{kh}\right) + \sigma^{-1} \tilde b\left(\delta_k, \sigma X^h_{kh} \right)h,
\end{align*}
and we can clearly see that $\left(\sigma X^h_{kh}\right)_{k \in \left \llbracket 0, \left\lfloor \frac{T}{h} \right\rfloor \right \rrbracket}$ coincides exactly with the Euler scheme $\left(Y^h_{kh}\right)_{k \in \left \llbracket 0, \left\lfloor \frac{T}{h} \right\rfloor \right \rrbracket}$ of the process $\left(Y_t\right)_{t \in [0,T]}$. Denoting by $\tilde \mu_t$ the law of $Y_t$ and by $\tilde \mu^h_t$ the law of $Y^h_t$ the continuous time interpolation of the Euler scheme, we have that $\left\|\tilde \mu_t - \tilde \mu^h_t \right\|_{\rm TV} = \left\| \mu_t - \mu^h_t \right\|_{\rm TV}$. Moreover, Lemma \ref{lema1} in the appendix, which relates the spatial divergences in the sense of distributions of $y\mapsto \tilde b(t,y)$ and $x\mapsto \sigma^{-1}\tilde b(t,\sigma x)$, ensures that when the drift coefficient $\tilde b(t,y)$ satisfies the strengthened hypotheses in Theorem \ref{cvgenceB} below, then so does $\sigma^{-1}\tilde b(t,\sigma x)$.\\ 

Before going any further, we introduce some additional notation.
\subsection*{Notation:}
\begin{itemize}
    \item For $x \in \R^d$, we denote by $|x| = \left(\sum \limits_{i=1}^d x_{i}^2 \right)^{1/2}$ the euclidean norm of $x$. 
    \item For $ 1 \le p < \infty$, we denote by $L^p\left(\R^d\right)$ the space of measurable functions on $\R^d$ which are $L^p$-integrable for the Lebesgue measure i.e. $f\in L^p$ if $\displaystyle \left\Vert f\right\Vert_{L^p} = \left( \int_{\R^d} |f(x)|^p \; dx \right)^{\frac{1}{p}}<+\infty$.
    \item The space $L^{\infty}\left(\R^d\right)$ refers to the space of almost everywhere bounded measurable functions on $\R^d$ endowed with the norm $\left\Vert f\right\Vert_{L^{\infty}} = \inf \left\{ C \ge 0: |f(x)|\le C \; dx \text{ a.e. on } \R^d \right\} $. 
    \item For notational simplicity, when a function $g$ is defined on $[0,T] \times \R^d$ and $x \in \R^d$, we may use sometimes the notation $g_0(x) := g(0,x)$.
    \item We denote by $W^{1,1}\left(\R^d\right)$ the Sobolev space over $\R^d$ defined as $W^{1,1}\left(\R^d\right) \equiv \left\{u \in L^1\left(\R^d\right): \nabla u \in L^1\left(\R^d\right)^d \right\}$ where $\nabla u$ refers to the spatial derivative of $u$ in the sense of distributions. The space is endowed with the norm $\left\| u\right\|_{W^{1,1}} = \left\| u\right\|_{L^1} + \sum\limits_{i=1}^{d}\left\|\partial_{x_i}u \right\|_{L^1}$. 
    \item For any open subset $A \subset \R^d$, we denote by $\mathcal{C}^k_c(A)$ the space of real functions continuously differentiable in $A$ up to the order $k \in \N$, with compact support on $A$. 
    \item We denote by $BV(\R)$ the space of functions with bounded variation on $\R$. For a function $f \in L^1_{\rm loc}\left(\R\right)$, if $f \in BV(\R)$ then the derivative of $f$ in the sense of distributions is a finite measure in $\R$. 
    \item Let $\left(\mu_1, \dots, \mu_{d^*}\right)$ be signed bounded measures on $\R^d$ and $f:\R^d \to \R^{d^*}$ a $\mathcal{C}^0$-integrable function, with $d^* \in \{ 1,d\}$. We define the convolution product of $f$ and $\mu = \left(\mu_1, \dots, \mu_{d^*}\right)$ by: $$\displaystyle \bigg(f * \mu\bigg)(x) = \sum_{i=1}^{d^*}\int_{\R^d} f_{i}(x-y)d\mu_{i}(y) \; \text{ for } x \in \R^d.$$ When each $\mu_{i}$ admits a density $g_{i}$ with respect to the Lebesgue measure, we also denote by $\Big(f*g\Big)$ this convolution product.
\end{itemize}


\section{Main results}\label{MainRes}

In this section, we give the main results concerning the convergence of $\mu^h_t$, the law of the Euler discretization with time-step $h$ towards its limit $\mu_t$. We will make an intensive use of the interpretation of the total variation norm of their difference as the $L^1$-norm of the difference between their respective densities $p^h(t,.)$ and $p(t,.)$.\\

We recall that the total variation norm for a signed measure $\mu$ on $\R^d$ is defined as:
\begin{align}\label{totVar}
  \displaystyle \left\|\mu \right\|_{\rm TV} = \sup_{\varphi \in \mathcal{L}} \int_{\R^d}\varphi(x) d\mu(x)
\end{align}
where $\mathcal{L}$ denotes the set of all measurable functions $\varphi:\R^d \to [-1,1]$.
Moreover, when $\mu$ admits a density $f_{\mu}$ with respect to a reference non-negative measure $\lambda$, we have the following equality:
\begin{align}\label{caract}
  \displaystyle \left\|\mu\right\|_{\rm TV} = \int_{\R^d} \left|f_{\mu}(x)\right|\lambda(dx) .
\end{align}
Let us now state our estimation of the weak convergence rate of the Euler scheme towards its limit.
\begin{thm}\label{cvgence}
  Assume $b:[0,T]\times\R^d \to \R^d$ is measurable and bounded by $B<+\infty$. Then:
  \begin{align*}
    \exists \, C<+\infty,\;\forall h \in (0,T],\; \forall k \in \left\llbracket 0, \left\lfloor \frac{T}{h} \right\rfloor \right\rrbracket, \quad \left\|\mu_{kh}-\mu^h_{kh} \right\|_{\rm TV}\le C \sqrt h.
  \end{align*}
\end{thm}
\begin{remark}
\begin{itemize}
  \item In dimension $d=1$, when specialized to the constant diffusion coefficient and absence of local time term case, Theorem 2.2 in \cite{NFrik} gives a finer estimation of the absolute difference between the densities by $C\sqrt{h}$ times some Gaussian density. Our result is recovered by spatial integration.
  It implies that for any bounded and measurable test function $f:\R^d\to\R$, $\E\left[f\left(X^h_{kh}\right)\right]-\E\left[f\left(X_{kh}\right)\right]={\mathcal O}(\sqrt{h})$.
  \item Up to some factor $h^{-\varepsilon}$ with $\varepsilon$ arbitrarily small, this behaviour was proved in dimension $d\ge 2$ for thrice continuously differentiable test functions (with polynomial growth together with their derivatives) $f$ by Kohatsu-Higa, Lejay and Yasuda \cite{KohY}, when the drift coefficient is time-homogeneous and Lipschitz outside some sufficiently small set.
  \item Suo, Yuan and Zhang \cite{SuYuZ} proved convergence with order $\frac{\alpha}{2}$ in total variation when the drift coefficient is time-homogeneous and satisfies some integrated against a Gaussian measure $\alpha$-H\"older type of regularity condition. Since $\alpha$ appears to take values smaller than one, we obtain the better order of convergence $1/2$ without any regularity assumption. On the other hand, we need boundedness of the drift coefficient whereas Suo, Yan and Zhang get rid of this assumption and only assume sublinear growth and even linear growth but with some restriction on the time-horizon $T$. 
\end{itemize}
\end{remark}

Now, when assuming more regularity on $b$ with respect to the space variables, we obtain a better rate of convergence:
\begin{thm}\label{cvgenceB}
  Assume $b:[0,T]\times\R^d \to \R^d$ is measurable and bounded by $B<+\infty$. If $\,\sup_{t \in [0,T]} \|\nabla \cdot b(t,.)\|_{L^\rho}<+\infty$ for some $\rho \in [d,+\infty]$ or for $d=1$, $\sup_{t \in [0,T]}\left\|\partial_x b(t,.)\right\|_{\rm TV}<+\infty$; where $\nabla \cdot b(t,.)$ and $\partial_x b(t,.)$ are respectively the spatial divergence and the spatial derivative of $b$ in the sense of distributions. Then:
  \begin{align*}
    \displaystyle \exists \, \tilde{C}<+\infty, \;\forall h \in (0,T],\; \forall k \in \left\llbracket 0, \left\lfloor \frac{T}{h} \right\rfloor \right\rrbracket, \quad \left\|\mu_{kh}-\mu^h_{kh} \right\|_{\rm TV} \le \frac{\tilde{C}}{\sqrt{kh}}\Big(1 + \ln\left(k\right)\Big)h.
  \end{align*}
\end{thm}
As a consequence, when $n$ is a positive integer, we have that $ \left\|\mu_{T}-\mu^{T/n}_{T} \right\|_{\rm TV} \le \tilde C\sqrt{T}\left(\frac{1+\ln(n)}{n}\right)$.
Therefore, the order of convergence at the terminal time improves to $1$ up to some logarithmic factor. This in particular applies to the bounded one-dimensional time-homogeneous drift coefficient with bounded variation defined by Suo, Yuan and Zhang in Example 2.3 \cite{SuYuZ} using some Cantor set, for which they obtain convergence with order $1/4$ uniformly in time.

\begin{remark} \begin{itemize}
\item Of course, the regularity assumption on the drift coefficient $b$ is satisfied when it is Lipschitz in space (which is equivalent to the boundedness of its spatial gradient in the sense of distributions). When $d\ge 2$, the regularity assumption only involves the spatial divergence and not the full spatial gradient $\nabla b(t,.)$ in the sense of distributions. Note that if we suppose the stronger assumption $\sup_{t \in [0,T]} \|\nabla b(t,.)\|_{L^\rho}<\infty$ for $\rho\in(d,+\infty]$, then, according to the boundedness assumption and Corollary $IX.14$ \cite{Brezis}, the drift is locally $\left(1 - \frac{d}{\rho} \right)$-H\"older continuous in space.
\item Theorem 1.6 \cite{MenoKona} deals with a drift coefficient bounded, continuously differentiable with bounded derivatives up to the order $2$ in time and the order $4$ in space outside a finite union of time-independent smooth submanifolds where it can be discontinuous. Its constant diffusion statement says that when $m$ is a Dirac mass, then the absolute difference between the densities is bounded from above by a Gaussian density multiplied by a factor sum of a term with order $\frac{1}{d}-$ in our total variation rate $\frac{h}{\sqrt{kh}}$ and a term with order $1-$ in $h$ over the distance to the discontinuity set. 
\end{itemize}
\end{remark}

Furthermore, if we assume more regularity on $m$ in addition to the spatial regularity on $b$, we obtain the following result:
\begin{prop}\label{propregM}
  Assume $b:[0,T]\times\R^d \to \R^d$ is measurable and bounded by $B<+\infty$. Moreover, assume that $\,\sup_{t \in [0,T]} \|\nabla \cdot b(t,.)\|_{L^\rho}<+\infty$ for some $\rho \in [d,+\infty]$ or that for $d=1$, $\sup_{t \in [0,T]}\left\|\partial_x b(t,.)\right\|_{\rm TV}<+\infty$; where $\nabla \cdot b(t,.)$ and $\partial_x b(t,.)$ are respectively the spatial divergence and the spatial derivative of $b$ in the sense of distributions. If $m$ admits a density w.r.t. the Lebesgue measure that belongs to $W^{1,1}\left(\R^d\right)$ then:
  \begin{align*}
    \displaystyle \exists \, \hat{C}<+\infty, \;\forall h \in (0,T],\; \forall k \in \left\llbracket 0, \left\lfloor \frac{T}{h} \right\rfloor \right\rrbracket, \quad \left\|\mu_{kh}-\mu^h_{kh} \right\|_{\rm TV} \le \hat{C}\Big(1 + \ln\left(k\right)\Big)h.
  \end{align*}
\end{prop}
\begin{remark}
  \begin{itemize}
    \item We can see in Theorem \ref{cvgenceB} that when $b$ is more regular with respect to the space variables, the weak convergence rate in total variation is bounded by $\left(1 + \ln\left(k\right)\right)h$ times a prefactor $\frac{1}{\sqrt{kh}}$ that decreases over time and explodes in small time. According to Proposition \ref{propregM}, this prefactor is removed when assuming more regularity on $m$.
    \item For $\varphi:\R^d\to\R$, measurable and bounded and using Equation \eqref{totVar}, we deduce from Theorem \ref{cvgence} that $\left|\E\left[\varphi\left(X^h_{kh}\right)\right]-\E\left[\varphi\left(X_{kh}\right)\right]\right|\le C \|\varphi\|_\infty \sqrt h$, from Theorem \ref{cvgenceB} that $\left|\E\left[\varphi\left(X^h_{kh}\right)\right]-\E\left[\varphi\left(X_{kh}\right)\right]\right|\le \frac{\tilde{C} \|\varphi\|_\infty}{\sqrt{kh}}\left(1 + \ln\left(k\right)\right)h$ and from Proposition \ref{propregM} that $\left|\E\left[\varphi\left(X^h_{kh}\right)\right]-\E\left[\varphi\left(X_{kh}\right)\right]\right|\le \hat{C} \|\varphi\|_\infty \left(1 + \ln\left(k\right)\right)h$.
  \end{itemize}
 \end{remark}

The proofs of the two theorems and the proposition, that we will detail in Sections \ref{preuve} and \ref{preuveB}, rely on the propositions that we present in Section \ref{mild}. Before going any further, we investigate through the Lamperti transform the application of those theorems to SDEs with non-constant diffusion coefficient in dimension $d=1$.


\subsection{Application to one-dimensional SDEs with non-constant diffusion coefficient}

Let us consider the one-dimensional stochastic differential equation:
\begin{align}\label{diffusY}
  Y_t = Y_0 + \int_0^t \sigma\left(Y_s \right)\,dW_s + \int_0^t \beta\left(s,Y_s \right)\,ds, \quad t \in [0,T]
\end{align}
where $(W_t)_{t \ge 0}$ is a one-dimensional Brownian motion independent from $Y_0$ which is some $(l,r)$-valued random variable with $-\infty \le l < r \le +\infty$. Let $z \in (l,r)$, we assume that $\sigma: (l,r) \to \R^{*}_{+}$ is a $\mathcal{C}^1$ function with $\displaystyle \lim_{y \to \substack{r \\ l} } \int_{z}^y \frac{dw}{\sigma(w)}= \pm \infty$, $\beta: [0,T]\times (l,r) \to \R$ is measurable and that $(t,x) \mapsto \left( \frac{\beta(t,x)}{\sigma(x)} - \frac{\sigma^{'}(x)}{2}\right)$ is bounded on $[0,T] \times (l,r)$. We introduce the Lamperti transform $\Big(X_t = \psi\left(Y_t\right)\Big)_{t\in [0,T]}$ where $\psi: (l,r) \to \R$ is defined by $\displaystyle \psi(y) = \int_z^y \frac{dw}{\sigma(w)}$ . By It\^o's formula, $\left(X_t\right)_{t\in [0,T]}$ is solution to the dynamics \eqref{edsnl} for the choice $d=1$, $\displaystyle b(t,.)= \left( \frac{\beta(t,.)}{\sigma} - \frac{\sigma^{'}}{2}\right)\circ \psi^{-1}$ and initialized by $\psi(Y_0)$:
\begin{align}\label{diffusX}
  X_t = \psi(Y_0) + W_t + \int_0^t \left( \frac{\beta(s,.)}{\sigma} - \frac{\sigma^{'}}{2}\right)\circ\psi^{-1}\left(X_s \right) \,ds, \quad t \in [0,T].
\end{align}
Existence and uniqueness for \eqref{diffusY} can be deduced from the existence and uniqueness for \eqref{diffusX}.  Indeed, according to \cite{Veret}, the SDE \eqref{diffusX} admits a pathwise unique strong solution $\left(X_t\right)_{t \in [0,T]}$. By It\^o's formula, $\left(\psi^{-1}\left(X_t \right)\right)_{t \in [0,T]}$ is a solution to \eqref{diffusY}. As for the uniqueness, the images of any two solutions of \eqref{diffusY} by $\psi$ coincide by uniqueness for \eqref{diffusX}. Since $\psi: (l,r) \to \R$ is one to one, these two solutions coincide. For $t \ge 0$, we denote by $\nu_t$ the probability distribution of $Y_t$. The law $\mu_t$ of $X_t$ is then the pushforward of $\nu_t$ by $\psi$ i.e. $\mu_t = \psi \# \nu_t$ and conversely  $\nu_t = \psi^{-1} \# \mu_t$. \\ 

We are going to approximate $\left(Y^h_{kh}\right)_{k \in \left\llbracket0,\left\lfloor \frac{T}{h} \right\rfloor\right\rrbracket}$ by $\left(\psi^{-1}\left(X^h_{kh}\right)\right)_{k \in \left\llbracket0,\left\lfloor \frac{T}{h} \right\rfloor\right\rrbracket}$ where $\left(X^h_{kh}\right)_{k \in \left\llbracket0,\left\lfloor \frac{T}{h} \right\rfloor\right\rrbracket}$ is the Euler scheme of \eqref{diffusX} with time-step $h \in (0,T]$ initialized by $X^h_0 = \psi(Y_0)$ and evolving inductively on the time grid $(kh)_{k \in \left\llbracket0,\left\lfloor \frac{T}{h} \right\rfloor\right\rrbracket}$ by:
\begin{align}\label{eulerX}
  X^h_{(k+1)h}=X^h_{kh}+\left(W_{(k+1)h}-W_{kh}\right)+\left( \frac{\beta(\delta_k,.)}{\sigma} - \frac{\sigma^{'}}{2}\right)\circ \psi^{-1}\left(X^h_{kh}\right)h, 
\end{align}
where the random variables $\delta_k$ are independent, distributed according to the uniform law on $[kh,(k+1)h]$ and are independent from $\left(Y_0,(W_t)_{t\ge 0}\right)$.\\ We denote by $\nu^h_t$ the law of $\psi^{-1}\left(X^h_t \right)$. Since for $\varphi \in \mathcal{L}$ defined right after \eqref{totVar}, $\varphi \circ \psi^{-1} \in \mathcal{L}$, we have:
\begin{align*}
  \displaystyle \left\|\nu_{kh}-\nu^h_{kh} \right\|_{\rm TV} = \left\|\psi^{-1} \#\mu_{kh}-\psi^{-1} \#\mu^h_{kh} \right\|_{\rm TV} &= \sup_{\varphi \in \mathcal{L}} \int_{\R}\varphi\left(\psi^{-1}(x)\right)\left(\mu_{kh}(dx) - \mu_{kh}^h(dx)\right) \le \left\|\mu_{kh}-\mu^h_{kh} \right\|_{\rm TV}.
\end{align*}
On the other hand, for $\mathcal{\hat L}$ denoting the set of all measurable functions $\hat \varphi:(l,r) \to [-1,1]$, if $\varphi \in \mathcal{L}$ then $\hat \varphi = \varphi \circ \psi \in \mathcal{\hat L}$ and:
\begin{align*}
  \displaystyle \left\|\mu_{kh}-\mu^h_{kh} \right\|_{\rm TV} = \sup_{\varphi \in \mathcal{L}}\int_{\R} \varphi(x)\Big(\mu_{kh}(dx) - \mu_{kh}^h(dx)\Big) &= \sup_{\varphi \in \mathcal{L}}\int_{\R} \varphi\circ \psi(x)\left(\nu_{kh}(dx) - \nu_{kh}^h(dx)\right)\\
  &\le \sup_{\hat \varphi \in \mathcal{\hat L}}\int_{(l,r)} \hat \varphi(x)\left(\nu_{kh}(dx) - \nu_{kh}^h(dx)\right) = \left\|\nu_{kh}-\nu^h_{kh}\right\|_{\rm TV}.
\end{align*}

Therefore, $\left\|\nu_{kh}-\nu^h_{kh} \right\|_{\rm TV} = \left\|\mu_{kh}-\mu^h_{kh} \right\|_{\rm TV}$ and we obtain directly from Theorem \ref{cvgence} the following result for the weak convergence rate of $\nu^h_t$ towards $\nu_t$:
\begin{thm}\label{cvgenceY}
  Assume $\sigma:(l,r) \to \R^{*}_{+}$ is $\mathcal{C}^1$ with $\displaystyle \lim_{y \to \substack{r \\ l} } \int_{z}^y \frac{dw}{\sigma(w)}= \pm \infty$, $\beta:[0,T]\times(l,r) \to \R$ is measurable and $(t,x) \mapsto \left(\frac{\beta(t,x)}{\sigma(x)} - \frac{\sigma^{'}(x)}{2}\right)$ is bounded on $[0,T] \times (l,r)$. Then:
  \begin{align*}
    \exists \, C<+\infty,\;\forall h \in (0,T],\; \forall k \in \left\llbracket 0, \left\lfloor \frac{T}{h}\right\rfloor \right\rrbracket, \quad \left\|\nu_{kh}-\nu^h_{kh} \right\|_{\rm TV}\le C \sqrt h.
  \end{align*}
\end{thm}
Let us now discuss the assumptions on $\beta$ and $\sigma$ in order to apply Theorem \ref{cvgenceB}. According to Definition $3.4$ \cite{AmbFusPal}, the variation $V\Big(\left( \frac{\beta(t,.)}{\sigma} - \frac{\sigma^{'}}{2}\right)\circ \psi^{-1},\R\Big)$ of $\left( \frac{\beta(t,.)}{\sigma} - \frac{\sigma^{'}}{2}\right)\circ \psi^{-1}$ in $\R$ is defined by:
\begin{align*}
\displaystyle V\left(\left( \frac{\beta(t,.)}{\sigma} - \frac{\sigma^{'}}{2}\right)\circ \psi^{-1},\R\right) &:= \sup\left\{ \int_{\R}\left( \frac{\beta(t,.)}{\sigma} - \frac{\sigma^{'}}{2}\right)\circ \psi^{-1}(x)\varphi^{'}(x)\,dx :\; \varphi \in \mathcal{C}^1_c(\R), \; \|\varphi \|_{\infty}\le 1 \right\} \\
&:= \sup\left\{ \int_{(l,r)}\left( \frac{\beta(t,.)}{\sigma} - \frac{\sigma^{'}}{2}\right)(y)\left(\varphi\circ\psi\right)^{'}(y)\,dy :\; \varphi \in \mathcal{C}^1_c(\R), \; \|\varphi \|_{\infty}\le 1 \right\}. 
\end{align*}
Since $\psi$ is a $\mathcal{C}^1$-diffeomorphism from $(l,r)$ to $\R$:
\begin{align*}
\displaystyle V\left(\left( \frac{\beta(t,.)}{\sigma} - \frac{\sigma^{'}}{2}\right)\circ \psi^{-1},\R\right) &:= \sup\left\{ \int_{(l,r)}\left( \frac{\beta(t,.)}{\sigma} - \frac{\sigma^{'}}{2}\right)(y)\tilde \varphi^{'}(y)\,dy :\; \tilde \varphi \in \mathcal{C}^1_c((l,r)), \; \|\tilde \varphi \|_{\infty}\le 1 \right\} \\
&= V\left(\left( \frac{\beta(t,.)}{\sigma} - \frac{\sigma^{'}}{2}\right), (l,r) \right).
\end{align*}
Moreover, using Proposition $3.6$ \cite{AmbFusPal}, we have that: 
$$\displaystyle \left\|\partial_x b(t,.) \right\|_{\rm TV} = V\left(\left( \frac{\beta(t,.)}{\sigma} - \frac{\sigma^{'}}{2}\right)\circ \psi^{-1},\R\right) = V\left(\left( \frac{\beta(t,.)}{\sigma} - \frac{\sigma^{'}}{2}\right), (l,r) \right) = \left\|\partial_x \left(\frac{\beta(t,.)}{\sigma} - \frac{\sigma^{'}}{2}\right)\right\|_{\rm TV}$$ 
where the spatial derivatives are defined in the sense of distributions on $\R$ and $(l,r)$. Therefore, when assuming more regularity on $\left( \frac{\beta(t,.)}{\sigma} - \frac{\sigma^{'}}{2}\right)$ with respect to the space variables, we obtain a better rate of convergence:
\begin{thm}\label{cvgenceBY}
  Assume $\sigma:(l,r) \to \R^{*}_{+}$ is $\mathcal{C}^1$ with $\displaystyle \lim_{y \to \substack{r \\ l} } \int_{z}^y \frac{dw}{\sigma(w)}= \pm \infty$, $\beta:[0,T]\times(l,r) \to \R$ is measurable and $(t,x) \mapsto \left(\frac{\beta(t,x)}{\sigma(x)} - \frac{\sigma^{'}(x)}{2}\right)$ is bounded on $[0,T] \times (l,r)$. Moreover, assume that $\sup_{t \in [0,T]}\left\|\partial_x \left(\frac{\beta(t,.)}{\sigma} - \frac{\sigma^{'}}{2}\right)\right\|_{\rm TV}<+\infty$ where the spatial derivative is defined in the sense of distributions on $(l,r)$. Then:
  \begin{align*}
    \displaystyle \exists \, \tilde{C}<+\infty, \;\forall h \in (0,T],\; \forall k \in \left\llbracket 0, \left\lfloor \frac{T}{h} \right \rfloor \right\rrbracket, \quad \left\|\nu_{kh}-\nu^h_{kh} \right\|_{\rm TV} \le \frac{\tilde{C}}{\sqrt{kh}}\Big(1 + \ln\left(k\right)\Big)h.
  \end{align*}
\end{thm}
We will now discuss the additional assumptions on the law of $Y_0$ and $\sigma$ in order to apply Proposition \ref{propregM}. Let us assume that $Y_0$ admits a density $q_0$ w.r.t. the Lebesgue measure. By a change of variables, $X_0$ admits the density $(\sigma q_0)\circ \psi^{-1}$ w.r.t. the Lebesgue measure. Since $q_0 \in L^1\left((l,r)\right)$ and $\sigma$ is $\mathcal{C}^1$ so locally bounded on $(l,r)$, $(\sigma q_0)$ defines a distribution. Moreover, we assume that $\Big(\sigma q_0\Big)' \in L^1\left((l,r)\right)$ where the derivative of $(\sigma q_0)$ is defined in the sense of distributions. Now, let $I$ be a compact subinterval of $(l,r)$ and let $\varphi$ be a $\mathcal{C}^{\infty}_c$ function on $\R$ such that $(\varphi \circ \psi)$ is null outside $I$. We have, through a change of variables, that:
\begin{align*}
  \displaystyle \int_{\R} \varphi'(x)(\sigma q_0)\circ \psi^{-1}(x)\,dx = \int_{(l,r)} \Big(\varphi \circ \psi \Big)'(y) (\sigma q_0)(y)\,dy.
\end{align*}
Since $\left(\varphi \circ \psi\right) \in W^{1,1}\left(I\right)$ and $(\sigma q_0) \in W^{1,1}\left(I\right)$, we can apply Corollary $VIII.9$ \cite{Brezis} and obtain that:
\begin{align*}
  \displaystyle \int_{\R} \varphi'(x)(\sigma q_0)\circ \psi^{-1}(x)\,dx = - \int_{(l,r)} (\varphi \circ \psi)(y) \Big(\sigma q_0\Big)'(y)\,dy = -\int_{\R} \varphi(x) \Big(\sigma q_0\Big)'\circ \psi^{-1}(x) \sigma \circ \psi^{-1}(x)\,dx.
\end{align*}
Hence, we get, in the sense of distributions, that $\Big((\sigma q_0) \circ \psi^{-1}\Big)'=(\sigma q_0)'\circ \psi^{-1} \times  \sigma \circ \psi^{-1}$. Through a change of variables, we obtain that $\displaystyle \left\|\Big((\sigma q_0)\circ \psi^{-1} \Big)'\right\|_{L^1(\R)}=  \left\|\Big(\sigma q_0\Big)'\right\|_{L^1((l,r))} $ and since $(\sigma q_0)' \in L^1\left((l,r)\right)$, we conclude that the density of $X_0$ is in $W^{1,1}\left(\R\right)$.
\begin{prop}
  Assume $\sigma:(l,r) \to \R^{*}_{+}$ is $\mathcal{C}^1$ with $\displaystyle \lim_{y \to \substack{r \\ l} } \int_{z}^y \frac{dw}{\sigma(w)}= \pm \infty$, $\beta:[0,T]\times(l,r) \to \R$ is measurable and $(t,x) \mapsto \left(\frac{\beta(t,x)}{\sigma(x)} - \frac{\sigma^{'}(x)}{2}\right)$ is bounded on $[0,T] \times (l,r) $. Moreover, assume that $\sup_{t \in [0,T]}\left\|\partial_x \left(\frac{\beta(t,.)}{\sigma} - \frac{\sigma^{'}}{2}\right)\right\|_{\rm TV}<+\infty$ where the spatial derivative is defined in the sense of distributions on $(l,r)$. If $Y_0$ admits a density $q_0$ such that $\left\|(\sigma q_0)'\right\|_{L^1\left((l,r)\right)} < +\infty$, where the spatial derivative is defined in the sense of distributions, then:
  \begin{align*}
    \displaystyle \exists \, \hat{C}<+\infty, \;\forall h \in (0,T],\; \forall k \in \left\llbracket 0, \left\lfloor \frac{T}{h} \right\rfloor \right\rrbracket, \quad \left\|\nu_{kh}-\nu^h_{kh} \right\|_{\rm TV} \le \hat{C}\Big(1 + \ln\left(k\right)\Big)h.
  \end{align*}
\end{prop}


\subsection{Existence of densities and mild equations}\label{mild}

 We are going to state, in the next result, the existence for $t>0$ of the densities $p(t,.)$ and $p^h(t,.)$ by showing that $p(t,.)$ solves a mild equation and $p^{h}(t,.)$ solves a perturbed version of this mild equation. Let $G_t(x) = \exp\left(-\frac{|x|^2}{2t}\right) \Big/ \sqrt{(2\pi t)^d } $ denote the heat kernel in $\R^{d}$, we have:
\begin{prop}\label{propspde}
  For each $t\in(0,T]$, $\mu_t$ and $\mu^h_t$ admit densities $p(t,.)$ and $p^h(t,.)$ with respect to the Lebesgue measure on $\R^d$ s.t. we have $dx$ a.e.:
  \begin{align}
    \displaystyle  p(t,x) &= G_t*m(x) - \int_0^t \nabla G_{t-s}*\Big(b(s,.)p(s,.)\Big)(x)\,ds, \label{mildf}\\
    \forall h \in (0,T], \quad p^h(t,x) &= G_t*m(x)-\int_0^t \E \left[\nabla G_{t-s}\left(x-X^h_s\right)\cdot b\left(\delta_{\ell_s}, X^h_{\tau^h_s}\right)\right]\,ds. \label{mildfn1} 
  \end{align}
\end{prop}
\begin{proof}
Let $t>0$, $f$ be a $\mathcal{C}^2$ and compactly supported function on $\R^d$. We set $\varphi(s,x) = G_{t-s} * f(x)$ for $(s,x) \in [0,t) \times \R^d$ and $\varphi(t,x) = f(x)$. The function $\varphi(s,x)$ is continuously differentiable w.r.t. $s$ and twice continuously differentiable w.r.t. $x$ on $[0,t]\times\R^d$ with bounded derivatives and solves
\begin{equation}\label{edpvarphi}
  \partial_s \varphi(s,x)+\frac{1}{2}\Delta \varphi(s,x)=0\mbox{ for }(s,x)\in[0,t]\times\R^d.
\end{equation}
We compute $\E\left[\varphi(t,X_t)\right]$ where $(X_s)_{s\ge 0}$ solves \eqref{edsnl}. Using \eqref{edpvarphi}, applying Ito's formula and taking expectations, we obtain that:
$$\E\left[\varphi(t,X_t)\right]=\E\left[\varphi(0,X_0)+\int_0^t \nabla \varphi(s,X_s) \cdot b(s,X_s) \,ds\right].$$
By Fubini's Theorem and since $G_t$ is even, we obtain:
\begin{align*}
 \displaystyle \E\left[f(X_t)\right] &= \int_{\R^d} G_t * f(x)m(dx) + \int_{(0,t] \times \R^d} \sum\limits_{i=1}^d \int_{\R^d} \partial_{x_i}G_{t-s}(x-y)f(y)\,dy\, b_i(s,x) \mu_s(dx)\,ds \\
 &= \int_{\R^d} f(x) \bigg(G_t * m(x) - \int_0^t \nabla G_{t-s} * \left(b(s,.)p(s,.)\right)(x) \,ds\bigg) \,dx.
\end{align*}
Since $f$ is arbitrary, we conclude that $X_t$ admits a density that we denote by $p(t,.)$ and that satisfies the mild formulation \eqref{mildf}. \\
Let us establish that $X^h_t$ admits a density $p^{h}(t,.)$ that satisfies a perturbed version of the previous equation. Using similar arguments, we get:
$$\E\left[\displaystyle \varphi\left(t,X^h_t\right)\right] = \E\left[\varphi(0,X_0)+ \int_0^t \nabla \varphi \left(s,X^h_s\right) \cdot b\left(\delta_{\ell_s},X^h_{\tau^h_s}\right)\,ds\right].$$
Once again, by Fubini's Theorem and since $G_t$ is even, we obtain:
\begin{align*}
 \displaystyle \E\left[f\left(X^h_t\right)\right] &= \int_{\R^d} G_t * f(x) m(x)\,dx + \E \left[\int_{(0,t] \times \R^d} \left(\nabla G_{t-s}(X^h_s-x)f(x)\,dx\right) \cdot b\left(\delta_{\ell_s}, X^h_{\tau^h_s}\right)\,ds \right] \\
 &= \int_{\R^d} f(x) \bigg(G_t * m(x) - \int_0^t \E\left[\nabla G_{t-s}(x -X^h_s)\cdot  b\left(\delta_{\ell_s}, X^h_{\tau^h_s}\right) \right] \,ds \bigg) \,dx.
\end{align*}
The function $f$ being arbitrary, we can conclude.
\end{proof}

Now, let us put \eqref{mildfn1} in a form closer to \eqref{mildf} but with an additional perturbation term that we control in the following proposition.

\begin{prop}\label{mildfn2}
  Assume $b:[0,T]\times\R^d \to \R^d$ is measurable and bounded by $B<+\infty$. Then:
  \begin{align*}
    \displaystyle \forall h \in (0,T],\forall k \in \left\llbracket 1, \left\lfloor \frac{T}{h} \right\rfloor \right\rrbracket,\quad &p^h\left(kh,.\right) =  G_{kh}*m - \int_{0}^{kh} \nabla G_{kh-\tau^h_s} * \bigg(b(s,.)\mu^h_{\tau^h_s} \bigg)\,ds + R^h(k,.),\\
    \text{ where } \quad &\left\| R^h(k,.) \right\|_{L^1} \le 2dB^2 \left(1 + \frac{d-1}{\pi}\right) \left(\frac{1}{2} + \ln\left(k\right) \right)h.
  \end{align*}
\end{prop}

\begin{proof}
Let $k \in \left\llbracket 1,  \left\lfloor\frac{T}{h} \right\rfloor\right\rrbracket$. By \eqref{mildfn1} written for $t=kh$ and since $X^h_s = X^h_{jh} + \left(W_s-W_{jh}\right) + b\left(\delta_{j}, X^h_{jh}\right)\left(s- jh\right)$ for $s \in \left[jh, (j+1)h\right)$, we have $dx$ a.e.:
\begin{align*}
  \displaystyle &p^h(kh,x) \\
  &= G_{kh}*m(x)- \sum_{j=0}^{k-1} \int_{jh}^{(j+1)h} \E \bigg.\bigg[\E \Big[\nabla G_{kh-s}\left(x-X^h_s\right)\cdot b\left(\delta_{j}, X^h_{jh}\right) \Big\vert X^h_{jh}, \delta_{j}\bigg] \bigg]\,ds \\
  &= G_{kh}*m(x)- \sum_{j=0}^{k-1} \int_{jh}^{(j+1)h}\E \Bigg[ \E \left. \left[\nabla G_{kh-s}\Big(x- X^h_{jh} - \left(W_s - W_{jh} \right) - b\left(\delta_{j}, X^h_{jh}\right)\left(s- jh\right)\Big)\right\vert X^h_{jh}, \delta_{j}  \right]\cdot b\left(\delta_{j}, X^h_{jh}\right)\Bigg] \,ds.  
\end{align*}
Using the independence between the increments $\left(W_s - W_{jh}\right)_{s \ge jh}$ and $\left(X^h_{jh}, \delta_{j}\right)$ as well as the fact that the heat kernel is a convolution semi-group s.t. for $0 \le u < s < t, \; \nabla G_{t-s}*G_{s-u} = \nabla G_{t-u}$, we deduce:
\begin{align*}
  \displaystyle p^h(kh,x)= G_{kh}*m(x)- \sum_{j=0}^{k-1} \int_{jh}^{(j+1)h} \E\Bigg[\nabla G_{kh-jh}\bigg(x - X^h_{jh}- b\left(\delta_{j}, X^h_{jh}\right)\left(s- jh\right)\bigg)\cdot b\left(\delta_{j}, X^h_{jh}\right)\Bigg] \,ds.
\end{align*}
 Using Taylor's formula with integral reminder at first order, we obtain:
\begin{align*}
  &\displaystyle \nabla G_{kh-jh}\bigg(x- X^h_{jh} - b\left(\delta_{j},X^h_{jh}\right)\left(s- jh\right)\bigg) \\
  &= \nabla G_{kh-jh}\bigg(x- X^h_{jh}\bigg) - (s - jh) \int_0^1 \sum_{i=1}^d \partial_{x_i} \nabla G_{kh-jh}\bigg(x-X^h_{jh}-\alpha\, b\left(\delta_{j},X^h_{jh}\right)(s-t_j) \bigg) b_i\left(\delta_{j},X^h_{jh}\right)\,d\alpha.
\end{align*} 
 We plug this equality in the previous equation. We can easily see by induction using Equation \eqref{euler} that $\delta_{j}$ is independent from $X^h_{jh}$ for $j  \le \left\lfloor\frac{T}{h}\right\rfloor$. Therefore, we obtain:
\begin{align*}
  \displaystyle &p^h(kh,x) - G_{kh}*m(x) + \sum_{j=0}^{k-1} \int_{jh}^{(j+1)h} \nabla G_{kh-jh} * \Big(b(s,x)\mu^h_{jh}(dx)\Big) \,ds \\
  &= \sum_{j=0}^{k-1} \int_{jh}^{(j+1)h}(s-jh) \E\left[\int_0^1 \sum_{i=1}^d \sum_{l=1}^d \frac{\partial^2 }{\partial x_i x_l} G_{kh-jh}\Big(x-X^h_{jh}-\alpha\, b\left(\delta_{j},X^h_{jh}\right)(s-jh) \Big)b_i\left(\delta_{j},X^h_{jh}\right)b_l\left(\delta_{j},X^h_{jh}\right)\,d\alpha \right] \,ds.
\end{align*}
 We denote by $R^h(k,x)$ the right-hand side of the previous equation. To upper-bound the $L^1$-norm of $R^h$, we use the estimates \eqref{SecondDerivG1} and \eqref{SecondDerivG2} from Lemma \ref{EstimHeatEq}, and the boundedness of $b$ to obtain:
\begin{align*}
  \displaystyle \left\|R^h(k,.)\right\|_{L^1} &\le \int_0^{kh} (s-\tau^h_s) \sum_{i=1}^d \sum_{l=1}^d \left\|\frac{\partial^2 }{\partial x_i x_l} G_{kh-\tau^h_s} \right\|_{L^1}\left\|b_i \right\|_{L^{\infty}}\left\|b_l \right\|_{L^{\infty}} \,ds \\
  &\le 2dB^2 \left(1 + \frac{d-1}{\pi}\right) \int_0^{kh} \frac{s-\tau^h_s}{kh - \tau^h_s}\,ds \\
  &\le 2dB^2 \left(1 + \frac{d-1}{\pi}\right) \left( \int_0^{(k-1)h} \frac{h}{kh - s}\,ds + \int_{(k-1)h}^{kh} \frac{s-(k-1)h}{h}\,ds \right) \\
  &= 2dB^2 \left(1 + \frac{d-1}{\pi}\right) \left(-h\ln\left(h\right) + h \ln(kh) + \frac{h}{2} \right).
\end{align*} 
One can easily conclude. 
\end{proof}


\section{Proof of the convergence rate in total variation}\label{preuve}

To prove Theorem \ref{cvgence}, we need the following Lemma that gives an estimation of the regularity of $p(t,.)$ with respect to the time variable. 
\begin{lem}\label{evoldens}
  Assume $b:[0,T]\times\R^d \to \R^d$ is measurable and bounded by $B<+\infty$.
  $$\displaystyle \exists \, Q<+\infty, \forall \, 0 < r\le s\le T,\quad \left\|p(s,.) - p(r,.) \right\|_{L^1} \le Q\bigg(\ln(s/r)+\sqrt{s-r}\bigg).$$
\end{lem}
\begin{proof}
Let $0 < r \le s \le T$. We have:
\begin{align*}
  \displaystyle p(s,.) - p(r,.) = \left( G_s - G_r \right) * m - \int_0^s \nabla G_{s-u} *\Big(b(u,.)p(u,.)\Big)\,du + \int_0^r \nabla G_{r-u} *\Big(b(u,.)p(u,.)\Big)\,du.
\end{align*}
Using the estimates \eqref{timeDerivG}, \eqref{FirstDerivG} and \eqref{ThirdDerivG} from Lemma \ref{EstimHeatEq}, we obtain:
\begin{align*}
  \displaystyle &\left\|p(s,.) - p(r,.) \right\|_{L^1} \\
  &\le \left\|\left( G_s - G_r \right) * m \right\|_{L^1} + \left\|\int_0^r \left(\nabla G_{s-u} - \nabla G_{r-u} \right)*\Big(b(u,.)p(u,.)\Big)\,du \right\|_{L^1} + \left\|\int_r^s \nabla G_{s-u} *\Big(b(u,.)p(u,.)\Big)\,du \right\|_{L^1} \\
  &\le \left\|\int_r^s \left(\partial_u G_u* m\right)\,du \right\|_{L^1} + \left\|\int_0^r \int_{r-u}^{s-u} \partial_{\theta} \nabla G_{\theta} *\Big(b(u,.)p(u,.)\Big)\,d\theta\,du\right\|_{L^1}+ dB\int_r^s \sum\limits_{i=1}^d \left\| \partial_{x_i} G_{s-u}\right\|_{L^1} \,du \\
  &\le \int_r^s \left\|\partial_u G_u \right\|_{L^1} \,du + \frac{B}{2} \int_0^r \int_{r-u}^{s-u} \sum\limits_{i=1}^d \sum\limits_{j=1}^d \left\|  \frac{\partial^3G_{\theta}}{ \partial x_i\partial x_j^2} \right\|_{L^1}d\theta \,du  + 2\sqrt{\frac{2}{\pi}}dB  \sqrt{s-r} \\
  &\le d\ln(s/r) + 2\sqrt{\frac{2}{\pi}} (2d+3)dB  \Big(\sqrt{s-r} - \left(\sqrt{s} - \sqrt{r}\right)\Big)+ 2\sqrt{\frac{2}{\pi}}dB \sqrt{s-r}.
\end{align*}
  The conclusion holds with $\displaystyle Q = d\max\left(1,4\sqrt{\frac{2}{\pi}}(d+2)B\right)$.
\end{proof}
We are now ready to prove Theorem \ref{cvgence}. Since $\mu_0 = \mu_0^h = m$, using Equality \eqref{caract} and Proposition \ref{propspde}, to prove the theorem amounts to prove that:
\begin{align*}
  \exists \, C<+\infty, \forall h \in (0,T],\; \forall k \in \left\llbracket 1, \left \lfloor \frac{T}{h} \right \rfloor\right\rrbracket, \quad \left\|p^h\left(kh,.\right)-p\left(kh,.\right)\right\|_{L^1} \le C \sqrt h.
\end{align*}
Let $k  \in \left\llbracket 1, \left \lfloor \frac{T}{h} \right \rfloor \right\rrbracket$, we have: 
\begin{align*}
  \displaystyle p^h\left(kh,.\right) - p\left(kh,.\right) =   V^h(k,.) + R^h(k,.) - \int_0^{h} \bigg(\nabla G_{kh-s} * \Big(b(s,.)p(s,.)\Big) - \nabla G_{kh} * \Big(b(s,.)m\Big)\bigg)\,ds
\end{align*}
where $R^h(k,.)$ is defined in Proposition \ref{mildfn2} and $V^h(k,.)$ is defined by: 
$$\displaystyle V^h(k,.) = \int_{h}^{kh}\bigg(\nabla G_{kh-s} * \Big(b(s,.)p(s,.) \Big)-\nabla G_{kh-\tau^h_s} * \Big(b(s,.)p^h\left(\tau^h_s,.\right)\Big)\bigg)\,ds.$$ 
Since $V^h(1,.) = 0$, we suppose $k \ge 2$ and express $V^h(k,.)$ as $V^h(k,.)= \sum\limits_{p=1}^{3}V^h_p(k,.)$ where:
\begin{align*}
  V_1^h(k,.) &= \int_{h}^{kh}\bigg(\nabla G_{kh-s}- \nabla G_{kh-\tau^h_s} \bigg)* \Big(b(s,.)p(s,.) \Big)\,ds,\\
  V_2^h(k,.) &= \int_{h}^{kh}\nabla G_{kh-\tau^h_s} * \bigg(b(s,.)\Big[p(s,.)-p\left(\tau^h_s,.\right) \Big]\ \bigg)\,ds,\\
  V_3^h(k,.) &= \int_{h}^{kh}\nabla G_{kh-\tau^h_s} * \bigg(b(s,.)\Big[p(\tau^h_s,.)-p^h\left(\tau^h_s,.\right) \Big]\ \bigg)\,ds.
\end{align*}
On the one hand, using the estimate \eqref{FirstDerivG} from Lemma \ref{EstimHeatEq}, we obtain immediately that:
\begin{align}\label{v3}
  \displaystyle \left\|V_3^h(k,.)\right\|_{L^1} &\le \sqrt{\frac 2 \pi} dB\int_{h}^{kh}\frac{1}{\sqrt{kh-\tau^h_s}}\left\|p(\tau^h_s,.) - p^h(\tau^h_s,.) \right\|_{L^1}\,ds = \sqrt{\frac 2 \pi}dB\sum_{j=1}^{k-1} \frac{h}{\sqrt{kh - jh}}\left\|p(jh,.) - p^h(jh,.) \right\|_{L^1}.
\end{align}
On the other hand, using the estimate \eqref{FirstDerivG} from Lemma \ref{EstimHeatEq}, we have for the first time-step that:

\begin{align}\label{firstStep}
  \left\| \int_0^{h} \bigg(\nabla G_{kh-s} * \Big(b(s,.)p(s,.)\Big) - \nabla G_{kh} * \Big(b(s,.)m\Big)\bigg)\,ds \right\|_{L^1} &\le B \int_0^h \sum_{i=1}^d \Big(\left\|\partial_{x_i} G_{kh-s} \right\|_{L^1} + \left\|\partial_{x_i} G_{kh} \right\|_{L^1} \Big)\,ds \nonumber \\
  &= \sqrt{\frac 2 \pi}dB \left( 2\left(\sqrt{kh} - \sqrt{(k-1)h}\right)+\frac{h}{\sqrt{kh}}\right)\nonumber\\
  &\le \sqrt{\frac 2 \pi}dB\frac{3h}{\sqrt{kh}}.
\end{align}

Now, using Inequality \eqref{v3}  and Proposition \ref{mildfn2} for the first inequality, Inequality \eqref{firstStep} for the second inequality and finally the fact that $\ln(k) \le \ln\left(\frac T h\right)$ with $\sup_{x >0}\left\{ \left(\frac{1}{2}+\ln\left(\frac{T}{x}\right) \right)\sqrt{x} \right\}$ is attained for $x = T e^{-\frac 3 2}$ for the third inequality, we obtain: 

\begin{align}
  \displaystyle &\left\|p(kh,.)-p^h\left(kh,.\right)\right\|_{L^1} \nonumber\\
  &\le \left\|V^h_1(k,.)\right\|_{L^1}+ \left\|V^h_2(k,.)\right\|_{L^1}+ \sqrt{\frac 2 \pi}dB\sum_{j=1}^{k-1} \frac{h}{\sqrt{kh - jh}}\left\|p(jh,.) - p^h(jh,.) \right\|_{L^1} \nonumber \\
  &\phantom{\|V}+ 2dB^2 \left(1 + \frac{d-1}{\pi}\right) \left(\frac{1}{2} + \ln\left(k\right) \right)h + \left\| \int_0^{h} \bigg(\nabla G_{kh-s} * \Big(b(s,.)p(s,.)\Big) - \nabla G_{kh} * \Big(b(s,.)m\Big)\bigg)\,ds \right\|_{L^1} \label{firstintegB} \\
  &\le \left\|V^h_1(k,.)\right\|_{L^1}+ \left\|V^h_2(k,.)\right\|_{L^1}+ \sqrt{\frac 2 \pi}dB\sum_{j=1}^{k-1} \frac{h}{\sqrt{kh - jh}}\left\|p(jh,.) - p^h(jh,.) \right\|_{L^1}  \nonumber \\
  &\phantom{\|V\|}+ 2dB^2 \left(1 + \frac{d-1}{\pi}\right) \left(\frac{1}{2} + \ln\left(k \right) \right)h + \sqrt{\frac 2 \pi}dB\frac{3h}{\sqrt{kh}} \label{firstinteg} \\
  &\le \left\|V^h_1(k,.)\right\|_{L^1}+ \left\|V^h_2(k,.)\right\|_{L^1}+ \sqrt{\frac 2 \pi}dB \sum_{j=1}^{k-1} \frac{h}{\sqrt{kh - jh}}\left\|p(jh,.) - p^h(jh,.) \right\|_{L^1} \nonumber \\
  &\phantom{\|V\|}+ \sqrt{\frac 2 \pi}dB\left(2\sqrt{\frac{2T}{\pi}}B(\pi + d-1)e^{-\frac{3}{4}} + 3 \right) \sqrt h \label{ici1}.
\end{align}

Let us now estimate $\left\|V^h_1(k,.)\right\|_{L^1}$ and $\left\|V^h_2(k,.)\right\|_{L^1}$ for $k \ge 2$.\\ \\
$\bullet$ For $p=1$, using the estimates \eqref{FirstDerivG} and \eqref{ThirdDerivG} from Lemma \ref{EstimHeatEq}, we obtain:
\begin{align*}
\displaystyle \left\|V_1^h(k,.)\right\|_{L^1} &\le \frac{B}{2} \int_{h}^{(k-1)h} \int_{kh - s}^{kh - \tau^h_s}\sum_{i=1}^d \sum_{j=1}^d\left\|\frac{\partial^3 G_r}{\partial x_i \partial x^2_j}\right\|_{L^1}dr\,ds + B\int_{(k-1)h}^{kh}\bigg(\left\|\nabla G_{kh-s} \right\|_{L^1} + \left\|\nabla G_{h} \right\|_{L^1}\bigg)\,ds\\
&\le \frac{1}{2}\sqrt{\frac{2}{\pi}}(2d+3)dB\int_{h}^{(k-1)h} \int_{kh - s}^{kh - \tau^h_s}\frac{dr}{r^{3/2}}\,ds + dB\sqrt{\frac{2}{\pi}}\int_{(k-1)h}^{kh}\bigg(\frac{1}{\sqrt{kh-s}} + \frac{1}{\sqrt{ h}} \bigg)\,ds\\
&\le \sqrt{\frac{2}{\pi}}(2d+3)dB \int_{h}^{(k-1)h} \frac{h}{2\left(kh-s\right)^{3/2}}\,ds + 3dB\sqrt{\frac{2}{\pi}}\sqrt{h} \le 2\sqrt{\frac{2}{\pi}}(d+3)dB\sqrt{h}.
\end{align*}
$\bullet$ For $p=2$, using the estimate \eqref{FirstDerivG} from Lemma \ref{EstimHeatEq} and Lemma \ref{evoldens}, we obtain:
\begin{align*}
\displaystyle \left\|V_2^h(k,.)\right\|_{L^1} &\le \sqrt{\frac 2 \pi}dB\int_{h}^{kh}\frac{1}{\sqrt{kh-\tau^h_s}} \left\|p(s,.) - p(\tau^h_s,.) \right\|_{L^1}\,ds \\
&\le \sqrt{\frac 2 \pi}dBQ\Bigg( \int_{h}^{kh} \frac{\ln(s/ \tau^h_s) }{\sqrt{kh-\tau^h_s}}\,ds + \int_{h}^{kh} \frac{\sqrt{s-\tau^h_s}}{\sqrt{kh-\tau^h_s}}\,ds \Bigg).
\end{align*}
Using the fact that for $\displaystyle s \ge h,\; \ln\left(\frac{s}{\tau^h_s}\right) \le \frac{s-\tau^h_s}{\tau^h_s} \le \frac{2h}{s}$ and Lemma \ref{integ}, we have that:
\begin{align}\label{integLog}
\displaystyle \int_{h}^{kh} \frac{\ln(s/ \tau^h_s) }{\sqrt{kh-\tau^h_s}}\,ds &\le \int_{h}^{kh}\frac{2h}{s\sqrt{kh-s}}\,ds \le \frac{2h}{\sqrt{kh}} \ln\left(4k\right).
\end{align}
Moreover, using the fact that $\displaystyle \frac{\sqrt{s-\tau^h_s}}{\sqrt{kh - \tau^h_s}} \le \frac{\sqrt h}{\sqrt{kh -s}}$, we deduce that:
\begin{align*}
\displaystyle \left\|V_2^h(k,.)\right\|_{L^1} &\le 2\sqrt{\frac 2 \pi}dBQ\Bigg(\sup_{x\ge 1}\frac{\ln(4x)}{\sqrt x} + \sqrt{(k-1)h}\Bigg) \sqrt h \le 2\sqrt{\frac 2 \pi}dBQ\Bigg(\frac{4}{e} +\sqrt T\Bigg) \sqrt h.
\end{align*}
Hence, using \eqref{ici1} and the above estimates for $\left\|V_1^h(k,.)\right\|_{L^1}$ and $\left\|V_2^h(k,.)\right\|_{L^1}$, we obtain:
\begin{align*}
  \displaystyle \left\|p(kh,.) - p^h(kh,.) \right\|_{L^1} \le L \sqrt{h} + \sqrt{\frac{2}{\pi}}dB \sum_{j=1}^{k-1} \frac{h}{\sqrt{kh - jh}}\left\|p(jh,.) - p^h(jh,.) \right\|_{L^1}
\end{align*}
where $\displaystyle L = 2\sqrt{\frac 2 \pi}dB\Bigg(\left(d + \frac 9 2 \right)+ Q\left(\frac{4}{e} + \sqrt T \right)+ \sqrt{\frac{2T}{\pi}}B(\pi + d-1)e^{-\frac 3 4} \Bigg)$. We iterate this inequality to obtain:
\begin{align*}
  \displaystyle \left\|p(kh,.) - p^h(kh,.) \right\|_{L^1} \le L\left(1 + 2\sqrt{\frac{2}{\pi}}dB \sqrt{(k-1)h} \right)\sqrt{h}+ \frac{2d^2B^2}{\pi} \sum_{j=1}^{k-1} \sum_{l=1}^{j-1}\frac{h}{\sqrt{k-j}\sqrt{j-l}}\left\|p(lh,.) - p^h(lh,.) \right\|_{L^1}.
\end{align*}
We re-write the double-sum the following way:
\begin{align*}
  \displaystyle \sum_{j=1}^{k-1} \sum_{l=1}^{j-1}\frac{h}{\sqrt{k-j}\sqrt{j-l}}\left\|p(lh,.) - p^h(lh,.) \right\|_{L^1} &= \sum_{l=1}^{k-2} \sum_{j=l+1}^{k-1}\frac{h}{\sqrt{k-j}\sqrt{j-l}}\left\|p(lh,.) - p^h(lh,.) \right\|_{L^1} \\
  &= h \sum_{l=1}^{k-2} \Bigg(\left\|p(lh,.) - p^h(lh,.) \right\|_{L^1} \sum_{i=1}^{k-l-1} \frac{1}{\sqrt{i}\sqrt{(k-l) - i} }\Bigg)\\
  &\le \pi \int_{h}^{kh}\left\|p\left(\tau^h_s,.\right) - p^h\left(\tau^h_s,.\right) \right\|_{L^1} \,ds,
\end{align*}
where we used Lemma \ref{sumPi} for the last inequality. Therefore, 
\begin{align*}
\displaystyle \left\|p(kh,.) - p^h(kh,.) \right\|_{L^1} \le L\left(1 + 2\sqrt{\frac{2T}{\pi}}dB \right)\sqrt{h} + 2d^2B^2 h \sum_{j=1}^{k-1}\left\|p\left(jh,.\right) - p^h\left(jh,.\right) \right\|_{L^1}.
\end{align*}
We apply Lemma \ref{Gronw} and obtain:
\begin{align*}
\displaystyle \left\|p(kh,.) - p^h(kh,.) \right\|_{L^1} &\le L\left(1 + 2\sqrt{\frac{2T}{\pi}}dB \right)\sqrt{h} + 2d^2B^2 L\left(1 + 2\sqrt{\frac{2T}{\pi}}dB \right)\sqrt{h} \sum_{j=1}^{k-1} h \exp\Big(2d^2B^2(kh -(j+1)h) \Big) \\
&\le L\left(1 + 2\sqrt{\frac{2T}{\pi}}dB \right)\sqrt{h} + 2d^2B^2 L\left(1 + 2\sqrt{\frac{2T}{\pi}}dB \right)\sqrt{h}\exp\Big(2d^2B^2T\Big)(kh-h)\\
&\le  L\left(1 + 2\sqrt{\frac{2T}{\pi}}dB \right)\left(1 + 2d^2B^2T\exp\Big(2d^2B^2T\Big) \right)\sqrt{h}.
\end{align*}
The conclusion holds with $\displaystyle C =  L\left(1 + 2\sqrt{\frac{2T}{\pi}}dB \right)\left(1 + 2d^2B^2T\exp\Big(2d^2B^2T\Big) \right)$.


\section{Proof of the convergence rate in total variation when assuming more regularity on $b$ w.r.t. to the space variables}\label{preuveB}

The following proposition, developed in Subsection \ref{hypH}, enables to establish an estimate of the total variation norm of the divergence of $b(t,.)p(t,.)$ for $t>0$ from the regularity assumed on $b$ w.r.t. to the space variables. When assuming extra regularity on $m$, the estimate is improved.
\begin{prop}\label{hyp}
  Assume $b:[0,T]\times\R^d \to \R^d$ is measurable and bounded by $B<+\infty$. If $\,\sup_{t \in [0,T]} \|\nabla \cdot b(t,.)\|_{L^\rho}<+\infty$ for some $\rho \in [d,+\infty]$ or for $d=1$, $\sup_{t \in [0,T]}\left\|\partial_x b(t,.)\right\|_{\rm TV}<+\infty$; where $\nabla \cdot b(t,.)$ and $\partial_x b(t,.)$ are respectively the spatial divergence and the spatial derivative of $b$ in the sense of distributions. Then:
  \begin{align}\label{hyp1}
    \displaystyle \exists\, M <+\infty, \forall t \in(0,T], \quad \left\|\nabla \cdot \Big(b(t,.)p(t,.) \Big) \right\|_{\rm TV} \le \frac{M}{\sqrt t},
  \end{align}
  and:
  \begin{align}\label{nablaP}
    \forall t \in (0,T], \quad p(t,.) = G_t*m - \int_0^t G_{t-s} * \nabla \cdot \Big(b(s,.)p(s,.)\Big)\,ds.
  \end{align}
  Moreover, if $m$ admits a density w.r.t. the Lebesgue measure that belongs to $W^{1,1}\left(\R^d\right)$, we obtain:
  \begin{align}\label{hyp2}
    \displaystyle \exists \, \tilde{M} <+\infty, \forall t \in [0,T], \quad \left\|\nabla \cdot \Big(b(t,.)p(t,.) \Big) \right\|_{\rm TV} \le \tilde{M}.
  \end{align}
\end{prop}
\begin{remark}
In fact, when $\,\sup_{t \in [0,T]} \|\nabla \cdot b(t,.)\|_{L^\rho}<+\infty$ for some $\rho \in [d,+\infty]$, we will prove that for $t \in (0,T]$, $\nabla \cdot \Big(b(t,.)p(t,.) \Big) \in L^1\left(\R^d\right)$ and $\left\|\nabla \cdot \Big(b(t,.)p(t,.) \Big) \right\|_{\rm TV} = \left\|\nabla \cdot \Big(b(t,.)p(t,.) \Big) \right\|_{L^1}$.
\end{remark}

The results of Proposition \ref{hyp} will be used in the proof of Theorem \ref{cvgenceB} detailed in Subsection \ref{demoo}, and in the proof of Proposition \ref{propregM} detailed in Subsection \ref{demooo}.\\ \\
We bring to attention that, in this subsection, all the derivatives and divergence are defined in the sense of distributions.


\subsection{Proof of Proposition \ref{hyp}}\label{hypH}

Let $\theta \in (0,T]$, we define the Banach spaces $\displaystyle \mathcal{\bar C}\left((0,\theta],L^1\left(\R^d\right)\right) = \left\{q \in \mathcal{C}\left((0,\theta],L^1\left(\R^d\right)\right): \sup_{t \in (0,\theta]}\left\|q(t,.)\right\|_{L^1}<+\infty \right\}$, $\displaystyle \mathcal{\tilde{C}}\left((0,\theta],W^{1,1}\left(\R^d\right)\right) = \left\{q \in \mathcal{C}\left((0,\theta],W^{1,1}\left(\R^d\right)\right): \opnorm{q} = \sup_{t \in (0,\theta]}\left\|q(t,.)\right\|_{L^1} + \sup_{t \in (0,\theta]}\sqrt{t}\left\|\nabla q(t,.)\right\|_{L^1}<+\infty \right\}$ and $\mathcal{C}\left([0,\theta],W^{1,1}\left(\R^d\right)\right)$ endowed respectively with the norms $\displaystyle \sup_{t \in (0,\theta]}\left\|q(t,.)\right\|_{L^1}$, $\opnorm{q}$ and $\displaystyle \sup_{t \in [0,\theta]}\left\|q(t,.)\right\|_{W^{1,1}}$. One has $\mathcal{C}\left([0,\theta],W^{1,1}\left(\R^d\right)\right) \subset \mathcal{\tilde{C}}\left((0,\theta],W^{1,1}\left(\R^d\right)\right) \subset \mathcal{\bar C}\left((0,\theta],L^1\left(\R^d\right)\right)$.\\

The next theorem states regularity properties of the density $\left(p(t,.)\right)_{t \in (0,T]}$ when assuming more regularity on $b$ w.r.t. the space variables. 
\begin{thm}\label{sobolP}
  Assume $b:[0,T]\times\R^d \to \R^d$ is measurable and bounded by $B<+\infty$. If $\,\sup_{t \in [0,T]} \|\nabla \cdot b(t,.)\|_{L^\rho}<+\infty$ for some $\rho \in [d,+\infty]$ or for $d=1$, $\sup_{t \in [0,T]}\left\|\partial_x b(t,.)\right\|_{\rm TV}<+\infty$ then $p \in \mathcal{\tilde C}\left((0,T],W^{1,1}\left(\R^d \right) \right)$. Moreover, if $m$ admits a density w.r.t. the Lebesgue measure in $W^{1,1}\left(\R^d\right)$ then $p\in \mathcal{C}\left([0,T],W^{1,1}\left(\R^d\right)\right)$.
\end{thm}
The proof of Theorem \ref{sobolP} relies on the uniqueness of the mild equation \eqref{mildf}. This latter can be proved by a fixed-point method. To do so, for $\theta \in (0,T]$, we define on the space $\mathcal{\bar C}\left((0,\theta],L^1\left(\R^d\right)\right)$ the map $\Phi$:
\begin{align*}
  \displaystyle \Phi:q \mapsto \Bigg(\Phi_t(q) = G_t * m - \int_0^t \nabla G_{t-s} * \Big( b(s,.)q(s,.)\Big)\,ds \Bigg)_{t \in (0,\theta]}.
\end{align*}
By a slight abuse of notation, we do not make explicit the dependence of the map $\Phi$ on the time horizon $\theta$. Let us check that $\Phi$ is well-defined. For $t \in (0,\theta]$, we have, using the estimate \eqref{FirstDerivG} from Lemma \ref{EstimHeatEq}, that:
\begin{align}\label{estimL1}
  \displaystyle \left\|\Phi_t(q)\right\|_{L^1} \le \left\|G_t*m\right\|_{L^1} + dB\sqrt{\frac 2 \pi}\int_0^t \frac{1}{\sqrt{t-s}}\left\|q(s,.)\right\|_{L^1}\,ds \le 1 + 2dB\sqrt{\frac{2 \theta}{\pi}}\sup_{u \in (0,\theta]}\|q(u,.)\|_{L^1}.
\end{align} 
Hence, since $q \in \mathcal{\bar C}\left((0,\theta],L^1\left(\R^d\right)\right)$, we have that $\sup_{t \in (0,\theta]}\left\|\Phi_t(q)\right\|_{L^1} <+\infty$.\\
The following result ensures that the map $\Phi$ admits a unique fixed-point in $\displaystyle \mathcal{\bar C}\left((0,\theta],L^1\left(\R^d\right)\right)$.
\begin{lem}\label{uniqFP}
  Assume $b:[0,T]\times\R^d \to \R^d$ is measurable and bounded by $B<+\infty$. For all $\theta \in (0,T]$, $\left(p(t,.)\right)_{t \in (0,\theta]}$ is the unique fixed-point of the map $\Phi$ in $\mathcal{\bar C} \left((0,\theta], L^1\left(\R^d \right)\right)$. 
\end{lem}
\begin{proof}
Let $q \in \mathcal{\bar C}\left((0,\theta], L^1\left(\R^d \right)\right)$. Using Inequality \eqref{estimL1}, we have that $\sup_{t \in (0,\theta]}\left\|\Phi_t(q)\right\|_{L^1} <+\infty$.\\ For $0 < r \le s \le \theta$, adapting the proof of Lemma \ref{evoldens}, we obtain that: 
\begin{align}\label{timeCtyPhi}
  \displaystyle \left\|\Phi_s(q) - \Phi_r(q) \right\|_{L^1} \le \max\left(1,\sup_{u \in [r,s]}\|q(u,.)\|_{L^1}\right)Q \Big(\ln(s/r) + \sqrt{s-r} \Big).
\end{align}
Therefore, $t \mapsto \Phi_t(q)$ is continuous on $(0,\theta]$ with values in $L^1\left(\R^d\right)$ and $\Phi(q) \in \mathcal{\bar C}\left((0,\theta],L^1\left(\R^d\right)\right)$. \\
Now, let $q,\tilde q \in \mathcal{\bar C}\left((0,\theta],L^1\left(\R^d\right)\right)$. Using the same reasoning as for Inequality \eqref{estimL1}, we obtain:
\begin{align}\label{equA}
  \displaystyle \left\|\Phi_t(q) - \Phi_t(\tilde q)\right\|_{L^1} &\le dB\sqrt{\frac 2 \pi}\int_0^t \frac{1}{\sqrt{t-s}}\left\|q(s,.) - \tilde q(s,.)\right\|_{L^1}\,ds. 
\end{align} 
Let $n \in \N^{*}$, we define $\Phi^{n+1} = \Phi^n \circ \Phi = \Phi \circ \Phi^n$ and iterate Inequality \eqref{equA} $2n$-times to obtain:
\begin{align*}
  \displaystyle \left\|\Phi^{2n}_t(q) - \Phi^{2n}_t(\tilde q)\right\|_{L^1} &\le \left(dB\sqrt{\frac 2 \pi}\right)^{2n} \pi^n \int_0^t \frac{(t-s)^{n-1}}{(n-1)!} \left\|q(s,.) - \tilde q(s,.)\right\|_{L^1}\,ds \\
  &\le (dB)^{2n}\frac{(2\theta)^n}{n!}\; \sup_{u \in (0,\theta]}\left\|q(u,.) - \tilde q(u,.)\right\|_{L^1}.
\end{align*} 
Therefore, $\sup_{t \in (0,\theta]}\left\|\Phi^{2n}_t(q) - \Phi^{2n}_t(\tilde q)\right\|_{L^1} \le (dB)^{2n}\frac{(2T)^n}{n!}\; \sup_{t \in (0,\theta]}\left\|q(t,.) - \tilde q(t,.)\right\|_{L^1} $ and for $n$ big enough, $\Phi^{2n}$ is a contraction on $\mathcal{\bar{C}}\left((0,\theta], L^1\left(\R^d \right)\right)$. By Picard's Theorem, $\Phi^{2n}$ admits then a unique fixed-point $q$ in $\mathcal{\bar{C}} \left((0,\theta], L^1\left(\R^d \right)\right)$. We have that $\Phi(q) = \Phi\left(\Phi^{2n}(q)\right) =  \Phi^{2n}\left(\Phi(q)\right)$ making $\Phi(q)$ a fixed-point of $\Phi^{2n}$, but since this latter is unique, we conclude that $\Phi(q) = q$. For $t>0$, $p(t,.)$ is solution to the mild equation \eqref{mildf} and using Lemma \ref{evoldens}, $p \in \mathcal{\bar C}\left((0,\theta],L^1\left(\R^d\right)\right)$. Consequently, $\left(p(t,.)\right)_{t \in (0,\theta]}$ is the unique fixed-point of the map $\Phi$ in $\mathcal{\bar C} \left((0,\theta], L^1\left(\R^d\right)\right)$.
\end{proof}
Now, we seek to establish more regularity on the fixed-point of the map $\Phi$. 
\begin{prop}\label{regFixpoint}
Assume $b:[0,T]\times\R^d \to \R^d$ is measurable and bounded by $B<+\infty$. Moreover, assume that $\,\sup_{t \in [0,T]} \|\nabla \cdot b(t,.)\|_{L^\rho}<+\infty$ for some $\rho \in [d,+\infty]$ or that for $d=1$, $\sup_{t \in [0,T]}\left\|\partial_x b(t,.)\right\|_{\rm TV}<+\infty$.
\begin{itemize}
  \item If $m$ admits a density w.r.t the Lebesgue measure in $W^{1,1}\left(\R^d\right)$ then for all $\theta \in (0,T]$, $\Phi$ admits a unique fixed-point in the space $\mathcal{C}\left([0,\theta], W^{1,1}\left(\R^d \right)\right)$.
  \item Otherwise, there exists $\theta_0 \in (0,T]$ s.t. $\Phi$ admits a unique fixed-point in the space $\mathcal{\tilde C} \left((0,\theta_0], W^{1,1}\left(\R^d \right)\right)$.
\end{itemize}
\end{prop}
Let us deduce Theorem \ref{sobolP} before giving the proof of Proposition \ref{regFixpoint}.
\begin{proof}
$\bullet$ For $\theta = \theta_0$ given by Proposition \ref{regFixpoint}:
\begin{itemize}
  \item[$-$] Let $t \in (0, \theta_0]$. According to Proposition \ref{regFixpoint}, the map $\Phi$ admits a unique fixed-point in $\mathcal{\tilde C}\left((0,\theta_0], W^{1,1}\left(\R^d \right)\right)$. With the inclusion $\mathcal{\tilde C}\left((0,\theta_0], W^{1,1}\left(\R^d \right)\right) \subset \mathcal{\bar C} \left((0,\theta_0], L^1\left(\R^d \right)\right)$, this fixed-point coincides with the unique fixed-point of $\Phi$ in $ \mathcal{\bar C} \left((0,\theta_0], L^1\left(\R^d \right)\right)$ which is $(p(t,.))_{t \in (0,\theta_0]}$ according to Lemma \ref{uniqFP}. Therefore, we have that $(p(t,.))_{t \in (0,\theta_0]} \in \mathcal{\tilde C}\left((0,\theta_0], W^{1,1}\left(\R^d \right)\right)$ and $\sup_{t \in (0,\theta_0]}\sqrt{t}\|p(t,.) \|_{W^{1,1}}\le \max\left(1,\sqrt{\theta_0}\,\right)\opnorm{p} <+\infty$.
  \item[$-$] Now, let $t \in [\theta_0,T]$. Using the fact that the heat kernel is a convolution semi-group and that $$\displaystyle p(\theta_0,.) = G_{\theta_0}*m - \int_0^{\theta_0}\nabla G_{\theta_0 - r}*\Big(b(r,.)p(r,.)\Big)\,dr,$$ we obtain: $$\displaystyle  p(t,.) = G_{t-\theta_0}*p(\theta_0,.) - \int_0^{t-\theta_0}\nabla G_{(t-\theta_0) - s}*\Big(b(\theta_0 + s,.)p(\theta_0 + s,.)\Big)\,ds$$
  such that for $u=(t - \theta_0) \in [0,T-\theta_0]$:
  $$\displaystyle  p(\theta_0+u,.) = G_{u}*p(\theta_0,.) - \int_0^{u}\nabla G_{u - s}*\Big(b(\theta_0 + s,.)p(\theta_0 + s,.)\Big)\,ds.$$
  Hence, by Lemma \ref{uniqFP}, $\left(p(\theta_0+u,.)\right)_{u \in [0,T-\theta_0]}$ is the unique fixed-point in $\mathcal{\bar C}\left((0,T-\theta_0], L^{1}\left(\R^d \right)\right)$ of the functional defined like $\Phi$ but with $m$ replaced by $p(\theta_0,.)$ and $b$ shifted by $\theta_0$ in the time variable. Since $p(\theta_0,.) \in W^{1,1}\left(\R^d\right)$ and according to Proposition \ref{regFixpoint}, this functional admits a unique fixed-point in $\mathcal{C} \left([0,T-\theta_0], W^{1,1}\left(\R^d \right)\right)$ and this fixed-point coincides with $\left(p(\theta_0+u,.)\right)_{u \in [0,T-\theta_0]}$. Therefore, $\left(p(t,.)\right)_{t \in [\theta_0,T]} \in \mathcal{C}\left([\theta_0,T], W^{1,1}\left(\R^d \right)\right)$ and $\sup_{t \in [\theta_0,T]}\sqrt{t}\| p(t,.)\|_{W^{1,1}} \le \sqrt{T}\sup_{t \in [\theta_0,T]}\| p(t,.)\|_{W^{1,1}} <+\infty$.
\end{itemize}
We can conclude.\\
$\bullet$ Now, we assume that $m$ admits a density w.r.t. the Lebesgue measure in $W^{1,1}\left(\R^d\right)$. For $\theta = T$ and according to Proposition \ref{regFixpoint}, the map $\Phi$ admits a unique fixed-point in $\mathcal{C}\left([0,T], W^{1,1}\left(\R^d \right)\right)$. With the inclusion $\mathcal{C}\left([0,T], W^{1,1}\left(\R^d \right)\right) \subset \mathcal{\bar C} \left((0,T], L^1\left(\R^d \right)\right)$, this fixed-point coincides with the unique fixed-point of $\Phi$ in $ \mathcal{\bar C} \left((0,T], L^1\left(\R^d \right)\right)$ which is $(p(t,.))_{t \in (0,T]}$ according to Lemma \ref{uniqFP}. Therefore, $(p(t,.))_{t \in [0,T]} \in \mathcal{C}\left([0,T], W^{1,1}\left(\R^d \right)\right)$.
\end{proof}
To prove Proposition \ref{regFixpoint}, we need the following convolution and derivation in the sense of distributions result.
\begin{lem}\label{Sobol}
  Let $q:\R^d \to \R$ be a function in $W^{1,1}\left(\R^d\right)$ and $g:\R^d \to \R^d$ a bounded measurable function. We assume either that $\|\nabla \cdot g\|_{L^\rho}<+\infty$ for some $\rho \in [d,+\infty]$ or that for $d=1$, $\left\|g'\right\|_{\rm TV}<+\infty$.\\
  $\bullet$ Under the first assumption, for any $\varphi:\R^d \to \R \;$ $\mathcal{C}^{\infty}$-bounded together with its first order derivatives, we have:
  \begin{align*}
  \displaystyle \int_{\R^d} \nabla \varphi(x) \cdot \Big(q(x)g(x)\Big)\,dx &= - \int_{\R^d}\varphi(x) \Big(\nabla q(x) \cdot g(x) + q(x) \nabla \cdot g(x) \Big)\,dx
  \end{align*}
  so that in the sense of distributions, $\nabla \cdot \Big(qg \Big)= q\nabla \cdot g + \nabla q\cdot g$.\\Moreover, for $\displaystyle \check{C} = \sup_{\substack{f \in W^{1,1} \\ f \neq 0 }} \frac{\|f\|_{L^{\frac{\rho}{\rho-1}}(\R^d)}}{\|f\|_{W^{1,1}(\R^d)}}$ which is finite according to Corollary $IX.10$ \cite{Brezis}, we have that:
  \begin{align}\label{normMaj}
    \displaystyle \left\|\nabla \cdot \Big(qg \Big) \right\|_{L^1} &\le \check{C}\left\|q \right\|_{W^{1,1}} \left\| \nabla \cdot g\right\|_{L^{\rho}} + \left\|\nabla q \right\|_{L^1} \left\| g\right\|_{L^{\infty}}. 
  \end{align}
  $\bullet$ Under the second assumption, for any $\varphi:\R \to \R \;$ $\mathcal{C}^{\infty}$-bounded together with its first order derivative, we have:
  \begin{align*}
  \displaystyle \int_{\R}  \varphi'(x)\Big(q(x)g(x)\Big)\,dx &= - \int_{\R}\varphi(x)q'(x)g(x)\,dx - \int_{\R}\varphi(x)q(x)g'(dx)
  \end{align*}
  where the continuous representative of $q$ which exists according to Theorem $VIII.2$ \cite{Brezis}, is chosen to define $q(x)g'(dx)$. Moreover, the derivative of $qg$ in the sense of distributions is a bounded measure on the real line and for $\displaystyle \check{C} = \sup_{\substack{f \in W^{1,1} \\ f \neq 0 }} \frac{\|f\|_{\infty}}{\|f\|_{W^{1,1}(\R)}}$ which is finite according to Theorem $VIII.7$ \cite{Brezis}, we have that:
  \begin{align}\label{normMaj1}
    \displaystyle \left\|(qg)' \right\|_{\rm TV} &\le \check{C} \left\|q \right\|_{W^{1,1}} \left\| g'\right\|_{\rm TV} + \left\|q' \right\|_{L^1} \left\| g\right\|_{\infty}.
  \end{align}
\end{lem}
\begin{proof}
According to Theorem $IX.2$ \cite{Brezis}, there exists a sequence of functions $\left(q_n\right)_{n}$ in $\mathcal{C}^{\infty}_c\left(\R^d\right)$ such that  when $n \to +\infty$, $q_n \to q$ and $\nabla q_n \to \nabla q$ respectively in $L^1\left(\R^d\right)$ and $L^1\left(\R^d\right)^d$. We have:
\begin{align}\label{convLim}
  \displaystyle \int_{\R^d} \nabla\varphi(x)\cdot \Big(q_n(x)g(x)\Big)\,dx &= \int_{\R^d}\nabla\Big(\varphi(x)q_n(x)\Big)\cdot g(x)\,dx - \int_{\R^d}\varphi(x)\nabla q_n(x)\cdot g(x)\,dx.
\end{align}
Since $\nabla \varphi$, $g$ and $\varphi$ are bounded functions on $\R^d$, we have $\displaystyle \int_{\R^d} \nabla\varphi(x)\cdot \Big(q_n(x)g(x)\Big)\,dx \underset{n \to +\infty}{\longrightarrow} \int_{\R^d} \nabla\varphi(x)\cdot \Big(q(x)g(x)\Big)\,dx$ and 
$\displaystyle \int_{\R^d} \varphi(x)\nabla q_n(y)\cdot g(x)\,dx \underset{n \to +\infty}{\longrightarrow} \int_{\R^d} \varphi(x)\nabla q(x)\cdot g(x)\,dx$.\\
$\bullet$ Under the first assumption, since $\varphi q_n \in C^{\infty}_c\left(\R^d\right)$, we have that $\displaystyle \int_{\R^d}\nabla\Big(\varphi(x)q_n(x)\Big)\cdot g(x)\,dx = - \int_{\R^d}\varphi(x)q_n(x)\nabla \cdot g(x)\,dx$. Using H\"older's inequality, we obtain that:
\begin{align*}
  \displaystyle \left|\int_{\R^d}\varphi(x)q_n(x)\nabla \cdot g(x)\,dx -  \int_{\R^d}\varphi(x)q(x)\nabla \cdot g(x)\,dx\right| &\le \sup_{x \in \R^d}|\varphi(x)| \left\| q_n - q\right\|_{L^{\frac{\rho}{\rho-1}}} \left\| \nabla \cdot g\right\|_{L^{\rho}} \\
  &\le \check{C} \sup_{x \in \R^d}|\varphi(x)| \left\| q_n - q\right\|_{W^{1,1}} \left\| \nabla \cdot g\right\|_{L^{\rho}}\\
  &\underset{n \to +\infty}{\longrightarrow} 0.
\end{align*}
Hence, taking the limit $n \to +\infty$ in Equation \eqref{convLim}, we get, in the sense of distributions, that $\nabla \cdot \left(qg\right) = q\nabla \cdot g + \nabla q \cdot g$. Now, using once again H\"older's inequality and Corollary $IX.10$ \cite{Brezis}, one has $$\left\| q \nabla \cdot g\right\|_{L^1} \le \left\|q\right\|_{L^{\frac{\rho}{\rho-1}}} \left\| \nabla \cdot g\right\|_{L^{\rho}} \le \check{C} \left\| q\right\|_{W^{1,1}} \left\| \nabla \cdot g\right\|_{L^{\rho}} $$ and one deduces that $$\left\|\nabla \cdot \Big(qg \Big) \right\|_{L^1} \le \check{C} \left\|q \right\|_{W^{1,1}} \left\| \nabla \cdot g\right\|_{L^{\rho}} + \left\|\nabla q \right\|_{L^1} \left\| g\right\|_{L^{\infty}}.$$
$\bullet$ Under the second assumption, since $\varphi q_n \in \mathcal{C}^{\infty}_c(\R)$, we have that $\displaystyle \int_{\R}\Big(\varphi q_n\Big)'(x)g(x)\,dx = - \int_{\R} \varphi(x)q_n(x)g'(dx)$. According to Theorem $VIII.2$ and Theorem $VIII.7$ \cite{Brezis}, $q$ admits a bounded and continuous representative, and for this representative the integral $\displaystyle \int_{\R}\varphi(x)q(x) g'(dx)$ makes sense. Using H\"older's inequality, we have:
\begin{align*}
  \displaystyle \left|\int_{\R}\varphi(x)q_n(x)g'(dx) -  \int_{\R}\varphi(x)q(x)g'(dx)\right| &\le \sup_{x \in \R}|\varphi(x)| \sup_{x \in \R}\left| q_n(x) - q(x) \right| \left\| g'\right\|_{\rm TV} \\
  &\le \check{C} \sup_{x \in \R}|\varphi(x)| \left\| q_n - q\right\|_{W^{1,1}} \left\| g'\right\|_{\rm TV}\\
  &\underset{n \to +\infty}{\longrightarrow} 0.
\end{align*}
Hence, taking the limit $n \to +\infty$ in Equation \eqref{convLim}, we get, in the sense of distributions, that $\left(qg\right)' = q g' + q'g$. Now, using once again H\"older's inequality and Theorem $VIII.7$ \cite{Brezis}, one has $$\left\| qg'\right\|_{\rm TV} \le \sup_{x \in \R}\left|q(x) \right| \left\| g'\right\|_{\rm TV} \le \check{C} \left\|q\right\|_{W^{1,1}} \left\| g'\right\|_{\rm TV}  $$ and one deduces that $$\left\|(qg)' \right\|_{\rm TV} \le \check{C} \left\|q \right\|_{W^{1,1}} \left\| g'\right\|_{\rm TV} + \left\|q' \right\|_{L^1} \left\| g\right\|_{\infty}.$$
\end{proof}
We are now ready to prove Proposition \ref{regFixpoint}.
\begin{proof}
We are going to suppose that $\,\sup_{t \in [0,T]} \|\nabla \cdot b(t,.)\|_{L^\rho}<+\infty$ for some $\rho \in [d,+\infty]$. When $d=1$, $\sup_{t \in [0,T]}\left\|\partial_x b(t,.)\right\|_{\rm TV}<+\infty$, the proof is analogous and the estimations remain valid when replacing $\left\|\nabla \cdot b(t,.) \right\|_{L^{\rho}}$ by $\left\|\partial_x b(t,.)\right\|_{\rm TV}$.\\

For $\theta \in (0,T]$, let $ 0  \le s < t \le \theta $. If $q(s,.)$ is in $W^{1,1}\left(\R^d\right)$, we apply Lemma \ref{Sobol} with $\varphi = G_{t-s}$ which is $\mathcal{C}^{\infty}$-bounded together with its first order derivatives and $g = b(s,.)$ to obtain that:
\begin{align*}
  \nabla G_{t-s} * \Big(b(s,.)q(s,.)\Big) = G_{t-s} * \nabla \cdot \Big(b(s,.)q(s,.)\Big).
\end{align*}
If $q \in \mathcal{\tilde C}\left( (0,\theta],W^{1,1}\left(\R^d\right)\right)$, we first have:
\begin{align}\label{nablaq}
  \forall t \in (0,\theta], \quad \Phi_t(q) = G_t * m - \int_0^t G_{t-s} * \nabla \cdot \left(b(s,.)q(s,.)\right)\,ds.
\end{align}

 $\bullet$ We start by proving that there exists $\theta_0$ s.t. $\Phi$ admits a unique fixed-point in $\mathcal{\tilde C}\left((0,\theta_0],W^{1,1}\left(\R^d\right)\right)$:\\

 Let $t \in (0,\theta]$ and $q \in \mathcal{\tilde C}\left((0,\theta],W^{1,1}\left(\R^d\right)\right)$. We have, using the estimate \eqref{FirstDerivG} from Lemma \ref{EstimHeatEq} and Inequality \eqref{normMaj} from Lemma \ref{Sobol}, that:
\begin{align*}
  \displaystyle &\int_{0}^t \left\|\nabla G_{t-s} * \nabla \cdot \Big(b(s,.)q(s,.)\Big)\right\|_{L^1}\,ds \\
  &\le d\sqrt{\frac 2 \pi}\int_0^t \frac{1}{\sqrt{t-s}} \left(B + \check{C} \sup_{u\in [0,T]}\|\nabla \cdot b(u,.) \|_{L^{\rho}} \right) \Big(\|q(s,.)\|_{L^1} + \|\nabla q(s,.)\|_{L^1} \Big)\,ds \\
  &\le d\sqrt{2 \pi}\max(1,\sqrt{T})\left(B + \check{C} \sup_{u\in [0,T]}\|\nabla \cdot b(u,.) \|_{L^{\rho}}\right)\left(\sup_{u \in (0,\theta]}\|q(u,.)\|_{L^1} + \sup_{u \in (0,\theta]}\sqrt u\|\nabla q(u,.)\|_{L^1} \right).
\end{align*}
Therefore, $\displaystyle \int_{0}^t \left\|\nabla G_{t-s} * \nabla \cdot \Big(b(s,.)q(s,.)\Big)\right\|_{L^1}\,ds \le d\sqrt{2 \pi}\max(1,\sqrt{T})\left(B + \check{C} \sup_{u\in [0,T]}\|\nabla \cdot b(u,.) \|_{L^{\rho}}\right)\opnorm{q}$ which is finite. We can then apply Fubini's theorem and obtain that, in the sense of distributions, the gradient of $\Phi_t(q)$ defined in \eqref{nablaq}, is equal to:
\begin{align}\label{gradPhi}
  \displaystyle \nabla \Phi_t(q) = \nabla G_t* m - \int_0^t \nabla G_{t-s}*\nabla \cdot \Big(b(s,.)q(s,.) \Big)\,ds.
\end{align}
We can now estimate $\opnorm{\Phi(q)}$. Using the same arguments as before, we have that: $$\displaystyle \opnorm{G*m} = \sup_{t \in [0,\theta]}\|G_t*m \|_{L^1} + \sup_{t \in [0,\theta]}\sqrt{t}\sum_{i=1}^d\|\partial_{x_i} G_t*m\|_{L^1} \le 1 + d\sqrt{\frac 2 \pi} <+\infty,$$
and that:
\begin{align*}
  \displaystyle &\|\Phi_t(q)\|_{L^1} + \sqrt{t}\|\nabla \Phi_t(q)\|_{L^1} \\
    &\le \left(1 + d\sqrt{\frac 2 \pi}\right) + d\sqrt{\frac 2 \pi}\int_0^t \frac{1}{\sqrt{t-s}} \left\{ B\|q(s,.)\|_{L^1} + \sqrt{t} \Big(\check{C} \sup_{u\in [0,T]}\|\nabla \cdot b(u,.) \|_{L^{\rho}}\|q(s,.)\|_{W^{1,1}} + B\|\nabla q(s,.)\|_{L^1}\Big) \right\}\,ds \\
    &\le \left(1 + d\sqrt{\frac 2 \pi}\right) + d\sqrt{\frac 2 \pi}\int_0^t \frac{1}{\sqrt{t-s}} \bigg\{ \left(B + \sqrt{T}\check{C} \sup_{u\in [0,T]}\|\nabla \cdot b(u,.) \|_{L^{\rho}} \right)\|q(s,.)\|_{L^1} \\
    &\phantom{\left(1 + d\sqrt{\frac 2 \pi}\right) + d\sqrt{\frac 2 \pi}\int_0^t \frac{1}{\sqrt{t-s}} \bigg\{\Big(B + \sqrt{T}}+ \sqrt{\theta}\left(B + \check{C} \sup_{u\in [0,T]}\|\nabla \cdot b(u,.) \|_{L^{\rho}} \right)\|\nabla q(s,.)\|_{L^1}  \bigg\}\,ds 
\end{align*}
\begin{align*}   
&\le \left(1 + d\sqrt{\frac 2 \pi}\right) + \sqrt{\theta}d\sqrt{\frac 2 \pi}\Bigg\{ 2\left(B + \sqrt{T}\check{C} \sup_{u\in [0,T]}\|\nabla \cdot b(u,.) \|_{L^{\rho}} \right)\sup_{u\in (0,\theta]}\|q(u,.)\|_{L^1} \\
&\phantom{\left(1 + d\sqrt{\frac 2 \pi}\right) + \sqrt{\theta}\bigg\{ 2d\sqrt{\frac 2 \pi}\Big(B + \sqrt{T}\check{C} }+ \pi\left(B + \check{C} \sup_{u\in [0,T]}\|\nabla \cdot b(u,.) \|_{L^{\rho}} \right) \sup_{u\in (0,\theta]}\sqrt{u}\|\nabla q(u,.)\|_{L^1}  \Bigg\} \\ 
&= \left(1 + d\sqrt{\frac 2 \pi}\right) + \sqrt{\theta}d\sqrt{\frac 2 \pi}\Bigg\{ B\left(2\sup_{u\in (0,\theta]}\|q(u,.)\|_{L^1} + \pi \sup_{u\in (0,\theta]}\sqrt{u}\|\nabla q(u,.)\|_{L^1}  \right) \\
    &\phantom{\left(1 + d\sqrt{\frac 2 \pi}\right) + \sqrt{\theta}\bigg\{ 2d\sqrt{\frac 2 \pi}\Big(B + }+ \check{C} \sup_{u\in [0,T]}\|\nabla \cdot b(u,.) \|_{L^{\rho}}\left(2\sqrt{T}\sup_{u\in (0,\theta]}\|q(u,.)\|_{L^1} + \pi \sup_{u\in (0,\theta]}\sqrt{u}\|\nabla q(u,.)\|_{L^1}  \right)  \Bigg\}\\
    &\le \left(1 + d\sqrt{\frac 2 \pi}\right) + d\sqrt{2 \pi \theta}\left(B + \max(1,\sqrt{T})\check{C} \sup_{u\in [0,T]}\|\nabla \cdot b(u,.) \|_{L^{\rho}} \right) \left(\sup_{u\in (0,\theta]}\|q(u,.)\|_{L^1} + \sup_{u\in (0,\theta]}\sqrt{u}\|\nabla q(u,.)\|_{L^1} \right).
\end{align*} 
Hence, $\displaystyle \opnorm{\Phi(q)} \le \left(1 + d\sqrt{\frac 2 \pi}\right) + d\sqrt{2 \pi \theta}\left(B + \max(1,\sqrt{T})\check{C} \sup_{u\in [0,T]}\|\nabla \cdot b(u,.) \|_{L^{\rho}} \right)\opnorm{q}$ and $\opnorm{\Phi(q)} <+\infty$.
Now, let $0 <r\le s\le \theta$, adapting the proof of Lemma \ref{evoldens} and using Inequality \eqref{normMaj} from Lemma \ref{Sobol}, we obtain that:
\begin{align*}
  \displaystyle &\left\|\nabla \Phi_s(q) - \nabla \Phi_r(q) \right\|_{L^1} \\
  &\le d(2d+3)\sqrt{\frac 2 \pi}\left(\frac{1}{\sqrt{r}} - \frac{1}{\sqrt s} \right) + 4d(d+2)\sqrt{\frac 2 \pi}\left(B + \check{C}\sup_{u \in [0,T]}\|\nabla \cdot b(u,.) \|_{L^{\rho}} \right)\sup_{u \in [r,s]}\|q(u,.)\|_{W^{1,1}}\sqrt{s-r}.
\end{align*}
Since $\opnorm{\Phi(q)} = \sup_{u \in (0,\theta]}\|q(u,.)\|_{L^1} + \sup_{u \in (0,\theta]}\sqrt{u}\|\nabla q(u,.)\|_{L^1}$ and using Inequality \eqref{timeCtyPhi}, we can conclude that $t\mapsto \Phi_t(q)$ is continuous on $(0,\theta]$ with values in $W^{1,1}\left(\R^d\right)$. Hence, $\Phi(q) \in \mathcal{\tilde C}\left((0,\theta],W^{1,1}\left(\R^d\right)\right)$. \\
Now, let $q,\tilde q \in \mathcal{\tilde C} \left((0,\theta], W^{1,1}\left(\R^d \right)\right)$, we obtain, with the same reasoning above, that $$\displaystyle \opnorm{\Phi(q) - \Phi(\tilde q)} \le d\sqrt{2 \pi \theta}\left(B + \max(1,\sqrt{T})\check{C} \sup_{u\in [0,T]}\|\nabla \cdot b(u,.) \|_{L^{\rho}} \right)\opnorm{q-\tilde q}.$$
If: 
$$\displaystyle \theta < \frac{1}{2\pi d^2 \left(B + \max(1,\sqrt{T})\check{C} \sup_{u\in [0,T]}\|\nabla \cdot b(u,.) \|_{L^{\rho}} \right)^2} =: 2\theta_0$$ 
then the map $\Phi$ is a contraction on the space $\mathcal{\tilde C} \left((0,\theta_0], W^{1,1}\left(\R^d \right) \right)$. Using Picard's Theorem, the map $\Phi$ admits then a unique fixed-point on the space $\mathcal{\tilde C} \left((0,\theta_0], W^{1,1}\left(\R^d \right)\right)$.\\

$\bullet$ Now, we assume that $m$ admits a density w.r.t. the Lebesgue measure in $W^{1,1}\left(\R^d\right)$. Let us prove that the map $\Phi$ admits a unique-fixed point in $\mathcal{C}\left([0,\theta],W^{1,1}\left(\R^d\right)\right)$ for all $\theta \in (0,T]$:\\

 Let $t \in (0,\theta]$ and $q \in \mathcal{C}\left([0,\theta],W^{1,1}\left(\R^d\right)\right)$ for all $\theta \in (0,T]$. Using Equation \eqref{gradPhi} and the fact that $m$ admits a density w.r.t. the Lebesgue measure in $W^{1,1}\left(\R^d\right)$, we obtain:
\begin{align*}
  \displaystyle \nabla \Phi_t(q) =  G_t* \nabla m - \int_0^t \nabla G_{t-s}*\nabla \cdot \Big(b(s,.)q(s,.) \Big)\,ds.
\end{align*}
Using the estimate \eqref{FirstDerivG} from Lemma \ref{EstimHeatEq} and Inequality \eqref{normMaj} from Lemma \ref{Sobol}, we have:
\begin{align*}
\displaystyle &\left\|\Phi_t(q)\right\|_{W^{1,1}} \le \left\|G_t*m\right\|_{W^{1,1}} + d\sqrt{\frac 2 \pi}\int_0^t \frac{1}{\sqrt{t-s}} \left(B\left\|q(s,.)\right\|_{L^1} + \left\|\nabla. \Big(b(s,.)q(s,.)\Big)\right\|_{L^1} \right)\,ds \\
&\le \left(1 + \sum_{i=1}^d\left\|\partial_{x_i} m\right\|_{L^1} \right) + d\sqrt{\frac 2 \pi}\int_0^t \frac{1}{\sqrt{t-s}} \Big(B\left\|q(s,.)\right\|_{L^1} + \check{C}\|\nabla \cdot b(s,.) \|_{L^{\rho}}\left\|q(s,.)\right\|_{W^{1,1}} + B \left\|\nabla q(s,.) \right\|_{L^1} \Big)\,ds
\end{align*}

\begin{align*}
  &\le \left(1 + \sum_{i=1}^d\left\|\partial_{x_i} m\right\|_{L^1} \right) + d\sqrt{\frac 2 \pi}\int_0^t \frac{1}{\sqrt{t-s}} \Big(B\left\|q(s,.)\right\|_{L^1} + \check{C}\|\nabla \cdot b(s,.) \|_{L^{\rho}}\left\|q(s,.)\right\|_{W^{1,1}} + B \left\|\nabla q(s,.) \right\|_{L^1} \Big)\,ds\\
  &\le \left(1 + \sum_{i=1}^d\left\|\partial_{x_i} m\right\|_{L^1} \right) + d\sqrt{\frac 2 \pi} \left(B + \check{C}\sup_{u \in [0,T]}\|\nabla \cdot b(u,.) \|_{L^{\rho}} \right)\int_0^t \frac{1}{\sqrt{t-s}}\left\|q(s,.)\right\|_{W^{1,1}}\,ds \\
  &\le \left(1 + \sum_{i=1}^d\left\|\partial_{x_i} m\right\|_{L^1} \right) + 2d\sqrt{\frac{ 2 t}{\pi}} \left(B + \check{C}\sup_{u \in [0,T]}\|\nabla \cdot b(u,.) \|_{L^{\rho}} \right)\sup_{u \in [0,\theta]}\left\|q(u,.)\right\|_{W^{1,1}}.
\end{align*}
Hence, $\displaystyle \sup_{t \in [0,\theta]}\left\|\Phi_t(q)\right\|_{W^{1,1}} \le \left(1 + \sum_{i=1}^d\left\|\partial_{x_i} m\right\|_{L^1} \right) + 2d\sqrt{\frac{ 2 \theta}{\pi}} \left(B + \check{C}\sup_{u \in [0,T]}\|\nabla \cdot b(u,.) \|_{L^{\rho}} \right)\sup_{t \in [0,\theta]}\left\|q(t,.)\right\|_{W^{1,1}}$ is finite.
Now, concerning the continuity of $t \mapsto \Phi_t(q)$ on $[0,\theta]$ with values in $W^{1,1}\left(\R^d\right)$, we already proved above the continuity on $(0,\theta]$. As for the continuity at $t=0$, we denote by $\tau_y w$ the translation of $w \in L^1\left(\R^d\right)$ by $y \in \R^d$ defined by $\tau_y w(x) = w(x -y)$ for $x \in \R^d$. We have:
 \begin{align*}
  \displaystyle \left\|G_t*w - w\right\|_{L^1} &= \int_{\R^d}\left|\int_{\R^d}G_t(y)w(x-y)\,dy - w(x)\right|\,dx = \int_{\R^d}\left|\int_{\R^d}G_1(y)w(x-\sqrt{t}y)\,dy - w(x)\right|\,dx \\
  &\le \int_{\R^d}G_1(y)\left\|\tau_{\sqrt{t}y}w-w\right\|_{L^1}\,dy.
 \end{align*}
Using Lemma $IV.4$ \cite{Brezis}, we have that $\left\|\tau_{\sqrt{t}y}w-w\right\|_{L^1} \to 0$ when $t \to 0$ and since $\left\|\tau_{\sqrt{t}y}w-w\right\|_{L^1} \le 2 \left\|w\right\|_{L^1}$, by dominated convergence, we obtain that $G_t*w \to w$ when $t \to 0$ in $L^1\left(\R^d\right)$. We replace $w$ by $m$ and $\nabla m$ since $m$ admits a density in $W^{1,1}\left(\R^d\right)$ and conclude that $\left\|G_t*m -m \right\|_{W^{1,1}} \to 0$ when $t \to 0$. Moreover, we have that:
\begin{align*}
  \displaystyle \left\|\int_0^t G_{t-s} * \nabla \cdot \Big(b(s,.)q(s,.)\Big)\,ds \right\|_{W^{1,1}} &\le \left(B + \check{C}\sup_{u \in [0,T]}\|\nabla \cdot b(u,.) \|_{L^{\rho}} \right) \sup_{u \in [0,\theta]}\left\|q(u,.)\right\|_{W^{1,1}}\, \left(t + d\sqrt{\frac 2 \pi}\sqrt{t} \right) 
\end{align*}
which converges to $0$ when $t \to 0$. Hence, $\displaystyle \left\|\Phi_t(q) - m \right\|_{W^{1,1}}$ converges to $0$ when $t \to 0$ and $\Phi \in \mathcal{C}\left([0,\theta],W^{1,1}\left(\R^d\right)\right)$.\\
Now, let $q,\tilde q \in \mathcal{C} \left([0,\theta], W^{1,1}\left(\R^d \right)\right)$. With the same reasoning above, we obtain that:
\begin{align}\label{machin}
    \displaystyle \left\|\Phi_t(q) - \Phi_t(\tilde q)\right\|_{W^{1,1}} &\le d\sqrt{\frac 2 \pi} \left(B + \check{C}\sup_{u \in [0,T]}\|\nabla \cdot b(u,.) \|_{L^{\rho}} \right)\int_0^t \frac{1}{\sqrt{t-s}}\left\|q(s,.) - \tilde q(s,.)\right\|_{W^{1,1}}\,ds.
\end{align} 
As done in the proof of Lemma \ref{uniqFP}, for $n \in \N^{*}$, we iterate Inequality \eqref{machin} $2n$-times and deduce that for $n$ big enough, $\Phi^{2n}$ is a contraction on $\mathcal{C}\left([0,\theta],W^{1,1}\left(\R^d\right)\right)$. We conclude, through Picard's Theorem, the existence of a unique fixed-point of the map $\Phi$ on $\mathcal{C}\left([0,\theta],W^{1,1}\left(\R^d\right)\right)$.
\end{proof}
We are now ready to prove Proposition \ref{hyp}.
\begin{proof}
We are going to suppose that $\,\sup_{t \in [0,T]} \|\nabla \cdot b(t,.)\|_{L^\rho}<+\infty$ for some $\rho \in [d,+\infty]$. When $d=1$, $\sup_{t \in [0,T]}\left\|\partial_x b(t,.)\right\|_{\rm TV}<+\infty$, the proof is analogous and the estimations remain valid when replacing $\left\|\nabla \cdot b(t,.) \right\|_{L^{\rho}}$ by $\left\|\partial_x b(t,.)\right\|_{\rm TV}$.\\

The proof is an immediate consequence of Theorem \ref{sobolP} and Lemma \ref{Sobol}. Indeed, assuming the regularity on $b$ w.r.t. the space variables, we have from Theorem \ref{sobolP} that $\left(p(t,.) \right)_{t \in (0,T]} \in \mathcal{\tilde C}\left((0,T],W^{1,1}\left(\R^d\right)\right)$. We can then apply Lemma \ref{Sobol} to obtain, for $t \in (0,T]$, that:
\begin{align*}
  \displaystyle \sqrt{t}\left\|\nabla \cdot \Big(b(t,.)p(t,.) \Big) \right\|_{\rm TV} &\le \left( B+ \check{C}\sup_{u \in [0,T]}\|\nabla \cdot b(u,.) \|_{L^{\rho}} \right)\sqrt{t} \left\| \nabla p(t,.) \right\|_{L^1} + \check{C} \sqrt{T}\sup_{u \in [0,T]}\|\nabla \cdot b(u,.) \|_{L^{\rho}}\left\|p(t,.) \right\|_{L^1} \\
  &\le \max\left(B+ \check{C}\sup_{u \in [0,T]}\|\nabla \cdot b(u,.) \|_{L^{\rho}}, \check{C} \sqrt{T}\sup_{u \in [0,T]}\|\nabla \cdot b(u,.) \|_{L^{\rho}}\right)\opnorm{p} <+\infty.
\end{align*} 
The conclusion holds with $M = \max\left(B+ \check{C}\sup_{u \in [0,T]}\|\nabla \cdot b(u,.) \|_{L^{\rho}}, \check{C} \sqrt{T}\sup_{u \in [0,T]}\|\nabla \cdot b(u,.) \|_{L^{\rho}}\right)\opnorm{p}$. Since $p \in \mathcal{\tilde C}\left((0,T],W^{1,1}\left(\R^d\right)\right)$, using Lemma \ref{Sobol}, we have as in Equality \eqref{nablaq} that:
$$ \displaystyle \forall t \in (0,T], \quad p(t,.) = G_t * m - \int_0^t G_{t-s}*\nabla \cdot \Big(b(s,.)p(s,.)\Big)\,ds.$$
Moreover, when $m$ admits a density w.r.t. the Lebesgue measure in $W^{1,1}\left(\R^d\right)$, according to Theorem \ref{sobolP}, $\left(p(t,.) \right)_{t \in [0,T]} \in \mathcal{C}\left([0,T],W^{1,1}\left(\R^d\right)\right)$. Therefore, using once again Lemma \ref{Sobol}, we obtain for $t \in [0,T]$ that:
\begin{align*}
  \displaystyle \left\|\nabla \cdot \Big(b(t,.)p(t,.) \Big) \right\|_{\rm TV} &\le \left( B+ \check{C}\sup_{u \in [0,T]}\|\nabla \cdot b(u,.) \|_{L^{\rho}} \right)\sup_{u \in [0,T]}\left\| p(u,.) \right\|_{W^{1,1}} <+\infty.
\end{align*} 
The conclusion holds with $\tilde M =  \left( B+ \check{C}\sup_{u \in [0,T]}\|\nabla \cdot b(u,.) \|_{L^{\rho}} \right)\sup_{u \in [0,T]}\left\| p(u,.) \right\|_{W^{1,1}}$.
\end{proof}


\subsection{Proof of Theorem \ref{cvgenceB}}\label{demoo}

We first use Inequality \eqref{hyp1} from Proposition \ref{hyp} to obtain a stronger regularity of $p(t,.)$ with respect to the time variable. 
\begin{lem}\label{evoldensB}
Assume Inequality \eqref{hyp1}. We have:
   $$\displaystyle \exists \, \tilde Q<+\infty, \forall \, 0 < r\le s\le T,\quad \left\|p(s,.) - p(r,.) \right\|_{L^1} \le \tilde Q\bigg( \ln(s/r) + \frac{s-r}{2\sqrt r} \ln\left(\frac{4s}{s-r} \right) +\left(\sqrt{s} - \sqrt{r}\right) \bigg).$$
\end{lem}
\begin{proof}
  We will adapt the proof of Lemma \ref{evoldens}. Let $0 < r \le s \le T$. Using Equality \eqref{nablaP} from Proposition \ref{hyp}, we have that:
  \begin{align}\label{diffP}
  \displaystyle p(s,.) - p(r,.) = \left( G_s - G_r \right) * m - \int_0^s  G_{s-u} * \nabla \cdot \Big(b(u,.)p(u,.)\Big)\,du + \int_0^r  G_{r-u} * \nabla \cdot \Big(b(u,.)p(u,.)\Big)\,du.
  \end{align}
  Therefore, using the estimates \eqref{timeDerivG} and \eqref{FirstDerivG} from Lemma \ref{EstimHeatEq}, the fact that $\ln(1+x) \le x, \; \forall x >0$ and Lemma \ref{intg}, we obtain:

  \begin{align*}
  \displaystyle &\left\|p(s,.) - p(r,.) \right\|_{L^1} \\
  &\le \left\|\left( G_s - G_r \right) * m \right\|_{L^1} + \left\|\int_0^r \left( G_{s-u} - G_{r-u} \right)*\nabla \cdot \Big(b(u,.)p(u,.)\Big)\,du \right\|_{L^1} + \left\|\int_r^s G_{s-u} * \nabla \cdot \Big(b(u,.)p(u,.)\Big)\,du \right\|_{L^1} \\
  &\le \int_r^s \left\|\partial_u G_u \right\|_{L^1} \,du +\int_0^s \int_{r-u}^{s-u} \left\|\partial_{\theta} G_{\theta} \right\|_{L^1} \left\|\nabla \cdot \Big(b(u,.)p(u,.) \Big) \right\|_{\rm TV}\,d\theta\,du  + M \int_r^s \frac{du}{\sqrt u}\\
  &\le d\ln(s/r) + dM \int_0^r \int_{r-u}^{s-u} \frac{d\theta}{\theta \sqrt{u}}\,du  + 2M \left(\sqrt{s} - \sqrt{r} \right) \\
  &= d\ln(s/r) + 2dM \left(\frac{s-r}{\sqrt s + \sqrt r} \ln\left(\frac{\left(\sqrt s + \sqrt r \right)^2}{s-r} \right) + 2\sqrt r \ln\left(1 + \frac{\sqrt s - \sqrt r}{2 \sqrt r} \right) \right) + 2M \left(\sqrt{s} - \sqrt{r}\right)\\
  &\le d\ln(s/r) + dM \; \frac{s-r}{\sqrt r} \ln\left(\frac{4s}{s-r} \right) + 2M(1+d)\left(\sqrt{s} - \sqrt{r}\right).
  \end{align*}
  The conclusion holds with $\tilde Q = \max\left(d,2M(1+d)\right)$.
\end{proof}
We are now ready to prove Theorem \ref{cvgenceB}. Once again, using Equality \eqref{caract} and Proposition \ref{propspde}, to prove the theorem amounts to prove that:
$$ \exists \, \tilde{C}<+\infty, \forall h \in (0,T], \, \forall k \in \left\llbracket 1, \left\lfloor\frac{T}{h}\right\rfloor \right\rrbracket, \quad \left\| p(kh,.) - p^h(kh,.)\right\|_{L^1} \le \frac{\tilde C}{\sqrt{kh}}\left(1 + \ln\left(k\right) \right)h. $$
For $h \in (0,T], \, k  \in \left\llbracket 1, \left\lfloor\frac{T}{h}\right\rfloor \right\rrbracket$, we recall Inequality \eqref{firstinteg}:
\begin{align*}
  \displaystyle \Big\|p(kh,.)-p^h\left(kh,.\right)\Big\|_{L^1} &\le \left\|V^h_1(k,.)\right\|_{L^1}+ \left\|V^h_2(k,.)\right\|_{L^1} + \sqrt{\frac 2 \pi}dB\sum_{j=1}^{k-1} \frac{h}{\sqrt{kh - jh}}\left\|p(jh,.) - p^h(jh,.) \right\|_{L^1} \\
  &+ 2dB^2 \left(1 + \frac{d-1}{\pi}\right) \left(\frac{1}{2} + \ln\left(k \right) \right)h + \sqrt{\frac{2}{\pi}} dB  \frac{3h}{\sqrt{kh}}.
\end{align*}
Let us estimate $\left\|V^h_1(k,.)\right\|_{L^1}$ and $\left\|V^h_2(k,.)\right\|_{L^1}$ for $k \ge 2$ by taking advantage of the additional regularity of $b$ and using Equality \eqref{nablaP} from Proposition \ref{hypH}:\\ \\
$\bullet$ We use Lemma \ref{Sobol} and the additional regularity of $b$ to transfer the gradient from $G$ to $bp$ and rewrite $V_1^h(k,.)$ as: 

\begin{align}\label{v1}
  \displaystyle &V_1^h(k,.) = \int_{h}^{kh}\bigg(G_{kh-s}- G_{kh-\tau^h_s} \bigg)* \nabla \cdot \Big(b(s,.)p(s,.) \Big)\,ds \nonumber \\
  &= \int_h^{(k-1)h}\left(\int_{kh-s}^{kh-\tau^h_s}\partial_u G_u \,du\right)*\nabla \cdot \Big(b(s,.)p(s,.) \Big)\,ds + \int_{(k-1)h}^{kh}\bigg(G_{kh-s}- G_{kh-\tau^h_s} \bigg)* \nabla \cdot \Big(b(s,.)p(s,.) \Big)\,ds.
\end{align}
Therefore, using Estimate \eqref{timeDerivG} from Lemma \ref{EstimHeatEq}, Inequality \eqref{hyp1} from Proposition \ref{hyp}, Inequality \eqref{integLog} and Lemma \ref{integ}, we obtain:
\begin{align*}
\displaystyle \left\|V_1^h(k,.)\right\|_{L^1} &\le dM \int_{h}^{(k-1)h}\frac{1}{\sqrt s} \ln\left(1 + \frac{s-\tau^h_s}{kh -s}\right)\,ds + 2M \int_{(k-1)h}^{kh}\frac{ds}{\sqrt s}\\
&\le dM \int_{h}^{(k-1)h} \frac{h}{(kh-s)\sqrt s}\,ds + 4M \left(\sqrt{kh}-\sqrt{(k-1)h} \right) \\
&\le dM\frac{h}{\sqrt{kh}}\ln\left(4k\right) + 4M \frac{h}{\sqrt{kh} + \sqrt{(k-1)h}} \\
&\le \frac{2M\left( 2 + d\ln(2)\right)}{\sqrt{kh}}\left(1 + \ln\left(k\right)\right)h.
\end{align*}
$\bullet$ Using Lemma \ref{evoldensB} and the estimate \eqref{FirstDerivG} from Lemma \ref{EstimHeatEq}, we have:
\begin{align*}
\displaystyle \left\|V_2^h(k,.)  \right\|_{L^1} &\le \sqrt{\frac 2 \pi}dB\tilde Q \Bigg\{ \int_{h}^{kh} \frac{\ln(s/\tau^h_s)}{\sqrt{kh - \tau^h_s}}\,ds + \int_{h}^{kh} \frac{s-\tau^h_s}{2\sqrt{\tau^h_s}\sqrt{kh - \tau^h_s}} \ln\left(\frac{4s}{s-\tau^h_s} \right)\,ds + \int_{h}^{kh}\frac{\sqrt{s} - \sqrt{\tau^h_s}}{\sqrt{kh - \tau^h_s}} \,ds \Bigg\}.
\end{align*}
We use Inequality \eqref{integLog}, the fact that $\displaystyle \sup_{0<x\le h}\left( x\ln\left(\frac{4kh}{x}\right)\right) $ is attained for $x=h$ since $h \le kh$ and Lemma \ref{sumPi} to obtain:
\begin{align*}
\displaystyle \left\|V_2^h(k,.)  \right\|_{L^1} &\le \sqrt{\frac 2 \pi} dB\tilde Q\left\{ \frac{2h}{\sqrt{kh}}\ln(4k) + h \ln\left( 4k\right)\int_{h}^{kh} \frac{ds}{2\sqrt{\tau^h_s}\sqrt{kh - \tau^h_s}} + \int_{h}^{kh} \frac{s -\tau^h_s}{\left(\sqrt{s}+\sqrt{\tau^h_s} \right)\sqrt{kh - \tau^h_s}} \,ds \right\}\\
&\le \sqrt{\frac 2 \pi}dB\tilde Q \left\{\frac{2h}{\sqrt{kh}}\ln(4k) + \frac{h}{2} \ln\left(4k\right) \sum_{j=1}^{k-1} \frac{1}{\sqrt{j}\sqrt{k-j}} + \frac{h}{2} \sum_{j=1}^{k-1} \frac{1}{\sqrt{j}\sqrt{k-j}} \right\} \\
&\le \sqrt{\frac 2 \pi}dB\tilde Q\left\{\frac{2}{\sqrt{kh}}\ln\left(4k\right) + \frac{\pi}{2}\Big(1 +\ln\left(4k\right) \Big) \right\}h.
\end{align*}
Therefore, 
\begin{align*}
  \displaystyle \left\|p(kh,.) - p^h(kh,.) \right\|_{L^1} &\le \frac{\tilde L}{\sqrt{kh}}\Big(1 + \ln\left(k\right)\Big)h+ \sqrt{\frac{2}{\pi}}dB \sum_{j=1}^{k-1} \frac{h}{\sqrt{kh - jh}}\left\|p(jh,.) - p^h(jh,.) \right\|_{L^1}
\end{align*}
where $\displaystyle \tilde L = 2 M\Big(2+d\ln(2)\Big) + dB \left\{\sqrt{\frac 2 \pi}\left(3 + 4\ln(2)\tilde Q \right) + \left(2B\left(1+\frac{d-1}{\pi} \right)+ \sqrt{\frac{\pi}{2}}\Big(1+2\ln(2)\Big)\tilde Q \right)\sqrt{T} \right\}$. Iterating this inequality, using the fact that for $j \le k, \ln(j) \le \ln(k)$ and using Lemma \ref{sumPi}, we obtain:
\begin{align*}
\displaystyle &\left\|p(kh,.) - p^h(kh,.) \right\|_{L^1} \\
  &\le \frac{\tilde L}{\sqrt{kh}}\Big(1 + \ln\left(k\right)\Big)h+ \sqrt{\frac{2}{\pi}}dB \sum_{j=1}^{k-1} \frac{h}{\sqrt{kh - jh}}\left(\frac{\tilde L}{\sqrt{jh}}\Big(1 + \ln\left(j\right)\Big)h+ \sqrt{\frac{2}{\pi}}dB \sum_{l=1}^{j-1} \frac{h}{\sqrt{jh - lh}}\left\|p(lh,.) - p^h(lh,.) \right\|_{L^1} \right) 
\end{align*}

\begin{align*}
  &\le \tilde{L}\left(\frac{1}{\sqrt{kh}} + \sqrt{\frac 2 \pi}dB \sum_{j=1}^{k-1}\frac{h}{\sqrt{kh-jh}\sqrt{jh}} \right)\Big(1 + \ln\left(k\right)\Big)h+ 2d^2B^2 h\sum_{j=1}^{k-1}\left\|p\left(jh,.\right) - p^h\left(jh,.\right) \right\|_{L^1}\\
&\le \tilde L\sqrt{h}\left(1 + dB\sqrt{2\pi T }\right)\frac{1 + \ln\left(k\right)}{\sqrt{k}}+ 2d^2B^2 h\sum_{j=1}^{k-1}\left\|p\left(jh,.\right) - p^h\left(jh,.\right) \right\|_{L^1}.
\end{align*}
We apply Lemma \ref{Gronw} and obtain that:
\begin{align*}
\displaystyle &\left\|p(kh,.) - p^h(kh,.) \right\|_{L^1} \\
&\le \tilde L\sqrt{h}\left(1 + dB\sqrt{2\pi T }\right)\frac{1 + \ln\left(k\right)}{\sqrt{k}} + 2d^2B^2 \tilde L\sqrt{h}\left(1 + dB\sqrt{2\pi T }\right) h \sum_{j=1}^{k-1}\frac{1+\ln(j)}{\sqrt{j}}\exp\Big(2d^2B^2\left(kh -(j+1)h\right)\Big).
\end{align*}
Now, using once again the fact that for $j \le k, \ln(j) \le \ln(k)$, we have:
\begin{align*}
  \displaystyle \sqrt{h}\sum_{j=1}^{k-1}\frac{1+\ln(j)}{\sqrt{j}}\exp\Big(2d^2B^2\left(kh -(j+1)h\right)\Big) &\le (1+\ln(k))\exp\Big(2d^2B^2 T \Big) \int_h^{kh}\frac{ds}{\sqrt{s}} \\
  &= 2(1+\ln(k))\exp\Big(2d^2B^2 T \Big)\sqrt{T}.
\end{align*}
Therefore, we deduce that:
$$ \displaystyle  \forall h \in (0,T], \forall k \in \left \llbracket 1, \left \lfloor \frac{T}{h}\right \rfloor \right \rrbracket , \quad \left\|p(kh,.) - p^h(kh,.) \right\|_{L^1} \le \frac{\tilde C}{\sqrt{kh}}\Big(1 + \ln\left(k\right)\Big)h$$
where $\displaystyle \tilde C = \tilde L \left(1 + 4d^2B^2 T\left(1 + dB\sqrt{2\pi T }\right) \exp\Big(2d^2B^2T \Big)\right)$.


\subsection{Proof of Proposition \ref{propregM}}\label{demooo}

We first use Inequality \eqref{hyp2} from Proposition \ref{hyp} to obtain a stronger regularity of $p(t,.)$ with respect to the time variable. 

\begin{lem}\label{evoldensM}
Assume $b:[0,T]\times\R^d \to \R^d$ is measurable and bounded by $B<+\infty$ such that $\,\sup_{t \in [0,T]} \|\nabla \cdot b(t,.)\|_{L^\rho}<+\infty$ for some $\rho \in [d,+\infty]$ or that for $d=1$, $\sup_{t \in [0,T]}\left\|\partial_x b(t,.)\right\|_{\rm TV}<+\infty$; where $\nabla \cdot b(t,.)$ and $\partial_x b(t,.)$ are respectively the spatial divergence and the spatial derivative of $b$ in the sense of distributions.
Moreover, assume that $m$ admits a density w.r.t. the Lebesgue measure that belongs to $W^{1,1}\left(\R^d\right)$. We have:
   $$\displaystyle \exists \, \hat Q<+\infty, \forall \, 0 < r\le s\le T,\quad \left\|p(s,.) - p(r,.) \right\|_{L^1} \le \hat Q \Bigg((\sqrt s - \sqrt r)+ (s-r)\ln\left(\frac{s}{s-r}\right) + r\ln\left(s/r\right) + (s-r)\Bigg).$$
\end{lem}
\begin{proof}
  We will adapt, once again, the proof of Lemma \ref{evoldens}. Let $0 < r \le s \le T$, using Equality \eqref{diffP}, the estimates \eqref{heateq}, \eqref{timeDerivG} and \eqref{FirstDerivG} from Lemma \ref{EstimHeatEq}, Inequality \eqref{hyp2} from Proposition \ref{hyp} and Lemma \ref{intg}, we obtain:
  \begin{align*}
  \displaystyle &\left\|p(s,.) - p(r,.) \right\|_{L^1} \\
  &\le \left\|\left( G_s - G_r \right) * m \right\|_{L^1} + \left\|\int_0^r \left( G_{s-u} - G_{r-u} \right)*\nabla \cdot \Big(b(u,.)p(u,.)\Big)\,du \right\|_{L^1} + \left\|\int_r^s G_{s-u} * \nabla \cdot \Big(b(u,.)p(u,.)\Big)\,du \right\|_{L^1} \\
  &\le \left\|\int_r^s \left(\partial_u G_u* m\right)\,du \right\|_{L^1} + \tilde M \int_0^s \int_{r-u}^{s-u} \left\|\partial_{\theta} G_{\theta} \right\|_{L^1}\,d\theta\,du  + \tilde M (s-r)\\
  &\le \frac{1}{2}\sum_{i=1}^d\|\partial_{x_i} m\|_{L^1}\int_r^s \left\| \partial_{x_i} G_{u}\right\|_{L^1}\,du + d \tilde M\int_0^r \ln\left(\frac{s-u}{r-u}\right)\,du  + \tilde M (s-r)\\
  &\le 2\sqrt{\frac 2 \pi} \sum_{i=1}^d\|\partial_{x_i} m\|_{L^1} \left(\sqrt{s} - \sqrt{r} \right) + d \tilde M \left((s-r)\ln\left(\frac{s}{s-r}\right) + r\ln\left(s/r\right) \right) +  \tilde M (s-r).
  \end{align*}
  The conclusion holds with $\displaystyle \hat Q = \max\left(2\sqrt{\frac 2 \pi} \sum_{i=1}^d\|\partial_{x_i} m\|_{L^1}, d\tilde M \right)$.
\end{proof}

We are now ready to prove Proposition \ref{propregM}. Once more, using Equality \eqref{caract} and Proposition \ref{propspde}, to prove the theorem amounts to prove that:
$$ \exists \, \hat{C}<+\infty, \forall h \in (0,T], \, \forall k \in \left\llbracket 1, \left\lfloor\frac{T}{h}\right\rfloor \right\rrbracket, \quad \left\| p(kh,.) - p^h(kh,.)\right\|_{L^1} \le \hat{C}\Big(1 + \ln\left(k\right) \Big)h. $$
For $h \in (0,T], \, k  \in \left\llbracket 1, \left\lfloor\frac{T}{h}\right\rfloor \right\rrbracket$, we recall Inequality \eqref{firstintegB}:
\begin{align*}
  \displaystyle \Big\|p(kh,.)-&p^h\left(kh,.\right)\Big\|_{L^1} \le \left\|V^h_1(k,.)\right\|_{L^1}+ \left\|V^h_2(k,.)\right\|_{L^1} + \sqrt{\frac 2 \pi}dB\sum_{j=1}^{k-1} \frac{h}{\sqrt{kh - jh}}\left\|p(jh,.) - p^h(jh,.) \right\|_{L^1}\\
  &+ 2dB^2 \left(1 + \frac{d-1}{\pi}\right) \left(\frac{1}{2} + \ln\left(k \right) \right)h + \left\| \int_0^{h} \bigg(\nabla G_{kh-s} * \Big(b(s,.)p(s,.)\Big) - \nabla G_{kh} * \Big(b(s,.)m\Big)\bigg)\,ds \right\|_{L^1}.
\end{align*}
Concerning the last term of the right-hand side of this previous inequality, we use Lemma \ref{Sobol} and the additional regularity of $b$ to transfer the gradient from $G$ to $bp$; and using Inequality \eqref{hyp2} from Proposition \ref{hyp}, we obtain that:
\begin{align*}
  \displaystyle \bigg\| \int_0^{h} \bigg(G_{kh-s} * \nabla \cdot \Big(b(s,.)p(s,.)\Big) &- G_{kh} * \nabla \cdot \Big(b(s,.)m\Big)\bigg)\,ds \bigg\|_{L^1} \\
  &\le \int_0^{h} \left(\left\|\nabla  \cdot \Big(b(s,.)p(s,.)\Big) \right\|_{\rm TV} + \left\|\nabla  \cdot \Big(b(s,.)m\Big) \right\|_{\rm TV}\right)\,ds \le 2\tilde M h.
\end{align*}
Let us estimate $\left\|V^h_1(k,.)\right\|_{L^1}$ and $\left\|V^h_2(k,.)\right\|_{L^1}$ for $k \ge 2$ by taking advantage of the additional regularity of $b$ and $m$, and using Equality \eqref{nablaP} from Proposition \ref{hypH}:\\ \\
$\bullet$ We recall Equality \eqref{v1}:
\begin{align*}
\displaystyle &V_1^h(k,.) = \int_h^{(k-1)h}\left(\int_{kh-s}^{kh-\tau^h_s}\partial_u G_u \,du\right)*\nabla \cdot \Big(b(s,.)p(s,.) \Big)\,ds + \int_{(k-1)h}^{kh}\bigg(G_{kh-s}- G_{kh-\tau^h_s} \bigg)* \nabla \cdot \Big(b(s,.)p(s,.) \Big)\,ds.
\end{align*}
Therefore, using the fact that $\ln(1+x) \le x, \, \forall x>0$ and Estimate \eqref{timeDerivG} from Lemma \ref{EstimHeatEq}, we obtain: 
\begin{align*}
\displaystyle \left\|V_1^h(k,.)\right\|_{L^1} &\le d\tilde M \int_{h}^{(k-1)h}\ln\left(1 + \frac{s-\tau^h_s}{kh -s}\right)\,ds + 2 \tilde M h\\
&\le d\tilde M \int_{h}^{(k-1)h} \frac{h}{kh-s}\,ds + 2 \tilde M h = \tilde M\left( d\ln(k-1) + 2\right)h \\
&\le \tilde M \Big(2 + d\ln\left(k\right) \Big)h.
\end{align*}
$\bullet$ Using Lemma \ref{evoldensM}, the estimate \eqref{FirstDerivG} from Lemma \ref{EstimHeatEq} and the fact that $\ln(1+x) \le x, \, \forall x>0$, we have:
\begin{align*}
\displaystyle \left\|V_2^h(k,.)\right\|_{L^1} &\le \sqrt{\frac 2 \pi}dB\hat Q \bigg\{ \int_{h}^{kh} \frac{\sqrt s - \sqrt{\tau^h_s}}{\sqrt{kh - \tau^h_s}}\,ds + \int_{h}^{kh} \frac{s-\tau^h_s}{\sqrt{kh-\tau^h_s}}\ln\left(\frac{s}{s-\tau^h_s} \right)\,ds \\
&\phantom{\sqrt{\frac 2 \pi}dB\hat Q \bigg\{ \int_{h}^{kh} \frac{\sqrt s - \sqrt{\tau^h_s}}{\sqrt{kh - \tau^h_s}}\,ds}+ \int_{h}^{kh} \frac{\tau^h_s}{\sqrt{kh-\tau^h_s}}\ln\left(1 + \frac{s-\tau^h_s}{\tau^h_s} \right)\,ds +  \int_{h}^{kh}\frac{s - \tau^h_s}{\sqrt{kh - \tau^h_s}} \,ds \bigg\}\\
&\le \sqrt{\frac 2 \pi}dB\hat Q \bigg\{\int_{h}^{kh} \frac{h}{\sqrt{kh - s}\sqrt{s}}\,ds+ \int_{h}^{kh}\frac{s-\tau^h_s}{\sqrt{kh-s}}\ln\left(\frac{kh}{s-\tau^h_s} \right)\,ds + 2\int_{h}^{kh}\frac{h}{\sqrt{kh - s}} \,ds\bigg\}.
\end{align*}
The function $x \mapsto x\ln\left(kh/x\right)$ is increasing on the interval $(0,kh/e]$, attains its maximum at $x = kh/e$ and non-increasing on the interval $[kh/e, +\infty)$. Therefore, we get that:
$$\displaystyle \left(s-\tau^h_s\right)\ln\left(\frac{kh}{s-\tau^h_s} \right)\le \left( h \ln\left(k\right) \mathbf{1}_{\left\{ h \le \frac{kh}{e}\right\}} + \frac{kh}{e}\mathbf{1}_{\left\{ h > \frac{kh}{e}\right\}} \right) \le \Big(1+\ln\left(k\right)\Big)h.$$
We then deduce that:
\begin{align*}
\displaystyle \left\|V_2^h(k,.)\right\|_{L^1} &\le \sqrt{\frac 2 \pi}dB\hat Q \left(\pi + 2\sqrt{T}\Big(1+\ln\left(k\right)\Big) + 4\sqrt{T} \right)h.
\end{align*}
Therefore, 
\begin{align*}
  \displaystyle \left\|p(kh,.) - p^h(kh,.) \right\|_{L^1} &\le \hat L \Big(1+\ln\left(k\right)\Big)h+ \sqrt{\frac{2}{\pi}}dB \sum_{j=1}^{k-1} \frac{h}{\sqrt{kh - jh}}\left\|p(jh,.) - p^h(jh,.) \right\|_{L^1}
\end{align*}
where $\displaystyle \hat L = \max(2,d)\tilde M  + dB \left(2B\left(1 + \frac{d-1}{\pi} \right) + \sqrt{\frac 2 \pi}\hat Q \left(\pi + 6\sqrt{T}\right) \right)$. Iterating this inequality, using the fact that for $j \le k, \ln(j) \le \ln(k)$ and using Lemma \ref{sumPi}, we obtain:
\begin{align*}
 \displaystyle &\left\|p(kh,.) - p^h(kh,.) \right\|_{L^1} \\
 &\le \hat L \Big(1+\ln\left(k\right)\Big)h + \sqrt{\frac{2}{\pi}}dB \sum_{j=1}^{k-1} \frac{h}{\sqrt{kh - jh}}\left(\hat L \Big(1+\ln\left(j\right)\Big)h + \sqrt{\frac{2}{\pi}}dB \sum_{l=1}^{j-1} \frac{h}{\sqrt{jh - lh}} \left\|p(lh,.) - p^h(lh,.) \right\|_{L^1}\right) \\
 &\le \hat L\left(1 + \sqrt{\frac 2 \pi}dB\int_h^{kh} \frac{ds}{\sqrt{kh - s}} \right) \Big(1+\ln\left(k\right)\Big)h + 2d^2B^2h \sum_{j=1}^{k-1}\left\|p(jh,.) - p^h(jh,.) \right\|_{L^1} \\
 &\le \hat L\left(1 + 2\sqrt{\frac{2T}{\pi}}dB \right) \Big(1+\ln\left(k\right)\Big)h + 2d^2B^2h \sum_{j=1}^{k-1}\left\|p(jh,.) - p^h(jh,.) \right\|_{L^1}.
\end{align*}
Finally using Lemma \ref{Gronw} and for $\hat C = \hat L \left(1+ 2\sqrt{\frac{2T}{\pi}}dB \right)\left(1 + 2d^2B^2T\exp\Big(2d^2B^2T \Big) \right)$, we conclude that:
\begin{align*}
 \displaystyle \left\|p(kh,.) - p^h(kh,.) \right\|_{L^1} &\le \hat C \Big(1+\ln\left(k\right)\Big)h.
\end{align*}


\section{Numerical Experiments}\label{NumExp}

In order to confirm our theoretical estimates for the convergence rate in total variation of $\mu^h$ to its limit $\mu$, we study SDEs with a piecewise constant drift coefficient and additive noise, as done by G\"ottlich, Lux and Neuenkirch in \cite{GotNeue}. We consider the special case of one-dimensional SDEs with one drift change at zero: 
\begin{align}
	X_t = x + W_t + \int_0^t\bigg( \alpha \mathbf{1}_{(-\infty,0)}(X_s) + \beta \mathbf{1}_{[0,+\infty)}(X_s) \bigg)\,ds \label{generalSDE}
\end{align}
where $X_0=x \in \R$ is the initial value and $\alpha, \beta \in \R$. The difference $(\beta - \alpha)$ represents the height of the jump at the discontinuity point zero. Here, the drift satisfies the reinforced hypothesis of Theorem \ref{cvgenceB} where the derivative of the drift in the sense of distributions is equal to $(\beta-\alpha)\delta_0$.\\

We analyze how the initial value $x$ affects the error and how the jump height influences the empirical rate of convergence. We also observe how the drift direction towards or away from the discontinuity point zero influences the error. When $\alpha > 0 > \beta$, we speak about inward pointing drift coefficient. Inversely, when $\alpha < 0 < \beta$, it is about outward pointing drift coefficient. 

\subsection{The specific case $\alpha = -\beta = \theta >0$}

For $\theta >0$, we study SDEs with inward pointing drift coefficient of the form: $$\displaystyle  X_t = x + W_t - \int_0^t \sgn(X_s)\,\theta \,ds.$$ 
This process is called a Brownian motion with two-valued, state-dependent drift. This example was also used by Kohatsu-Higa, Lejay and Yasuda in \cite{Koh} to estimate the weak convergence rate of the Euler-Maruyama scheme. 
According to \cite{KaratShre}, the transition density function of the process $(X_t)_{t \ge 0}$ starting at $x  \ge 0$ is the following:
\begin{align*}
p_t(x,z) = \left\{
      \begin{aligned}
        &\frac{1}{\sqrt{2\pi t}}\exp\left(-\frac{\left(x-z-\theta t\right)^2}{2t} \right) + \frac{\theta e^{-2\theta z}}{\sqrt{2\pi t}}\int_{x+z}^{+ \infty} \exp\left(-\frac{(y-\theta t)^2}{2t} \right)\,dy \quad \text{when } z > 0, \\
        &\frac{e^{2\theta x}}{\sqrt{2\pi t}}\exp\left(-\frac{\left(x-z+\theta t\right)^2}{2t} \right) + \frac{\theta e^{2\theta z}}{\sqrt{2\pi t}}\int_{x-z}^{+ \infty} \exp\left(-\frac{(y-\theta t)^2}{2t} \right)\,dy \quad \text{when } z \le 0.\\
      \end{aligned}
    \right. 
\end{align*}
For $x \le 0$, the transition density can be deduced from the symmetry of the Brownian motion that gives $p_t(x,z) = p_t(-x,-z)$.\\

We seek to observe the dependence of the error in total variation at terminal time $T$: $\left\| \mu_{T} - \mu^h_{T}\right\|_{\rm TV}$ on the time step $h$ that we choose s.t. $\frac{T}{h}$ is an integer. To do so, we estimate $ \left\|p(T,.) - p^h(T,.) \right\|_{L^1}$ using a kernel density estimator for $p^h(T,.)$. We denote by $N$ the number of random variables $\left( X^{i,h}_T\right)_{1\le i \le N}$ that are i.i.d. with density $p^h(T,.)$. The kernel density estimator of this latter is defined by: $$p_{\epsilon, N}^h(T,x) = \frac{1}{\epsilon N} \sum\limits_{j=1}^N K\left( \frac{x-X^{j,h}_T}{\epsilon}\right)$$ where $K$ represents the kernel and $\epsilon >0$ is a smoothing parameter called the bandwidth. The kernel is a non-negative and integrable even function that ensures the required normalization of a density i.e. $\int_{-\infty}^{+\infty} K(x)\,dx = 1$. As for the smoothing parameter, its influence is critical since a very small $\epsilon$ makes the estimator show insignificant details and a very large $\epsilon$ causes oversmoothing and may mask some characteristics. So a compromise is needed. The optimal smoothing parameter can be chosen through a minimisation of the asymptotic mean integrated squared error. For an explicit known density, as it is the case here, we have that: 
\begin{align}\label{epsilon}
	\epsilon = c N^{-1/5} \quad \text{with} \quad c = \frac{R(K)^{1/5}}{m_2(K)^{2/5}R\left(\partial^2_{xx} p_t\right)^{1/5}}
\end{align}
where for a given function $g$, $\displaystyle R(g) = \int_{\R}g(x)^2\,dx$ and $\displaystyle m_2(g) = \int_{\R}x^2 g(x)\,dx$. We will choose, in what follows, the Epanechnikov kernel defined by:
$$ K(x) = \frac{3}{4}\left(1 - x^2\right) \mathbf{1}_{\left\{|x| \le 1\right\}} $$
 which is known to be theoretically optimal in a mean square error sense with $R(K) = 3/5$ and $m_2(K) = 1/5$. For $\left( X^{(i),h}_T\right)_{1\le i \le N}$ denoting the increasing reordering of $\left( X^{i,h}_T\right)_{1\le i \le N}$, we make the following trapezoidal approximation:
\begin{align*}
 \displaystyle \left\|p^h(T,.) - p(T,.) \right\|_{L^1} &\simeq \sum_{i=1}^{N-1} \frac{1}{2} \left(X^{(i+1),h}_T - X^{(i),h}_T\right)\Bigg\{\left|p_{\epsilon, N}^h\left(T,X^{(i+1),h}_T\right) - p\left(T,X^{(i+1),h}_T \right) \right| \\
 &\phantom{\sum_{i=1}^{N-1} \frac{1}{2} \left(X^{(i+1),h}_t - X^{(i),h}_t\right)\Bigg\{p_{\epsilon, N}^h\left(X^{(i),h}_t\right) - p\left(X\right)} + \left|p_{\epsilon, N}^h\left(T,X^{(i),h}_T\right) - p\left(T,X^{(i),h}_T \right) \right|\Bigg\}.
\end{align*}

We also define the precision of this estimation as half the width of the $95 \%$ confidence interval of the empirical error i.e. $\text{Precision} = 1.96 \times \sqrt{\text{Variance}/R}$ where $R$ denotes the number of Monte-Carlo runs and Variance denotes the empirical variance over these runs of the empirical error.

\subsubsection{Illustration of the theoretical order of convergence in total variation}

To observe the convergence rate in total variation for the case $\theta = 1.0$ and $x=0.0$, we fix the time horizon $T=1$ and the number $N=500000$ of i.i.d. samples in the kernel density estimator large enough in order to observe the effect of the time-step $h$ on the error. The simulation is done with $R = 20$ Monte-Carlo runs. We obtain the following results for the estimation of the error and its associated precision:

\begin{center}
  \begin{tabular}{ |c|c|c|c|c|}
    \hline
    \multicolumn{5}{|c|}{Evolution of the total variation error w.r.t. $h$ } \\ \hline
    Time-step $h$ & Estimation & Precision & Ratio of decrease & Theoretical Ratio \\ \hline
    $T/4$   & $0.2903$ & $ 5.48 \times 10^{-4}$ & $\times$ & $\times$  \\ \hline
    $T/8$   & $0.1680$ & $ 6.96 \times 10^{-4}$ & $1.73$  & $1.63$ \\ \hline
    $T/16$  & $0.0956$ & $ 4.88 \times 10^{-4}$  & $1.76$ & $1.69$ \\ \hline
    $T/32$  & $0.0543$ & $ 5.26 \times 10^{-4}$ & $1.76$ & $1.73$ \\ \hline
    $T/64$  & $0.0314$ & $ 6.53 \times 10^{-4}$ & $1.73$  & $1.76$ \\ \hline
    $T/128$ & $0.0191$ & $ 3.55 \times 10^{-4}$ & $1.64$ & $1.79$ \\ \hline
    $T/256$ & $0.0133$ & $ 2.71 \times 10^{-4}$ & $1.43$ & $1.80$ \\\hline
    $T/512$ & $0.0101$ & $ 3.10 \times 10^{-4}$ & $1.31$ & $1.82$ \\
    \hline
  \end{tabular}
\end{center}

\vskip 0.4cm

\begin{itemize}
	\item We observe that the ratio of successive estimations  $\frac{\text{Estimation}(h)}{\text{Estimation}(h/2)}$ is roughly around $1.72$. But when $h$ becomes small, the ratio decreases towards $1$ (a constant error) because for so small discretizations steps, the effect of the kernel density estimation parameter $N$ cannot be neglected unless $N$ is extremely large.
	\item The last column refers to the theoretical ratios equal to $2\left(\frac{1 + \ln(T/h)}{1+\ln(2T/h)}\right)$ which is the expected behaviour of the error. On the range of values $\left\{\frac{T}{8},\frac{T}{16},\frac{T}{32},\frac{T}{64},\frac{T}{128} \right\}$, both the empirical and the theoretical ratios are equal to $1.72$ in average.
	\item Moreover, the order of convergence in total variation of the Euler scheme is here equal to $0.76$. This order is given by the slope of the regression line, which we obtain when plotting $\log\left\|p(T,.) - p^h(T,.)\right\|_{L^1}$ versus $\log(h)$.
\end{itemize}

\subsubsection{Dependence of the order of convergence on the initial value $x$ for fixed $\theta=1$}

To underline the influence of the initial value of the SDE, we start by generating plots of the explicit transition density function for various initializations $x \in \{-1, 0, 1, 2.5, 5 \}$ and different time horizons $T \in \{1,3,6\}$. We choose a fixed $\theta=1$. We also plot the kernel transition density estimation for the different values of $x$ at $T=1$ for $N=100000$ and a time-step $h = 0.0001$.\\

We first observe from Figures \ref{fig:fig1} and \ref{fig:fig2} that the kernel density estimator catches the discontinuity and reproduces well the expected distribution. We also see from Figures \ref{fig:fig2}, \ref{fig:fig3} and \ref{fig:fig4} that when the process starts from the discontinuity point $x=0.0$ or close to it $x \in \{-1,1\}$, it visits the discontinuity point several times. When we increase the time horizon $T$, the inward pointing drift allows the process to visit the discontinuity point zero when starting far from it.

\newpage

\begin{figure}[ht]
  \centering
  \begin{subfigure}[b]{0.4\linewidth}
    \includegraphics[width=\linewidth]{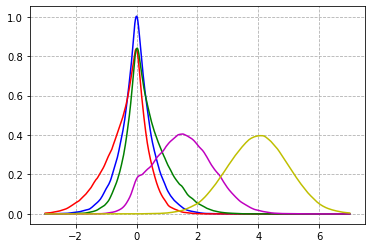}
    \caption{Estimated densities for $T=1$}
    \label{fig:fig1}
  \end{subfigure}
  \begin{subfigure}[b]{0.4\linewidth}
    \includegraphics[width=\linewidth]{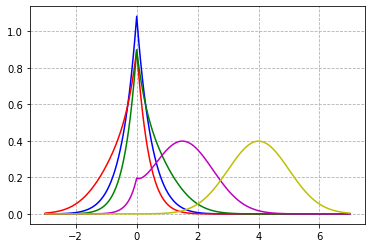}
    \caption{Explicit densities for $T=1$}
    \label{fig:fig2}
  \end{subfigure}
  \begin{subfigure}[b]{0.4\linewidth}
    \includegraphics[width=\linewidth]{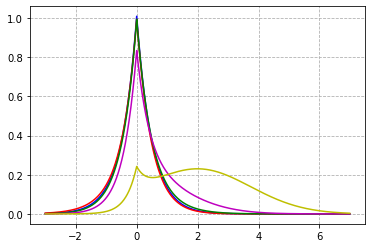}
    \caption{Explicit densities for $T=3$}
    \label{fig:fig3}
  \end{subfigure}
  \begin{subfigure}[b]{0.4\linewidth}
    \includegraphics[width=\linewidth]{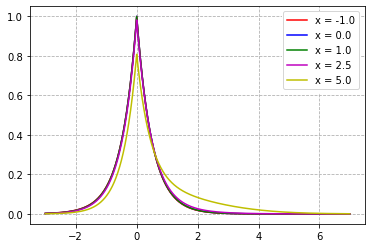}
    \caption{Explicit densities for $T=6$}
    \label{fig:fig4}
  \end{subfigure}
  \caption{The transition density function for various initializations $x$}
  \label{fig:densities}
\end{figure}

We also generate an example of a solution sample path with time-step $h=0.006$ for various initializations $x$ to confirm that point.

\begin{figure}[ht]
\centering
  \includegraphics[width=8cm]{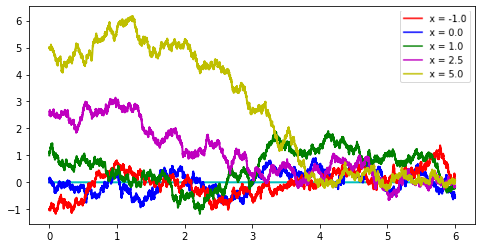}
  \caption{Example of a solution sample paths for various initializations $x$}
\end{figure}

Now, we give the empirical convergence orders obtained for various initializations $x$. These orders are given by the slopes of the regression lines in a log-log scale. The parameters used here are $T=5$, $N=500000$ and step sizes $h \in \left\{\frac{T}{8},\frac{T}{16},\frac{T}{32},\frac{T}{64},\frac{T}{128},\frac{T}{256},\frac{T}{512}\right\}$. \\

\begin{center}
  \begin{tabular}{|c|c|c|c|c|c|}
    \hline
    Initial value $x$  & $-1.0$ & $0.0$ & $1.0$ & $2.5$ & $5.0$ \\ \hline
    Empirical convergence order  & $0.77$ & $0.77$ & $0.77$ & $0.77$ & $0.70$ \\ \hline 
  \end{tabular} 
\end{center}

\vskip 0.4cm

 We can see from the above table that the empirical convergence orders are stable with respect to the initial value for this type of diffusion. A small difference is observed for the initial value $x=5.0$: then the process starts far from the discontinuity point zero and the time horizon is not long enough for the process to visit it with high probability. When the process does not reach the discontinuity, the Euler scheme is exact making the error smaller and the influence of the kernel estimation error stronger.

 \subsubsection{Dependence of the order of convergence on the jump height}

To underline the influence of the jump height equal to $2\theta$, we start by generating an example of a sample path for two values of $\theta \in \{1,10\}$ with fixed time-horizon $T=1$, time-step $h=0.003$  and initial value $x=0.0$.

\begin{figure}[ht]
\centering
  \includegraphics[width=8cm]{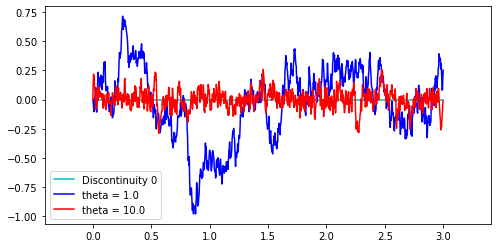}
  \caption{Example of a solution sample paths for various $\theta$}
\end{figure}

 We observe that when $\theta$ is big, the process is more likely to visit the discontinuity multiple times than for a smaller value of $\theta$. We can explain this by introducing the process $Y_t = \frac{1}{\theta}X_t$ that starts from $Y_0 = \frac{x}{\theta}$ and has the following dynamics:
$$ Y_t = Y_0 + \frac{1}{\theta}W_t - \int_0^t \sgn\left(Y_t\right)\,dt. $$
The diffusion coefficient equal to $1/\theta$ becomes very small when $\theta$ becomes large so that the process has an almost deterministic behaviour and is sticked to the discontinuity point zero by the drift.\\

Also, the explicit transition density tends to the Laplace density $\theta e^{-2\theta|x|}$ when $t \to +\infty$. We generate the plots of the kernel density estimate and the explicit transition density function for various $\theta = \{1,3,5,10,20\}$. We choose $T=1$, $x=0.0$, $N=100000$ and $h=0.0001$.

\begin{figure}[ht]
  \centering
  \begin{subfigure}[b]{0.4\linewidth}
    \includegraphics[width=\linewidth]{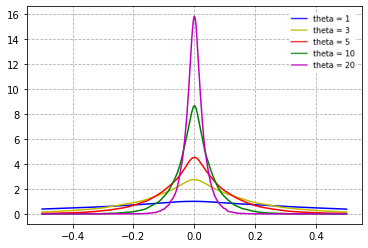}
      \caption{The kernel density estimator }
  \end{subfigure}
  \begin{subfigure}[b]{0.4\linewidth}
    \includegraphics[width=\linewidth]{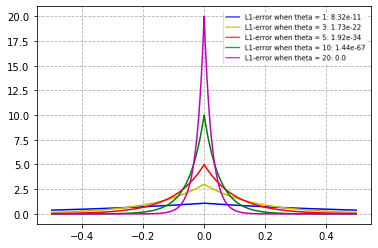}
      \caption{The explicit transition density function}
      \label{cettefigure}
  \end{subfigure}
  \caption{The transition density function for various $\theta$}
  \label{fig:theta_dens_b}
\end{figure}

We can see that when $\theta$ is big, the density converges quickly towards the Laplace density. In Figure \ref{cettefigure}, we confirm this behaviour by giving the $L^1$-error between the explicit transition densities and the Laplace densities for each $\theta$. \\

Now, we give the empirical convergence orders obtained for different values of $\theta$. The parameters used here are $T=1$, $x=0.0$ and $N=800000$. We choose ranges of step-sizes $h$ depending on $\theta$ since for small $\theta$, the discretization error are smaller and the kernel estimation error comparatively more influent. For large $\theta$, a large time-step implies a very large error because on each time-step, when starting close to the discontinuity point zero, the Euler scheme will move far away to the other side of this discontinuity, a behaviour forbidden for the limiting SDE by the large inward pointing drift.

\begin{center}
  \begin{tabular}{|c|c|c|c|c|c|}
    \hline
    Jump-height $\theta$  & $1.0$ & $3.0$ & $5.0$ & $10.0$ & $20.0$ \\ \hline
    time-step range $h$ & $\displaystyle \left\{\frac{1}{2^3},...,\frac{1}{2^8}\right\}$ & $\displaystyle\left\{\frac{1}{2^4},...,\frac{1}{2^9}\right\}$ & $\displaystyle\left\{\frac{1}{2^5},...,\frac{1}{2^{10}}\right\}$ & $\displaystyle\left\{\frac{1}{2^7},...,\frac{1}{2^{12}}\right\}$ & $\displaystyle\left\{\frac{1}{2^8},...,\frac{1}{2^{13}}\right\}$ \\ \hline
    Empirical convergence order  & $0.76$ & $0.77$ & $0.77$ & $0.77$ & $0.72$ \\ \hline 
  \end{tabular}
\end{center}

\vskip 0.4cm

We can see from the above table that the empirical convergence orders are relatively stable with respect to the jump-height for this type of diffusion. A small difference is observed for $\theta=20.0$ since the error is still large for the time-steps considered.   

\subsection{General case}

To our knowledge, no closed-form of a density of $X_t$ solving \eqref{generalSDE} is available for general $\alpha, \beta \in \R$. The idea is still to estimate the $L^1$-norm of the difference of the densities at maturity $T$ but this time, instead of comparing $p^h(T,.)$ to $p(T,.)$, we will compare $p^{h}(T,.)$ to $p^{h/2}(T,.)$ and the expected behaviour is:
\begin{align*}
 \displaystyle \left\|p^h(T,.) - p^{h/2}(T,.) \right\|_{L^1} \le \left\|p^h(T,.) - p(T,.) \right\|_{L^1} + \left\|p^{h/2}(T,.) - p(T,.) \right\|_{L^1} \le \left(\frac{3}{2} + \ln(2) \right) \tilde C \left(1 + \ln\left(\frac T h \right) \right)h.
\end{align*}
In order to estimate $p^h(T.)$, we use, once again, a kernel density estimator but this time, we choose the Gaussian kernel defined by:
$$ K(x) = \frac{1}{\sqrt{2 \pi}} \exp \left(-\frac{x^2}{2} \right) \quad \text{ for } x \in \R.$$
We make this choice since no explicit density is available to estimate the bandwith \eqref{epsilon} and for Gaussian kernels we can obtain use the so-called Silverman's rule of thumb \cite{Silv}:
$$ \epsilon = c N^{-1/5} \quad \text{ with } \quad c= 0.9 \times \min\left(\hat \sigma, \frac{IQR}{1.34}\right)$$
where the standard deviation $\hat \sigma$ and the interquantile range $IQR$ are easily computed from the sample of size $N$. When the density to estimate is a bimodal mixture, we apply Silverman's rule of thumb on each mode.\\
For $\left(X^{i,{h/2}}_T\right)_{1 \le i \le N}$ i.i.d. variables with density $p^{h/2}(T,.)$, the kernel density estimator of this latter is then defined by: $$p_{\epsilon, N}^{h/2}(T,x) = \frac{1}{\epsilon N} \sum\limits_{j=1}^N K\left( \frac{x-X^{j,h/2}_T}{\epsilon}\right)$$
and we make the following trapezoidal approximation:
\begin{align*}
  \displaystyle \left\|p^h(T,.) - p^{h/2}(T,.) \right\|_{L^1} &\simeq \sum_{i=1}^{N-1} \frac{1}{2} \left(X^{(i+1),h}_T - X^{(i),h}_T\right)\Bigg\{\left|p_{\epsilon, N}^h\left(T,X^{(i+1),h}_T\right) - p_{\epsilon, N}^{h/2}\left(T,X^{(i+1),h}_T \right) \right| \\
 &\phantom{\sum_{i=1}^{N-1} \frac{1}{2} \left(X^{(i+1),h}_t - X^{(i),h}_t\right)\Bigg\{p_{\epsilon, N}^h\left(X^{(i),h}_t\right) - } + \left|p_{\epsilon, N}^h\left(T,X^{(i),h}_T\right) - p_{\epsilon, N}^{h/2}\left(T,X^{(i),h}_T \right) \right|\Bigg\}.
\end{align*}

In what follows, we will study the case of outward pointing diffusions i.e. $\alpha < 0 < \beta$ and observe how the initial value and the jump-height influences the error. Beforehand, we observe the convergence rate in total variation when varying the time-step $h$.

\subsubsection{Illustration of the theoretical order of convergence in total variation}

To observe the convergence rate in total variation for the case $\alpha = - 3.0$, $\beta = 4.0$ and $x=0.0$, we fix the time horizon $T=1$ and the number $N=250000$ of i.i.d. samples in the kernel density estimator large enough in order to observe the effect of the time-step $h$ on the error. The simulation is done with $R = 20$ Monte-Carlo runs. We obtain the following results for the estimation of the error and the associated precision: 

\begin{center}
  \begin{tabular}{ |c|c|c|c|c|}
    \hline
    \multicolumn{5}{|c|}{Evolution of the total variation error w.r.t. $h$ } \\ \hline
    Time-step $h$ & Estimation & Precision & Ratio of decrease & Theoretical Ratio \\ \hline
    $T/32$   & $0.1105$ & $3.58 \times 10^{-4}$  & $\times$ & $\times$  \\ \hline
    $T/64$   & $0.0763$ & $3.17 \times 10^{-4}$  & $1.45$   & $1.76$ \\ \hline
    $T/128$  & $0.0478$ & $2.31 \times 10^{-4}$  & $1.60$   & $1.79$ \\ \hline
    $T/256$  & $0.0279$ & $1.91 \times 10^{-4}$  & $1.71$   & $1.80$ \\ \hline
    $T/512$  & $0.0156$ & $1.92 \times 10^{-4}$  & $1.79$   & $1.82$ \\ \hline
    $T/1024$ & $0.0081$ & $1.15 \times 10^{-4}$  & $1.93$   & $1.84$ \\ \hline
    $T/2048$ & $0.0043$ & $1.21 \times 10^{-4}$  & $1.87$   & $1.85$ \\
    \hline
  \end{tabular}
\end{center}

\vskip 0.4cm

\begin{itemize}
	\item We observe that the ratio of successive estimations  $\frac{\text{Estimation}(h)}{\text{Estimation}(h/2)}$ is roughly around $1.72$. 
	\item The last column refers to the theoretical ratios equal to $2\left(\frac{1 + \ln(T/h)}{1+\ln(2T/h)}\right)$ which is the expected behaviour of the error. On the range of values $\left\{\frac{T}{256},\frac{T}{512},\frac{T}{1024},\frac{T}{2048}\right\}$, both the empirical and the theoretical ratios are equal to $1.82$ in average.
	\item Moreover, the order of convergence in total variation of the Euler scheme is here equal to $0.79$. This order is given, once again, by the slope of the regression line in a log-log scale.
\end{itemize}

\subsubsection{Dependence of the order of convergence on the initial value $x$ for fixed $\alpha = -3.0$ and $\beta = 4.0$}

To underline the influence of the initial value of the SDE, we start by generating plots of the estimated transition density function for various initializations $x \in \{-1.4,-0.4,-0.2, -0.15, 0.0, 0.6 \}$. We choose $\alpha = -3.0$, $\beta=4.0$, $T=1$, $N=100000$ and $h=0.0001$. 

\begin{figure}[ht]
  \centering
  \begin{subfigure}[b]{0.45\linewidth}
    \includegraphics[width=\linewidth]{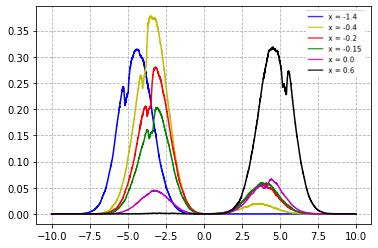}
  \end{subfigure}
  \begin{subfigure}[b]{0.45\linewidth}
    \includegraphics[width=\linewidth]{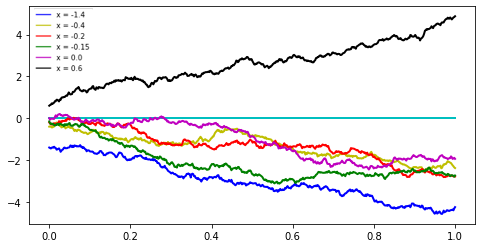}
  \end{subfigure}
  \caption{Transition density functions and examples of a solution sample paths for various initializations $x$}
  \label{iciFig}
\end{figure}

We see from Figure \ref{iciFig} that when the process starts from the discontinuity point $x=0.0$ or close to it $x \in \{-0.2,-0.15\}$, we have a bimodal mixture. When starting far from the discontinuity point zero, we are less likely to visit it and the distribution is Gaussian-like. We also confirm this point by an example of a solution sample path with time-step $h=0.001$ for various initializations $x$.\\

Now, we give the empirical convergence orders obtained for various initializations $x$. These orders are given, once again, by the slopes of the regression lines in log-log scales. The parameters used here are $T=1$, $N=500000$ and step sizes $h \in \left\{\frac{T}{16},\frac{T}{32},\frac{T}{64},\frac{T}{128},\frac{T}{256},\frac{T}{512},\frac{T}{1024},\frac{T}{2048}\right\}$. \\

\begin{center}
  \begin{tabular}{|c|c|c|c|c|c|}
    \hline
    Initial value $x$  & $-1.4$ & $-0.4$ & $-0.2$ & $0.0$ & $0.6$ \\ \hline
    Empirical convergence order  & $0.28$ & $0.68$ & $0.82$  & $0.71$ & $0.51$ \\ \hline 
  \end{tabular}
\end{center}

\vskip 0.4cm

 We can see from the above table that the empirical convergence orders seem to depend on the initial value and the spectrum of orders obtained for different initial values is very broad with values between $0.28$ and $0.82$. When starting further and further from the discontinuity point zero, we first obtain a better order of convergence but when $|x|$ becomes large it deteriorates since the kernel estimation error becomes more influent.

 \subsubsection{Dependence of the order of convergence on the jump height}

To underline the influence of the jump height equal to $(\beta-\alpha)$, we start by generating a sample path for two values of $(\alpha,\beta) \in \{(-4.0,3.0),(-0.6,1.0)\}$ with fixed time-horizon $T=1$ and initial value $x=0.0$.  

\begin{figure}[ht]
\centering
  \includegraphics[width=8cm]{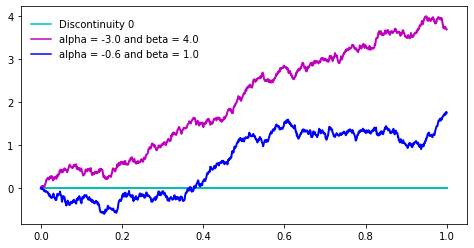}
  \caption{Solution sample paths for various $(\alpha,\beta)$}
\end{figure}

 We observe that the solution drifts away from the discontinuity point zero and therefore, there are not many chances for a drift correction to take place.\\

 Now, we give the empirical convergence orders obtained for various $(\alpha,\beta)$. These orders are given, once again, by the slopes of the regression lines in log-log scales. The parameters used here are $T=1$, $N=500000$, $x=0.0$ and step-sizes $h \in \left\{\frac{T}{8},\frac{T}{16},\frac{T}{32},\frac{T}{64},\frac{T}{128},\frac{T}{256},\frac{T}{512}\right\}$. \\

\begin{center}
  \begin{tabular}{|c|c|c|c|c|c|}
    \hline
    $(\alpha,\beta)$  & $(-6.0,8.0)$ & $(-3.0,4.0)$ & $(-1.5,2.0)$ & $(-0.75,1.0)$ & $(-0.375,0.5)$ \\ \hline
    Empirical convergence order  & $0.51$ & $0.66$ & $0.90$ & $1.02$ & $0.78$ \\ \hline 
  \end{tabular}
\end{center}

\vskip 0.4cm

 We can see from the table above that the empirical convergence orders are not stable with respect to the jump-height for outward pointing drift diffusions. Enlarging the jump height makes the solution to drift away from the discontinuity point zero and increases the error.

\subsection{Conclusion}

We were able, through our numerical experiments, to confirm our theoretical estimates for the convergence rate in total variation of $\mu^h_T$ to its limit $\mu_T$ since the order $1$ up to a logarithmic factor was recovered.

 Moreover, the study conducted when varying the type of drift (inward pointing or outward pointing) has highlighted several features. Our results and interpretations coincide with those obtained by G\"ottlich, Lux and Neuenkirch in \cite{GotNeue} when they estimate the root mean-squared strong error. As them, we show that for inward pointing drift coefficients, the convergence order is independent of the initial value and the jump-height. This is not the case for outward pointing drift coefficients: the numerical orders are less stable in the initial value and the jump-height. A possible explanation is that the inward pointing drift coefficient engender many drift changes, while only few drift changes occur in the case of an outward pointing drift coefficient. The solution, in the latter case, can quickly drift away from the discontinuity and because of a small probability of a drift change, the empirical convergence rate might be subject to rare event effects and the linear regression estimates become questionable.


\begin{appendix}

\section{Appendix}

\begin{lem}\label{lema1}
   Let $\sigma\in\R^{d\times d}$ be a non-degenerate matrix and $\tilde g:\R^d\to\R^d$ be a measurable and locally integrable function, the spatial divergence in the sense of distributions of which is a Radon measure denoted by $\nabla\cdot\tilde g(dy)$. Then, the function $g:\R^d\to\R^d$ defined by $g(x)=\sigma^{-1}\tilde g(\sigma x)$ is locally integrable and its spatial divergence $\nabla \cdot g(dx)$ in the sense of distributions is the image of $\left|{\rm det}(\sigma^{-1})\right|\nabla \cdot \tilde g$ by $y\mapsto\sigma^{-1} y$. In particular, the total mass of $\nabla \cdot g(dx)$ is equal to $\left|{\rm det}(\sigma^{-1})\right|$ times the total mass of $\nabla \cdot \tilde g(dy)$ and when $\nabla \cdot \tilde g(dy)$ admits the density $f(y)$ with respect to the Lebesgue measure, then $\nabla\cdot g(dx)$ admits the density $f(\sigma x)$.
 \end{lem}
 \begin{proof}
   The local integrability of $g$ is easily obtained by the change of variables $y=\sigma x$. For any ${\cal C}^\infty$ function $\varphi:\R^d\to\R$  with compact support, we obtain using the same change of variables that
   \begin{align*}
     \int_{\R^d}g(x).\nabla_x\varphi(x)dx&=\int_{\R^d}\sigma^{-1}\tilde g(\sigma x).\nabla_x\varphi(x)dx=
     |{\rm det}(\sigma^{-1})|\int_{\R^d}\tilde g(y).(\sigma^{-1})^*\nabla_x\varphi(\sigma^{-1}y)dy\\&=|{\rm det}(\sigma^{-1})|\int_{\R^d}\tilde g(y).\nabla_y[\varphi(\sigma^{-1}y)]dy=-|{\rm det}(\sigma^{-1})|\int_{\R^d}\varphi(\sigma^{-1}y)\nabla\cdot\tilde g(dy),
   \end{align*}
which implies the first statement. The one concerning the total masses immediately follows and the one concerning the densities is obtained by the inverse change of variables $x=\sigma^{-1}y$.
\end{proof}

For $t>0$, let $G_t$ denote the heat kernel in $\R^{d}$: $\displaystyle G_t(x) = \frac{1}{\sqrt{(2\pi t)^d }}\exp\left(-\frac{| x |^2}{2t}\right)$.  The following lemma provides a set of estimates that are very useful:
\begin{lem}\label{EstimHeatEq}
The function $G_t(x)$ solves the heat equation:
\begin{align}\label{heateq}
    \partial_tG_t(x) - \frac{1}{2} \Delta G_t(x) = 0, \quad (t,x) \in [0,+\infty)\times \R^d.
\end{align}
We have estimates of the $L^1$-norm of the first order time derivative and the spatial derivatives of $G$ up to the third order:
\begin{align}
  &\left\Vert \partial_t G_t \right\Vert_{L^1} \le \frac d t ,\label{timeDerivG} \\
  &\left\Vert \partial_{x_i} G_t \right\Vert_{L^1} = \sqrt{\frac{2}{\pi t}}, \label{FirstDerivG} \\
  &\left\Vert \frac{\partial^2 G_{t}}{\partial x_i^2} \right\Vert_{L^1} \le \frac{2}{t}, \label{SecondDerivG1}\\
  &\left\Vert \frac{\partial^2 G_{t}}{\partial x_i x_j} \right\Vert_{L^1} = \frac{2}{\pi t}\quad \text{when } j \neq i, \label{SecondDerivG2} \\
  & \left\Vert \frac{\partial^3 G_{t}}{\partial x_j x_i^2} \right\Vert_{L^1} \le \left\{
      \begin{aligned}
        &\frac{2}{t^{3/2}}\sqrt{\frac{2}{\pi}} \quad \text{when } j \neq i, \\
        &\frac{5}{t^{3/2}}\sqrt{\frac{2}{\pi}} \quad \text{when } j = i.\\
      \end{aligned}
    \right. \label{ThirdDerivG}
\end{align}
\end{lem}

\begin{proof}
  Let us compute the estimate \eqref{FirstDerivG}. To do so, we use Fubini's theorem and obtain:
  \begin{align*}
    \left\Vert \partial_{x_i} G_t \right\Vert_{L^1} &= \int_{\R^d} \frac{\left|x_i \right|}{t}\frac{1}{\left(2\pi t\right)^{d/2}}\exp\left(-\sum\limits_{j=1}^{d}\frac{x_j^2}{2t} \right)\,dx_1\dots dx_d \\
    &= \left(\int_{\R} \frac{\left|y \right|}{t}\frac{1}{\sqrt{2\pi t}}\exp\left(-\frac{y^2}{2t}\right)\,dy\right)\times\left(\int_{\R} \frac{1}{\sqrt{2\pi t}}\exp\left(-\frac{y^2}{2t}\right)\,dy\right)^{d-1}  \\
    &= \sqrt{\frac{2}{\pi t}}.
  \end{align*}
  We can express the second and third spatial derivatives of $G$ as:
  \begin{align*}
    &\frac{\partial^2}{\partial x_i x_j} G_t(x) = \left\{
      \begin{aligned}
        &\frac{x_i x_j}{t^2}G_t(x) \quad \text{when } j \neq i, \\
        &\left(-1 + \frac{x_i^2}{t}\right)\frac{G_t(x)}{t} \quad \text{when } j = i.\\
      \end{aligned}
    \right. \; \text{and } \; \frac{\partial^3}{\partial x_j \partial x_i^2} G_t(x) = \left\{
      \begin{aligned}
        &\left(1 - \frac{x_i^2}{t}\right)\frac{x_j}{t^2}G_t(x) \quad \text{when } j \neq i, \\
        &\left(3 - \frac{x_i^2}{t}\right)\frac{x_i}{t^2}G_t(x) \quad \text{when } j = i.\\
      \end{aligned}
    \right.
  \end{align*}
  Using Fubini's theorem as for the estimate \eqref{FirstDerivG}, $\displaystyle \int_{\R} \frac{y^2}{t^2}\frac{e^{-\frac{y^2}{2t}}}{\sqrt{2 \pi t}}\,dy =  \frac{1}{t}$ and $\displaystyle \int_{\R} \frac{|y|^3}{t^3}\frac{e^{-\frac{y^2}{2t}}}{\sqrt{2 \pi t}}\,dy = \frac{2}{t^{3/2}}\sqrt{\frac{2}{\pi}}$, we obtain the estimates \eqref{SecondDerivG1}, \eqref{SecondDerivG2} and \eqref{ThirdDerivG}. As for the estimate \eqref{timeDerivG}, we deduce it from the heat equation \eqref{heateq} and the estimate \eqref{SecondDerivG1}.
\end{proof}

\begin{lem}\label{sumPi}
  We have: $$\displaystyle \forall n \ge 2, \quad \sum\limits_{k=1}^{n-1} \frac{1}{\sqrt{k}\; \sqrt{n-k}} \le \pi - \frac{2}{n}.$$ 
\end{lem}

\begin{proof}
  We define the function $f(x) = \frac{1}{\sqrt{x}\sqrt{1-x}}$ on $(0,1)$. We easily check that $\forall x \in (0,1), f(x) \ge f\left(1/2\right) = 2$  and that $\displaystyle \int_0^1f(x)\,dx = \pi$. Using the monotonicity of $f$ on $(0,1/2]$ and $[1/2,1)$, we obtain: 
  $$\displaystyle \left\{ \begin{aligned}
        &\int_{\frac{k-1}{n} }^{ \frac{k}{n}} f(x)\,dx \ge \frac{1}{n} f\left(\frac{k}{n}\right) \quad \text{when } 1 \le k \le \frac{n}{2}, \\
        &\int_{\frac{k}{n} }^{ \frac{k+1}{n}} f(x)\,dx \ge \frac{1}{n} f\left(\frac{k}{n}\right) \quad \text{when } \frac{n}{2} \le k \le (n-1).\\
      \end{aligned}
    \right.$$ 
  Therefore,\\
  $\bullet$ When $n$ is even: 
  \begin{align*}
    \displaystyle \frac{1}{n}\sum_{k=1}^{\frac{n}{2}}f\left(\frac{k}{n}\right) + \frac{1}{n}\sum_{k=\frac{n}{2}+1}^{n-1}f\left(\frac{k}{n}\right) \le \sum_{k=1}^{\frac{n}{2}}\int_{\frac{k-1}{n} }^{ \frac{k}{n}} f(x)\,dx + \sum_{k=\frac{n}{2}+1}^{n-1}\int_{\frac{k}{n} }^{ \frac{k+1}{n}} f(x)\,dx &= \int_0^{\frac 1 2}f(x)\,dx + \int_{\frac 1 2 + \frac 1 n}^1f(x)\,dx \\
    &= \pi - \int_{\frac 1 2}^{\frac 1 2 + \frac 1 n}f(x)\,dx \le \pi - \frac 2 n .
  \end{align*}
  $\bullet$ When $n$ is odd: 
  \begin{align*}
    \displaystyle \frac{1}{n}\sum_{k=1}^{\frac{n-1}{2}}f\left(\frac{k}{n}\right) + \frac{1}{n}\sum_{k=\frac{n+1}{2}}^{n-1}f\left(\frac{k}{n}\right) \le \sum_{k=1}^{\frac{n-1}{2}}\int_{\frac{k-1}{n} }^{ \frac{k}{n}} f(x)\,dx + \sum_{k=\frac{n+1}{2}}^{n-1}\int_{\frac{k}{n} }^{ \frac{k+1}{n}} f(x)\,dx &= \int_0^{\frac{1}{2} - \frac{1}{2n}}f(x)\,dx + \int_{\frac{1}{2} + \frac{1}{2n}}^1f(x)\,dx \\
    &= \pi - \int_{\frac{1}{2} - \frac{1}{2n}}^{\frac{1}{2} + \frac{1}{2n}}f(x)\,dx \le \pi - \frac 2 n .
  \end{align*}
  We can conclude.
\end{proof}

\begin{lem}\label{integ}
  For $0 < a \le x \le T$, 
  \begin{align*}
    \displaystyle \int_a^x \frac{dy}{y\sqrt{x-y}} &= \frac{1}{\sqrt x} \ln\left(\frac{\left(\sqrt x + \sqrt{x-a} \right)^2}{a} \right)\le  \frac{1}{\sqrt x} \ln\left(\frac{4x}{a} \right).
  \end{align*}
\end{lem}

\begin{proof}
  Using the change of variable $u=\sqrt{x-y}$ then a partial fraction decomposition, we obtain: \\
  $ \displaystyle \int_a^x \frac{dy}{y\sqrt{x-y}} = 2\int_0^{\sqrt{x-a}} \frac{du}{u^2 - x} = \left[\frac{1}{\sqrt x} \ln\left(\frac{\sqrt x + u}{\sqrt x - u}\right)\right]^{\sqrt{x-a}}_0$. 
\end{proof}

\begin{lem}\label{intg}
  For $0 < r \le s \le T$, 
  \begin{align*}
    \displaystyle \frac{1}{2}\int_0^r \frac{\ln(s-u) - \ln(r-u)}{\sqrt u}\,du = \frac{s-r}{\sqrt s + \sqrt r}\ln\left(\frac{\left(\sqrt s + \sqrt r \right)^2}{s-r} \right) + 2 \sqrt r \ln\left(1 + \frac{\sqrt s - \sqrt r}{2 \sqrt r} \right).
  \end{align*}
\end{lem}

\begin{proof}
  We start by applying the change the variable $\theta = \sqrt u$ and obtain 
  \begin{align*}
    \displaystyle \frac{1}{2}\int_0^r \frac{\ln(s-u) - \ln(r-u)}{\sqrt u}\,du &= \int_0^{\sqrt r} \bigg(\ln\left(s-\theta^2 \right) - \ln\left(r -\theta^2\right)\bigg)\,d\theta \\
    &= \int_0^{\sqrt r} \bigg(\ln\left(\sqrt s-\theta\right)+\ln\left(\sqrt s+\theta\right)-\ln\left(\sqrt r-\theta\right)-\ln\left(\sqrt r+\theta\right) \bigg)\,d\theta.
  \end{align*}
  A simple integration of $\ln(x)$ permits us to conclude.
\end{proof}

The next lemma is a discrete version of Gronwall's lemma and was proved by Holte \cite{Holt}.
\begin{lem}\label{Gronw}
If $\left(y_n \right)_{n \in \N}$, $\left(f_n \right)_{n \in \N}$ and $\left(g_n \right)_{n \in \N}$ are non-negative sequences and $$ y_n \le f_n + \sum_{i= 0}^{n-1}g_i y_i \quad \text{ for } n \in \N $$ then $$ y_n \le f_n + \sum_{i= 0}^{n-1}f_i g_i \exp \left(\sum_{j=i+1}^{n-1}g_j \right) \quad \text{ for } n \in \N. $$
\end{lem}

\end{appendix}



\begin{thebibliography}{plain}
  \bibitem{AmbFusPal} L. Ambrosio, N. Fusco and D. Pallara, \textit{Functions of Bounded Variation and Free Discontinuity Problems}. Oxford Mathematical Monographs, pp.119--120 ,2000.

  \bibitem{Bally} V. Bally and C. Rey, \textit{Approximation of Markov semigroups in total variation distance}. Electronic Journal of Probability, Vol.21(12), 2016.

  \bibitem{BHY} J. Bao, X. Huang and C. Yuan, \textit{Convergence Rate of Euler-Maruyama Scheme for SDEs with H\"older-Dini Continuous Drifts}. Journal of Theoretical Probability, Vol.32, pp.848--871, 2019.

  \bibitem{Brezis} H. Brezis, \textit{Analyse Fonctionelle: Th\'eorie et applications}. Collection Math\'ematiques appliqu\'ees pour la ma\^itrise - Masson, $2^{e}$ tirage, 1987. 

  \bibitem{DarGer} K. Dareiotis and M. Gerencs\'er, \textit{On the regularisation of the noise for the Euler-Maruyama scheme with irregular drift}. ArXiv preprint arXiv:1812.04583, 2020.

  \bibitem{Daun} T. Daun, \textit{On the randomized solution of initial value problems}. Journal of Complexity, Vol.27, pp.300--311, 2011.

  \bibitem{EtoMArt} P. Etore and M. Martinez, \textit{Exact simulation for solutions of one-dimensional Stochastic Differen- tial Equations with discontinuous drift}. ESAIM: Probability and Statistics, EDP Sciences, Vol.18, pp.686--702, 2014.

  \bibitem{NFrik} N. Frikha, \textit{On the weak approximation of a skew diffusion by an Euler-type scheme}. Bernoulli, Vol.24(3), pp.1653--1691, 2018.

  \bibitem{GotNeue} S. G\"ottlich, K. Lux and A. Neuenkirch, \textit{The Euler scheme for stochastic differential equations with discontinuous drift coefficient: A numerical study of the convergence rate}. Advances in Difference Equation, Vol.429, 2019.

  \bibitem{GyoKryl} I. Gy\"ongy and N. Krylov, \textit{Existence of strong solutions for Ito's stochastic equations via approximations}. Probability Theory and Related Fields, Vol.105, pp.143--158, 1996.

  \bibitem{Gyo} I. Gy\"ongy, \textit{A note on Euler's approximations}. Potential Analysis, Vol.8, pp.205--216 , 1998.

  \bibitem{GR} I. Gy\"ongy and M. Rasonyi, \textit{A note on Euler approximations for SDEs with H\"older continuous diffusion coefficients}. Stochastic Processes and their Applications, Vol.121, pp.2189--2200, 2011.

  \bibitem{HalKloe} N. Halidias and P.E. Kloeden, \textit{A note on the Euler-Maruyama scheme for stochastic differential equations with a discontinuous monotone drift coefficient}. BIT Numerical Mathematics, Vol.48, pp.51--59 ,2008.

  \bibitem{HeinMil} S. Heinrich and B. Milla, \textit{The randomized complexity of initial value problems}. Journal of Complexity, Vol.24, pp.77--88, 2008.

  \bibitem{Holt} J. M. Holte, \textit{Discrete Gronwall Lemma And Applications}. Mathematical Association of America North Central Section Meeting at UND, 24 October 2009.

  \bibitem{JenNeue} A. Jentzen and A. Neuenkirch. \textit{A random euler scheme for carath\'eodory differential equations}. Journal of computational and applied mathematics, Vol.224(1), pp.346--359, 2009.

  \bibitem{KaratShre} I. Karatzas and S. Shreve, \textit{Brownian Motion and Stochastic Calculus}. Graduate Texts in Mathematics, Vol.113(2), p.441, 1998.

  \bibitem{KohY} A. Kohatsu-Higa, A. Lejay and K. Yasuda, \textit{On Weak Approximation of Stochastic Differential Equations with Discontinuous Drift Coefficient}. Mathematical Economics, Kyoto, Japan, pp.94--106, 2012.

  \bibitem{Koh} A. Kohatsu-Higa, A. Lejay and K. Yasuda, \textit{Weak rate of convergence of the Euler-Maruyama scheme for stochastic differential equations with non-regular drift}. Journal of Computational and Applied Mathematics, Vol.326, pp.138--158, 2017.

  \bibitem{MenoKona} V. Konakov and S. Menozzi, \textit{Weak error for the Euler scheme approximation of diffusions with non-smooth coefficients}. Electronic Journal of Probability, Vol.22(46), pp.1--47, 2017.

  \bibitem{CruW} R. Kruse and Y. Wu, \textit{A randomized milstein method for stochastic differential equations with non-differentiable drift coefficients}. Discrete and Continuous Dynamical Systems, Vol.24(8), pp.3475--3502, 2019.

   \bibitem{LeoSzo1} G. Leobacher and M. Sz\"olgyenyi, \textit{A numerical method for SDEs with discontinuous drift}. BIT Numerical Mathematics, Vol.56, pp.151--162, 2016.

  \bibitem{LeoSzo2} G. Leobacher and M. Sz\"olgyenyi, \textit{A strong order $1/2$ method for multidimensional SDEs with discontinuous drift}. Annals of Applied Probability, Vol.27(4), pp.2382--2418, 2017.

  \bibitem{LeoSzo3} G. Leobacher and M. Sz\"olgyenyi, \textit{Convergence of the Euler-Maruyama method for multidimensional SDEs with discontinuous drift and degenerate diffusion coefficient}. Numerische Mathematik, Vol.138(1), pp.219--239, 2017.

  \bibitem{MikP} R. Mikulevicius and E. Platen, \textit{Rate of convergence of the Euler approximation for diffusion processes}. Mathematische Nachrichten, Vol.151, pp.233--239, 1991.

  \bibitem{Mik} R. Mikulevicius, \textit{On the rate of convergence of simple and jump-adapted weak Euler schemez for L\'evy driven SDEs}. Stochastic Processes and their Applications, Vol.122(7), pp.2730--2757, 2012.

  \bibitem{MikZ} R. Mikulevicius and C. Zhang, \textit{Weak Euler approximation for It\^o's diffusion and jump processes}. Stochastic Analysis and Applications, Vol.33(3), pp.549--571, 2015.

  \bibitem{MGYaro1} T. M\"uller-Gronbach and L.Yaroslavtseva, \textit{On the performance of the Euler-Maruyama scheme for SDEs with discontinuous drift coefficient}. Annales de l'Institut Henri Poincar\'e - Probabilit\'es et Statistiques, Vol.56(2), pp.1162--1178, 2020.

  \bibitem{MGYaro2} T. M\"uller-Gronbach and L.Yaroslavtseva, \textit{A strong order $3/4$ method for SDEs with discontinuous drift coefficient}. ArXiv preprint arXiv:1904.09178, 2019.

  \bibitem{NeunSzoSpr} A. Neuenkirch, M. Sz\"olgyenyi and L. Szpruch, \textit{An Adaptive Euler-Maruyama Scheme for Stochastic Differential Equations with Discontinuous Drift and its Convergence Analysis}.
  SIAM Journal on Numerical Analysis, Vol.57(1), pp.378--403, 2019.

  \bibitem{NeuSzo} A. Neuenkirch and M. Sz\"olgyenyi, \textit{The Euler-Maruyama Scheme for SDEs with Irregular Drift: Convergence Rates via Reduction to a Quadrature Problem}. ArXiv preprint arXiv:1904.07784, 2020.

  \bibitem{NgTag1} H. Ngo and D. Taguchi, \textit{Strong rate of convergence for the Euler-Maruyama approximation of stochastic differential equations with irregular coefficients}. Mathematics of Computation, Vol.85, pp.1793--1819, 2016.

  \bibitem{NgTag2} H. Ngo and D. Taguchi, \textit{On the Euler-Maruyama approximation for one-dimensional stochastic differential equations with irregular coefficients}. IMA Journal of Numerical Analysis, Vol.37(4), pp.1864--1883, 2017.

  \bibitem{Silv} B. W. Silverman, \textit{Density estimation for statistics and data analysis}. London: Chapman \& Hall/CRC. p.45, 1998.

  \bibitem{Stengle1} G. Stengle, \textit{Numerical Methods for Systems with Measurable Coefficients}. Appl. Math. Lett., Vol.3(4), pp.25--29, 1990.

  \bibitem{Stengle2} G. Stengle, \textit{Error analysis of a randomized numerical method}. Numerische Mathematik, Vol.70, pp.119--128, 1995.

  \bibitem{SuYuZ} Y. Suo, C. Yuan and S. Zhang, \textit{Weak convergence of Euler scheme for SDEs with singular drift}. ArXiv preprint arXiv:2005.04631, 2020.
  
  \bibitem{Yan} L. Yan, \textit{The Euler scheme with irregular coefficients}. The Annals of Probability, Vol.30(3), pp.1172--1194, 2002.

  \bibitem{Veret} A. J. Veretennikov, \textit{On strong solutions and explicit formulas for solutions of stochastic integral equations}. Mathematics of the USSR-Sbornik, Vol.39(3), pp.387--403, 1981.
 
\end{thebibliography}
\end{document}